\documentclass[reqno,oneside,12pt]{amsart}
\usepackage{mathtools}
\usepackage[utf8]{inputenc}
\usepackage{amssymb,amsthm,amsmath, amscd, xcolor}
\usepackage{enumitem}
\usepackage{makecell}
\usepackage[all]{xy}
\usepackage{diagbox}
\usepackage{tikz-cd}
\usepackage{tikz}
\usetikzlibrary{decorations.pathmorphing}
\usepackage{multirow}
\usetikzlibrary{positioning}
\usepackage{todonotes}

\numberwithin{equation}{section}
\usepackage{hyperref}
\hypersetup{
linkcolor = {black},
urlcolor = {blue},
citecolor = {black}
}

\theoremstyle{plain}

\newtheorem{theorem}{Theorem}[section]
\newtheorem{lemma}[theorem]{Lemma}
\newtheorem{prop}[theorem]{Proposition}
\newtheorem{cor}[theorem]{Corollary}

\newtheorem*{claim}{Claim}

\theoremstyle{definition}

\newtheorem{definition}[theorem]{Definition}
\newtheorem{example}[theorem]{Example}

\newtheorem*{assumption}{Assumption}

\theoremstyle{remark}

\newtheorem*{remark}{Remark}

\newcommand{\HW}{{\mathcal D}\stackrel{\textup{HW}}\leadsto\tilde{\mathcal D}}

\begin{document}

% \begin{textblock*}{100mm}(.75\textwidth,-2cm)
% Draft: \today
% \end{textblock*}

\title{Existence and orthogonality of stable envelopes for bow varieties}
\author{Catharina Stroppel}
\address{Department of Mathematics, University of Bonn, 53115 Bonn, Germany}
\email{stroppel@math.uni-bonn.de}
\author{Till Wehrhan}
\address{Max-Planck Institute for Mathematics, Vivatgasse 7, 53111 Bonn, Germany}
\email{wehrhan@mpim-bonn.mpg.de}

\begin{abstract} 
Stable envelopes, introduced by Maulik and Okounkov, provide a family of bases for the equivariant cohomology of symplectic resolutions. The theory of stable envelopes provides a fascinating interplay between geometry, combinatorics and integrable systems. In this expository article, we give a self-contained introduction to cohomological stable envelopes of type $A$ bow varieties.  Our main focus is on the existence and the orthogonality properties of stable envelopes for bow varieties. The restriction to this specific class of varieties allows us to illustrate the theory combinatorially and to provide simplified proofs, both laying a basis for explicit calculations.
\end{abstract}

\maketitle

\tableofcontents

\section{Introduction}
\label{section:introduction}

The intersection theory of Schubert varieties on Grassmannians is a key ingredient of Schubert calculus and provides a fruitful interplay between algebraic geometry, combinatorics and representation theory. Schubert varieties can be defined as closures of attracting cells of torus fixed points with respect to a choice of torus cocharacter. The corresponding classes in singular (or more generally torus equivariant) cohomology are known as (equivariant) Schubert classes and form a basis of the respective cohomology ring. Equivariant Schubert classes are uniquely determined by their restrictions to torus fixed point and there is a convenient system of axioms, just involving a normalization, a support and a degree condition, that uniquely determines them, see e.g.~\cite{knutson2003puzzles},~\cite{gorbounov2020yang}.

An interesting aspect of equivariant Schubert calculus is that the base change matrices of equivariant Schubert bases with respect to different choices of cocharacters give solutions to Yang--Baxter equations. The Bethe algebras of the corresponding quantum integrable systems are known to be isomorphic to the equivariant quantum cohomology of Grassmannians providing a fascinating connection between the enumerative geometry of Grassmannians and the combinatorics of lattice models, see e.g.~\cite{gorbounov2014equivariant},~\cite{gorbounov2017quantum},~\cite{gorbounov2020yang}.

Passing to the world of symplectic varieties with torus action, Maulik and Okounkov defined in~\cite{maulik2019quantum} (cohomological) stable envelopes as equivariant cohomology classes satisfying certain axioms that are similar to the axioms for equivariant Schubert classes. They exist for a large family of symplectic varieties including Nakajima quiver varieties and form a basis of the respective localized equivariant cohomology rings. In the special case of cotangent bundles of Grassmannians, stable envelopes can be seen as one parameter deformations of the Schubert classes, see \cite{shenfeld2013abelianization} for an explicit treatment. 

Stable envelopes enjoy many remarkable properties, as described in~\cite{maulik2019quantum}. Most importantly, they produce, just like Schubert classes, solutions of (more general) Yang--Baxter equations. These solutions lead to quantum integrable systems in which the stable envelopes are viewed as spin basis. These systems turn out to be well suited for the study of operators of quantum multiplication and their related quantum differential equations, see \cite{maulik2019quantum}, \cite{tarasov2014hypergeometric}, \cite{tarasov2019q}, \cite{tarasov2022monodromy} and the references therein. 

The theory of stable envelopes has been generalized to equivariant K-theory in~\cite{okounkov2017k} and then further to elliptic cohomology in~\cite{aganagic2021elliptic} and~\cite{okounkov2021inductive}, where they govern the solution of quantum difference equations associated to the moduli of quasi maps, see ~\cite{aganagic2017quasimap}, \cite{felder2018elliptic}, \cite{rimanyi2019elliptic}, \cite{okounkov2020nonabelian}, \cite{kononov2020pursuing}, \cite{kononov2022pursuing}, \cite{okounkov2022quantum} for a partial list of further references. 

In this article, we give a self-contained introduction to the theory of cohomological stable envelopes in the framework of (type $A$) bow varieties following~\cite{maulik2019quantum}. Bow varieties form a rich family of symplectic varieties with a torus action. As special cases, they include type $A$ Nakajima quiver varieties and hence in particular cotangent bundles of partial flag varieties. Motivated by theoretical physics, they were introduced by Cherkis in~\cite{cherkis2009moduli},~\cite{cherkis2010instantons} and~\cite{cherkis2011instantons} as ADHM type descriptions of moduli spaces of instantons. In~\cite{nakajima2017cherkis}, Nakajima and Takayama gave an alternative construction, similar to the construction of Nakajima quiver varieties, starting from representations of quivers and using hamiltonian reduction. This work marked the starting point of an algebro-geometric study of these varieties, complementing Cherkis' differential geometric description (of their smooth parts) as hyper-K\"ahler manifold in terms of Nahm's equation. 

Bow varieties play a particular role in the mathematical manifestation, due to Cherkis and Nakajima--Takayama,  of $3d$ mirror symmetry which is a duality from theoretical physics relating pairs of $d=3,N=4$ supersymmetric quantum field theories. We refer to~\cite{nakajima2017cherkis} for a beautiful overview, precise statements and references to the physics literature. We will see below that a bow diagram comes from a diagrammatical object called brane diagram which encodes on the physics side a brane configuration in a  type IIB string theory in the sense of \cite{HW97} which for us however will appear as a purely combinatorial object. The creation of a new 3-brane has a combinatorial incarnation in terms of a move on the brane diagram. The resulting Hanany--Witten transition creates from a given bow variety a new one. As we will summarize below, this can be made mathematically rigorous on a combinatorial, a representation theoretic and a geometric level. For the connection to physics, see \cite{nakajima2017cherkis}.

The family of bow varieties is well-suited for explicit computations in equivariant cohomology. They in particular have only finitely many torus fixed points with a concrete combinatorial description and classification due to Nakajima, \cite{nakajima2021geometric}. The description of the fixed point combinatorics from \cite{rimanyi2020bow} is in terms of skein diagram like figures which are called tie diagrams. The combinatorics allows to explicitly compute tangent weights at fixed points and equivariant multiplicities of characteristic classes which are important building blocks of the equivariant cohomology algebra.

In our exposition, we focus on the existence and the orthogonality of stable envelope bases in the cohomology of bow varieties. 
We explain, following~\cite[Chapter~3]{maulik2019quantum}, how these bases can be constructed recursively as $\mathbb Z$-linear combination of the equivariant cohomology cycles corresponding to closures of attracting cells. The crucial argument in this recursive argument is that the equivariant multiplicity of lagrangian hyperplanes at a torus fixed point is equal to an integer multiple of the equivariant multiplicity of a lagrangian hyperplane of the tangent space. This result is a consequence of the deformation to the normal cone construction of  Fulton~\cite{fulton1984introduction}. 

A powerful fact in Schubert calculus is that the equivariant Schubert classes corresponding to opposite choices of cocharacters,
or equivalently to opposite choices of Borels, are orthogonal. We will see in Section~\ref{section:orthogonality} that stable envelopes satisfy an analogous orthogonality relation. This result is a consequence of the defining axioms of stable envelopes and the fact that the intersection of the closures of attracting cells corresponding to opposite cocharacters is always proper despite the fact that bow varieties are in general not proper. 

The orthogonality of stable envelopes is useful for computing multiplication operators of equivariant cohomology classes with respect to the stable envelope basis. More concretely, one would like to describe and characterize the operation of multiplication with first Chern classes of tautological bundles. In Schubert calculus, this is given by the famous Chevalley--Monk formulas. The orthogonality properties established below for bow varieties are important ingredients for a description of multiplication operators in~\cite{wehrhan2023chevalley} which can be interpreted as Chevalley--Monk formulas for bow varieties. One might see these results as a first step towards a generalized Schubert calculus for bow varieties. 
\subsubsection*{Acknowledgements} We thank Hiraku Nakajima for sharing his insights and for very useful comments on a preliminary version of the paper, and Rich\'ard Rim\'anyi for helpful discussions. The authors are supported by the Gottfried Wilhelm Leibniz Prize of the German Research Foundation and the Max-Planck Institute for Mathematics (IMPRS Moduli Spaces) respectively. This work will be part of the PhD thesis of the second author.
%\medskip
%\noindent
\section{Preliminaries}
\label{section:preliminaries}

\subsection{Notation and conventions}

In this article, all varieties and vector spaces are over $\mathbb C$ and cohomology coefficients in $\mathbb{Q}$. Given a smooth variety $X$, we denote its tangent bundle by $TX$. The tangent space at a point $x\in X$ is denoted by $T_xX$.

If $Y$ is a variety with an algebraic action of a torus $T=(\mathbb C^\ast)^r$ then we denote its $T$-equivariant cohomology by $H_T^\ast(Y)$ and its $T$-equivariant Borel--Moore homology by $\overline H_\ast^T(Y)$ both with coefficients in $\mathbb Q$. If $Y$ is smooth we identify $H_T^\ast(Y)$ and $\overline H_\ast^T(Y)$ by Poincar\'e duality. We denote by $\operatorname{pt}$ the space consisting just of one single point and use the identification $H_T^\ast(\operatorname{pt})\cong H^\ast_T((\mathbb P^\infty)^r)\cong \mathbb Q[t_1,\ldots, t_r]$ where $t_i$ is the first Chern class of the tautological bundle on the $i$-th factor of $(\mathbb P^\infty)^r$.

Given a $T$-equivariant morphism $f:Y\rightarrow Y'$ between $T$-varieties, we denote the respective push-forward and pull-back morphisms (whenever they are defined) on equivariant cohomology or equivariant Borel--Moore homology by $f_\ast,f^\ast$. If $Y''\subset Y$ is a $T$-invariant closed subvariety, we denote by $[Y'']^T$ its equivariant Borel--Moore cohomology class in $\overline H_\ast^T(Y)$. If $V$ is a $T$-equivariant vector bundle over $Y$, we denote by $e_T(V)$ its Euler class in $H_T^\ast(Y)$.

Given a finite dimensional $\mathbb C^\ast$-representation $W$, we denote by $W^+$ (resp. $W^-$) the subspace generated by all strictly positive (resp. negative) weight spaces. The subspace of $\mathbb C^\ast$-invariant vectors is denoted by $W^0$ and we set $W^{\geq 0}:=W^0\oplus W^+$ and $W^{\le 0}:=W^0\oplus W^-$. 

%Concretely, Nakajima and Takayama defined bow varieties depending on stability parameters
%$\nu_{\sigma}^{\mathbb C}$ and $\nu_{\sigma}^{\mathbb R}$. Throughout this article, we only consider bow varieties corresponding to the specializations $\nu_{\sigma}^{\mathbb C}=0$ and $\nu_{\sigma}^{\mathbb R}=-1$. Among others, one nice feature of this family of bow varieties admits is the fact that they can be realized as GIT quotients.

\subsection{Geometric invariant theory quotients}

In this subsection, we give a brief reminder on geometric invariant theory (short GIT). For more details on this subject, see~\cite{mumford1994geometric} as well as the expository works~\cite{mukai2003introduction} and~\cite{newstead2009geometric}.

We start with recalling the formal definition of GIT quotients. For this, let $X$ be an affine variety with coordinate ring $A$ and let $G$ be a reductive group acting on $X$ from the left with action map $G\times X\rightarrow X, (g,x)\mapsto g.x$. We denote by $A^G$ the algebra of $G$-invariants of $A$. Let $X/\!\!/G=\mathrm{Spec}(A^G)$ be the categorical quotient. Due to a theorem of Hilbert (see e.g. \cite[Theorem~4.51]{mukai2003introduction}), $X/\!\!/G$ is an affine variety. 

The definition of GIT quotients involves a choice of rational character of $G$. Given a rational character $\chi:G\rightarrow \mathbb C^\ast$, the associated \textit{algebra of semi-invariants of $A$} is defined as
\[
A_{\chi}:=\bigoplus_{n\geq 0} A_{\chi^n},\quad\textup{where }A_{\chi^n}=\{f\in A\mid f(g.x)=\chi(g)^nf(x)\textup{ for all }x\in X,g\in G\}.
\]
We endow $A_{\chi}$ with the grading where the $n$-th homogeneous piece of $A$ is given by $ A_{\chi^n}$. 
\begin{definition}
The \textit{GIT quotient} $X/\!\!/_{\!\chi} G$ is defined as the quasi-projective scheme
\[
X/\!\!/_{\!\chi} G := \mathrm{Proj}(A_\chi).
\]
\end{definition}
The definition of $X/\!\!/_{\!\chi} G$ as projective spectrum then directly gives that the inclusion $A^G\hookrightarrow A_\chi$ induces a projective morphism $X/\!\!/_{\!\chi} G\rightarrow X/\!\!/ G$.
%The following fact is a straight-forward consequence of the construction of $X/\!\!/_{\!\chi} G$:
%\begin{prop} The inclusion $A^G\hookrightarrow A_\chi$ induces a projective morphism $X/\!\!/_{\!\chi} G\rightarrow X/\!\!/ G$.
%\end{prop}
%By construction, $X/\!/_\chi G$ is a quasi-projective scheme and the inclusion $A^G\hookrightarrow A_\chi$ gives a projective morphism $X/\!/_\chi G\rightarrow X/\!/ G$.

Next, we recall the notion of $\chi$-(semi)stable points of $X$ and characterizations thereof:
\begin{definition} Let $x\in X$.
\begin{enumerate}[label=(\roman*)]
\item The point $x\in X$ is called \textit{$\chi$-semistable} if there exists $n\geq1$ and $f\in A_{\chi^n}$ such that $x\in D(f)$, where $D(f)=\{x\in X\mid f(x)\ne0\}$.
\item The point $x\in X$ is called \textit{$\chi$-stable} if there exists $n\geq1$ and $f\in A_{\chi^n}$ such that
\begin{enumerate}[label=(\alph*)]
\item $x\in D(f)$,
\item the action $G\times D(f)\rightarrow D(f)$ is a closed morphism,
\item the isotropy group $G_x$ is finite.
\end{enumerate}
\end{enumerate}
We write $X^{\mathrm{ss}}$ and $X^{\mathrm s}$ for the subset of $\chi$-semistable respectively $\chi$-stable points of $X$.
%\begin{enumerate}[resume, label=(\roman*)]
%\item Two points $x,y\in X^{\mathrm{ss}}$ are called $S$-equivailent if and only if the orbit closures $\overline{G.x}$ and $\overline{G.y}$ intersect in $X^{\mathrm{ss}}$.
%\end{enumerate}
\end{definition}
Both $X^{\mathrm{ss}}$ and $X^{\mathrm s}$ are open $G$-invariant subvarieties of $X$.
There are several equivalent definitions of $\chi$-semistability respectively $\chi$-stability. 
%The following characterization was used by Nakajima and Takayama in~\cite{nakajima2017cherkis}.
The following characterization was introduced by King, {\cite[Lemma~2.2]{king1994moduli}}:
\begin{prop}[King's stability] Equip the variety $X\times\mathbb C$ with the $G$-action $g.(x,z)=(g.x,\chi^{-1}(g).z)$, for $g\in G,x\in X$ and $z\in\mathbb C$.
\begin{enumerate}[label=(\roman*)]
\item A point $x\in X$ is $\chi$-semistable if and only if the orbit closure $\overline{G.(x,z)}$ does not intersect $X\times\{0\}$ for all $z\ne0$.
\item A point $x\in X$ is $\chi$-stable if and only if the orbit $G.(x,z)$ is closed and the isotropy group $G_{(x,z)}$ is finite for all $z\ne0$.
\end{enumerate}
\end{prop}
%\begin{proof} See e.g. \cite[Proposition 6.1]{hoskins2015moduli}. The cited reference works in the setting of projective varieties, but the proof directly transfers to the affine setting.
%\textcolor{orange}{These are unpublished lecture notes. However, it is the only reference with a proof that I could find. A version of the above proposition is also contained in King's paper on quiver moduli but without a proof. Maybe it is better to cite King?}
%\end{proof}
The standard inclusion $A_\chi\hookrightarrow A[z]$ gives a morphism of schemes $\Phi:X^{\mathrm{ss}}\rightarrow X/\!\!/_{\!\chi} G$. This morphism characterizes the points of $X/\!\!/_{\!\chi} G$ as follows, see {\cite[Theorem~1.10]{mumford1994geometric}}:

\begin{theorem}[GIT-Theorem]\label{thm:PointsOfGITQuotients} The following holds:
\begin{enumerate}[label=(\roman*)]
\item The morphism $\Phi:X^{\mathrm{ss}}\rightarrow X/\!\!/_{\!\chi} G$ is a categorical quotient and surjective.
\item For $x,y\in X^{\mathrm{ss}}$, we have $\Phi(x)=\Phi(y)$ if and only if the orbit closures $\overline{G.x}$ and $\overline{G.y}$ in $X$ intersect non-trivially in $X^{\mathrm{ss}}$, i.e. $\overline{G.x}\cap \overline{G.y}\cap X^{\mathrm{ss}}\ne\emptyset$. 
\item Let $U=\Phi(X^{\mathrm s})$. Then $\Phi_{\mid X^{\mathrm s}}:X^{\mathrm s}\rightarrow U$ is a geometric quotient. In particular, the morphism $\Phi$ induces a bijection
\[
\begin{tikzcd}
\{\textit{$G$-orbits in $X^{\mathrm{s}}$}\} \arrow[r, leftrightarrow,"1:1"] & \{\textit{Points of $U$}\}.
\end{tikzcd}
\]
\end{enumerate}
\end{theorem}

Mumford introduced a numerical criterion for $\chi$-(semi)stability~\cite[Chapter~2]{mumford1994geometric} that proved to be very valuable in practical computations. Our formulation of this criterion is following~\cite[Proposition~2.6]{king1994moduli}. Recall that a \textit{one parameter subgroup} of $G$ is an algebraic cocharacter $\lambda:\mathbb C^\ast \rightarrow G$. Moreover, let $\langle \lambda,\chi\rangle$ be the unique integer such that $\chi(\lambda(t))=t^{\langle \lambda,\chi\rangle}$ for all $t\in\mathbb C^\ast$. For a given point $x\in X$, we say that \emph{the limit $\lim_{t\to 0}\lambda(t).x$ exists} in $X$ if and only if the morphism $\mathbb C^\ast\rightarrow X,t\mapsto t.x$ extends to a morphism $\mathbb C\rightarrow X$. Since $X$ is quasi-projective, the limit $\lim_{t\to 0}\lambda(t).x$ exists if and only if the limit exists in the analytic topology of $X$.

\begin{theorem}[Mumford's numerical criterion]  A point $x\in X$ is $\chi$-semistable (resp. $\chi$-stable) if and only if for all non-trivial one parameter subgroups $\lambda$ such that $\lim_{t\to 0}\lambda(t).x$ exists in $X$, we have $\langle \lambda,\chi\rangle\geq 0$ (resp. $\langle \lambda,\chi\rangle>0$).
\end{theorem}

\section{Bow varieties}
\label{section:bowVarieties}

In this section, we recall the definition of bow varieties from~\cite{nakajima2017cherkis} and discuss some of their geometric properties. Herby, we use the language of brane diagrams and their combinatorics from~\cite{rimanyi2020bow}.

Bow varieties are obtained from a class of varieties that we call affine brane varieties\footnote{This should not be confused with affine bow varieties in ~\cite{nakajima2017cherkis}. In our terminology, {\it affine} should remind the reader on affine varieties instead of affine Dynkin diagrams.

} via hamiltonian reduction. These affine brane varieties are constructed using a family of smaller varieties that are called triangle parts.

We start this section by reviewing the definition of triangle parts and then explain  the construction of affine brane varieties in Subsection~\ref{subsection:AffineBraneVariety}. We continue with the definition of bow varieties in Subsection~\ref{subsection:bowVarieties} 
and then discuss their geometric properties.

%In this article, we always work with the conventions from~\cite{rimanyi2020bow} and also use the language of brane combinatorics that was introduced by Rimanyi and Shou. 

%\textcolor{orange}{In general, bow varieties are defined as the stable locus of GIT quotients. In our special case, the stability and semistability conditions coincide, see Proposition~\ref{prop:nakajimaSemistabilityCriterion} below. The crucial point is that just as in the setting of Nakajima quiver varieties, we have a numerical criterion for stability and semistability.
%However, in my opinion, Rimanyi and Shou (implicitly) claim in section 6.1 of their paper that it is unknown if the bow varieties of our interest are GIT quotients or not.}

\subsection{Triangle part}\label{subsection:trianglePart}

We fix finite dimensional $\mathbb C$-vector spaces $V_1,V_2$ and let
\[
\mathbb M:= \operatorname{Hom}(V_2,V_1)\oplus \operatorname{End}(V_1)\oplus \operatorname{End}(V_2)\oplus \operatorname{Hom}(\mathbb C,V_1) \oplus \operatorname{Hom}(V_2,\mathbb C) .
\]
The elements of $\mathbb M$ are tuples $(A,B_1,B_2,a,b)$ of linear maps as illustrated in the diagram: 
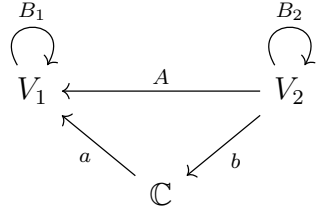
\[
\begin{tikzcd}[scale=0.02]
V_1\arrow[out=120,in=60,loop,looseness=5,"B_1"]&&V_2\arrow[out=120,in=60,loop,looseness=5,"B_2"]\arrow[ll, "A", swap]\arrow[dl,"b"]\\
&\mathbb C\arrow[ul,"a"]
\end{tikzcd}
\]
The group $\mathrm{GL}(V_1)\times\mathrm{GL}(V_2)$ acts on $\mathbb M$ via base change:
\[
(g_1,g_2).(A,B_1,B_2,a,b)=(g_1Ag_2^{-1},g_1B_1g_1^{-1},g_2B_2g_2^{-1},bg_2^{-1},g_1a).
\]
\begin{definition}
The \textit{triangle part} $\mathrm{tri}(V_1,V_2)$ is defined as the $\theta$-semistable locus of $\mu^{-1}(0)$, that is $\{x\in\mu^{-1}(0)\mid x \text{ is $\theta$-semistable}\}$,    where
\begin{equation}\label{eq:nonMomentMap}
\mu:\mathbb M\rightarrow \operatorname{Hom}(V_1,V_2),\quad (A,B_1,B_2,a,b)\mapsto B_1A-AB_2+ab
\end{equation}
and 
\[
\theta:\mathrm{GL}(V_1)\times\mathrm{GL}(V_2)\rightarrow \mathbb C^\ast,\quad (g_1,g_2)\mapsto\frac{\operatorname{det}(g_2)}{\operatorname{det}(g_1)}.
\]
\end{definition}
%Define the $\mathrm{GL}(V_1)\times\mathrm{GL}(V_2)$-equivariant linear operator
%\begin{equation}\label{eq:nonMomentMap}
%\mu:\mathbb M\rightarrow \operatorname{Hom}_{\mathbb C}(V_1,V_2),\quad (A,B_1,B_2,a,b)\mapsto B_1A-AB_2+ab.
%\end{equation}
%Then, $\mu^{-1}(0)$ is a closed $\mathrm{GL}(V_1)\times\mathrm{GL}(V_2)$-invariant subvariety of $\mathbb M$. Next, we fix the character
%\[
%\theta:\mathrm{GL}(V_1)\times\mathrm{GL}(V_2)\rightarrow \mathbb C^\ast,\quad (g_1,g_2)\mapsto\frac{\operatorname{det}(g_2)}{\operatorname{det}(g_1)}.
%\]
 We like to characterize $\theta$-semistable points $x=(A,B_1,B_2,a,b)\in\mu^{-1}(0)$. For this, we introduce the following \emph{subspace conditions}, see~\cite[Section~2]{nakajima2017cherkis}:
\begin{enumerate}[label=(S\arabic*)]
\item \label{item:S1} If $S\subset V_2$ is a subspace with $B_2(S)\subset S,A(S)=0,b(S)=0$ then $S=0$.
\item \label{item:S2} If $T\subset V_1$ is a subspace with $B_1(T)\subset T,\mathrm{Im}(A)+\mathrm{Im}(a)\subset T$ then $T=V_1$.
\end{enumerate}
Property~\ref{item:S1} is a useful criterion to check vanishing of subspaces of $V_2$ whereas~\ref{item:S1} is useful for proving that subspaces of $V_1$ actually coincide with $V_1$. We will use this criterion frequently in the proof of the Cocharacter Theorem in Section~\ref{section:torusFixedPoints}.

We further introduce the following \emph{triangle part conditions}:
\begin{enumerate}[label=(T\arabic*)]
\item \label{item:ssS1} If $S_1\subset V_1,S_2\subset V_2$ are subspaces with $B_1(S_1)\subset S_1$, $B_2(S_2)\subset S_2$, $A(S_2)\subset S_1$ and $b(S_2)=0$ then $\operatorname{dim}(S_1)\geq \operatorname{dim}(S_2)$.
\item \label{item:ssS2} If $T_1\subset V_1,T_2\subset V_2$ are subspaces with $B_1(T_1)\subset T_1$, $B_2(T_2)\subset T_2$, $A(T_2)\subset T_1$ and $\mathrm{Im}(a)=T_1$ then $\operatorname{codim}(T_1)\le \operatorname{codim}(T_2)$.
\end{enumerate}
Clearly,~\ref{item:ssS1} implies~\ref{item:S1} by setting $S_1=0,S_2=S$, whereas~\ref{item:ssS2} implies~\ref{item:S2} by setting $T_1=T,T_2=V_2$. The next proposition gives that these conditions are actually equivalent to $\theta$-semistability:

\begin{prop}\label{prop:S1S2SemistabilityConditions}
Let $x=(A,B_1,B_2,a,b)\in\mu^{-1}(0)$. Then, the following are equivalent:
\begin{enumerate}[label=(\roman*)]
\item\label{item:propSemistabilityConditionsI} The point $x$ is $\theta$-semistable.
\item\label{item:propSemistabilityConditionsII} The point $x$ satisfies \ref{item:ssS1} and \ref{item:ssS2}.
\item\label{item:propSemistabilityConditionsIII} The point $x$ satisfies \ref{item:S1} and \ref{item:S2}.
\end{enumerate}
\end{prop}

For the proof of Proposition~\ref{prop:S1S2SemistabilityConditions}, we follow \cite[Proposition~2.2]{nakajima2017cherkis} and~\cite[Corollary~2.21]{takayama2016nahm}.
First, we recall the following useful consequence of the subspace conditions~\ref{item:S1} and~\ref{item:S2} from~\cite[Lemma~2.18]{takayama2016nahm} without proof.

\begin{prop}\label{prop:takayamaS1S2implications}
Suppose $x=(A,B_1,B_2,a,b)\in\mu^{-1}(0)$ satisfies~\ref{item:S1} and~\ref{item:S2}. Then, $A$ has full rank.
\end{prop}

\begin{proof}[Proof of Proposition~\ref{prop:S1S2SemistabilityConditions}] We begin with $\textup{\ref{item:propSemistabilityConditionsII}}\Rightarrow\textup{\ref{item:propSemistabilityConditionsI}}$.
In order to apply Mumford's numerical criterion, let $\sigma:\mathbb C^\ast\rightarrow \mathrm{GL}(V_1)\times\mathrm{GL}(V_2)$ be a one-parameter subgroup such that the limit $\lim_{t\to0}\sigma(t).x$ exists. The vector spaces $V_1$ and $V_2$ decompose into weight spaces
\[
V_i=\bigoplus_{n\in\mathbb Z} V_i^n,\quad\textup{where }V_i^n=\{v\in V_i\mid \sigma(t).v=t^nv\textup{ for all }t\in\mathbb C^\ast\},\quad i=1,2,
\]
with corresponding filtrations
\[
F_m V_i=\bigoplus_{n\geq m}V_i^n,\quad  m\in\mathbb Z,i=1,2.
\]
We further view $\mathbb C$ as filtered vector space with filtration $F_m\mathbb C=0$ if $m<0$ and $F_m\mathbb C=\mathbb C$ if $m\geq0$.
The existence of the limit $\lim_{t\to0}\sigma(t).x$ is equivalent to the condition that all the operators $A,B_1,B_2,a,b$ are morphisms of filtered vector spaces. 
Let $n_0<0,n_1>0$ such that $F_{n_0}V_i=0$ and $F_{n_1}V_i=V_i$ for $i=1,2$.
Applying~\ref{item:ssS1} to the pairs $(F_mV_1,F_mV_2)$ with $m>0$
gives
\begin{equation}\label{eq:V2dominatesV1}
\sum_{j=m}^{n_1} \operatorname{dim}(V_2^j)\geq \sum_{j=m}^{n_1} \operatorname{dim}(V_1^j).
\end{equation}
Similarly, we can apply~\ref{item:ssS2} to the pairs
 $(F_mV_1,F_mV_2)$ with $m<0$ which implies
\begin{equation}\label{eq:V1dominatesV2}
\sum_{j=n_0}^{m} \operatorname{dim}(V_2^j)\le \sum_{j=n_0}^{m} \operatorname{dim}(V_1^j).
\end{equation}
Combining~\eqref{eq:V2dominatesV1} and~\eqref{eq:V1dominatesV2} yields
\begin{align*}
\langle \sigma,\theta \rangle &= \sum_{n\in \mathbb Z} n(\operatorname{dim}(V_2^n)-\operatorname{dim}(V_1^n)) \\
&= \underbrace{\Big( \sum_{n=1}^{n_1}\sum_{j=1}^n\operatorname{dim}(V_2^j)-\operatorname{dim}(V_1^j) \Big)}_{\geq0\textup{ by \eqref{eq:V2dominatesV1}}} +
\underbrace{\Big( \sum_{n=n_0}^{-1} \sum_{j=1}^n\operatorname{dim}(V_1^j)-\operatorname{dim}(V_2^j)  \Big)}_{\geq0\textup{ by \eqref{eq:V1dominatesV2}}} \quad\geq 0.
\end{align*}
Thus, $x$ is $\theta$-semistable. To show $\textup{\ref{item:propSemistabilityConditionsI}}\Rightarrow \textup{\ref{item:propSemistabilityConditionsIII}}$, suppose $x$ is $\theta$-semistable and we are given $S\subset V_2$ satisfying the conditions of~\ref{item:S1}. Let $W\subset V_2$ be a vector space complement of $S$ and define a one-parameter subgroup $\sigma:\mathbb C^\ast\rightarrow\mathrm{GL}(V_1)\times\mathrm{GL}(V_2)$ via
\[
\sigma(t)_{\mid V_1}=\operatorname{id}_{V_1},\quad\sigma(t)_{\mid S}=t^{-1}\operatorname{id}_S,\quad \sigma(t)_{\mid W}=\operatorname{id}_W. 
\]
Then, $\lim_{t\to0}\sigma(t).x$ exists and hence Mumford's numerical criterion implies
$\operatorname{dim}(S)\le 0$, i.e. $S=0$. Analogously, if $T\subset V_1$ satisfies the conditions of~\ref{item:S1} then pick a vector space complement $V\subset V_1$ of $T$ and define
one-parameter subgroup $\tau:\mathbb C^\ast\rightarrow\mathrm{GL}(V_1)\times\mathrm{GL}(V_2)$
via
\[
\tau(t)_{\mid V}=t\operatorname{id}_V,\quad\tau(t)_{\mid T}=\operatorname{id}_T,\quad \tau(t)_{\mid V_2}=\operatorname{id}_{V_2}.
\]
Again, $\lim_{t\to0}\tau(t).x$ exists and Mumford's numerical criterion gives $-\operatorname{dim}(V)\geq 0$, i.e. $T=V_1$. 
To prove finally $\textup{\ref{item:propSemistabilityConditionsIII}}\Rightarrow \textup{\ref{item:propSemistabilityConditionsII}}$, we first consider the case $\operatorname{dim}(V_1)\le \operatorname{dim}(V_2)$. Suppose $S_1\subset V_1,S_2\subset V_2$ satisfy the conditions of~\ref{item:ssS1}. The condition $\mu(x)=0$ implies that $B_2$ maps $\mathrm{ker}(A)\cap \mathrm{ker}(b)$ to $\mathrm{ker}(A)$. As $S_2$ is contained in $\mathrm{ker}(b)$, this implies that $S_2\cap \mathrm{ker}(A)$ is $B_2$-invariant. Hence, $S_2\cap \mathrm{ker}(A)$ satisfies the conditions of \ref{item:S1} which gives $S_2\cap \mathrm{ker}(A)=0$.
Thus, $A_{\mid S_2}$ is injective and we conclude $\operatorname{dim}(S_2)\geq\operatorname{dim}(S_1)$ as $A(S_2)\subset S_1$ which gives~\ref{item:ssS1}. The property~\ref{item:ssS2} follows immediately from the surjectivity of $A$ which follows from Proposition~\ref{prop:takayamaS1S2implications}. 
Next, $\operatorname{dim}(V_1)>\operatorname{dim}(V_2)$. In this case, Proposition~\ref{prop:takayamaS1S2implications} gives that $A$ is injective which directly implies~\ref{item:ssS1}. Assume $T_1\subset V_1,T_2\subset V_2$ satisfy the conditions of~\ref{item:ssS2}. Since $\mu(x)=0$, we conclude that $T_1+\mathrm{Im}(A)$ is $B_1$-invariant and thus, $T_1+\mathrm{Im}(A)$ satisfies the conditions of \ref{item:S2} and hence $T_1+\mathrm{Im}(A)=V_1$. As $A(T_2)\subset T_1$, this implies $\operatorname{codim}(T_1)\le \operatorname{codim}(T_2)$.
\end{proof}

Using Proposition~\ref{prop:takayamaS1S2implications} and the stability conditions \ref{item:S1} and \ref{item:S2}, Takayama constructed certain normal forms for the points in $\mathrm{tri}(V_1,V_2)$. These normal forms coincide with Hurtubise's normal forms of solutions of Nahm's equation over the interval, \cite{hurtubise1989classification}. Based on this, Takayama proved in {\cite[Proposition~2.20]{takayama2016nahm}} the following identifications:

\begin{prop}\label{prop:triangleAffineStructure}
Let $m_1=\operatorname{dim}(V_1),m_2=\operatorname{dim}(V_2)$. Then, the following holds:
\begin{enumerate}[label=(\roman*)]
\item If $m_1\ne m_2$, then $\mathrm{tri}(V_1,V_2)$ is isomorphic to $\mathrm{GL_m}\times \mathbb C^{m^2+m+n}$ where $m=\operatorname{min}(m_1,m_2)$ and $n=\operatorname{max}(m_1,m_2)$.
\item If $m_1=m_2$, then $\mathrm{tri}(V_1,V_2)$ is isomorphic to $\mathrm{GL_{m_1}}\times \mathbb C^{m_1^2+2m_1}$.
\end{enumerate}
In particular, $\mathrm{tri}(V_1,V_2)$ is always smooth and affine.
\end{prop}
In addition, it was shown in~\cite[Proposition~5.7]{nakajima2017cherkis} that $\mathrm{tri}(V_1,V_2)$ always carries an algebraic $(\mathrm{GL}(V_1)\times\mathrm{GL}(V_2))$-invariant symplectic form. It admits a moment map:
\begin{prop}[Moment map]\label{prop:MomentMapTrianglePart} The morphism
\[
m:\mathrm{tri}(V_1,V_2)\rightarrow \operatorname{End}(V_1)\oplus \operatorname{End}(V_2),\quad (A,B_1,B_2,a,b)\mapsto(B_1,-B_2)
\]
is a moment map for the symplectic structure on $\mathrm{tri}(V_1,V_2)$.
\end{prop}

\subsection{Brane diagrams and affine brane varieties}\label{subsection:AffineBraneVariety}
For the next step in the construction of bow varieties we use multiple triangle parts as building blocks to construct varieties that we call \emph{affine brane varieties}. For their construction, we use the language of brane diagrams from \cite{rimanyi2020bow} which we briefly recall.

A {brane diagram} is an object like this:
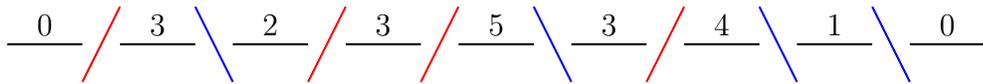
\[
\begin{tikzpicture}
[scale=.5]
\draw[thick] (0,0)--(2,0);
\draw[thick] (3,0)--(5,0);
\draw[thick] (6,0)--(8,0);
\draw[thick] (9,0)--(11,0);
\draw[thick] (12,0)--(14,0);
\draw[thick] (15,0)--(17,0);
\draw[thick] (18,0)--(20,0);
\draw[thick] (21,0)--(23,0);
\draw[thick] (24,0)--(26,0);

%red branes
\draw [thick,red] (2,-1) --(3,1); 
\draw [thick,red] (8,-1) --(9,1); 
\draw [thick,red] (11,-1) --(12,1); 
\draw [thick,red] (18,1) --(17,-1);

%blue branes
\draw [thick,blue] (5,1) --(6,-1); 
\draw [thick,blue] (15,-1) --(14,1);  
\draw [thick,blue] (21,-1) --(20,1); 
\draw [thick,blue] (24,-1) --(23,1);  

%numbers
\node at (1,0.5){$0$};
\node at (4,0.5){$3$};
\node at (7,0.5){$2$};
\node at (10,0.5){$3$};
\node at (13,0.5){$5$};
\node at (16,0.5){$3$};
\node at (19,0.5){$4$};
\node at (22,0.5){$1$};
\node at (25,0.5){$0$};
\end{tikzpicture}
\]
That is, a \textit{brane diagram} is a finite sequence of black horizontal lines drawn from left to right. Between each consecutive pair of horizontal lines there is either a blue SE-NW line \textcolor{blue}{$\backslash$} or a red SW-NE line \textcolor{red}{$\slash$}. Each horizontal line $X$ is labeled by a non-negative integer $d_X$. Moreover, we demand that the first and the last horizontal line is labeled by $0$.

\begin{remark}
The terminology in ~\cite{nakajima2017cherkis} and~\cite{rimanyi2020bow} is motivated from string theory; the horizontal lines are called D3 branes, the blue lines D5 branes and the red lines NS5 branes. Since it suffices for our purposes to view brane diagrams as purely  combinatorial objects, we will refer to the lines just by their colors.
\end{remark}

\begin{definition}\label{uglyDef}
We denote by $h(\mathcal D),b(\mathcal D)$ and $r(\mathcal D)$ the set of black, blue and red lines in a given brane diagram $\mathcal D$. If $M=|r(\mathcal D)|$ is the number of red lines, we denote by $V_1,\ldots,V_M$ the actual lines numbered from right to left. If $N$ is the number of blue lines then we denote by $U_1,\ldots,U_N$  these lines, numbered from left to right. The black lines are denoted by $X_1,\ldots,X_{M+N+1}$ also numbered from left to right.
\end{definition}

\begin{remark}
Note that the convention to number the red lines from right to left differs from the convention in~\cite{rimanyi2020bow}.
\end{remark}

Among the black lines impose the equivalence relation $X \sim X'$ if all colored lines between $X$ and $X'$ are blue. We call the equivalence classes with respect to this relation \textit{blue segments}.
Given a blue line U, we denote the black line directly to the left respectively to the right of $U$ by $U^-$ resp. $U^+$. Similarly, the colored lines directly left and right to a black line $X$ are denoted by $X^-$ and $X^+$. Moreover, to each black line $X$, we attach the vector space $W_X:=\mathbb C^{d_X}$ and set $W:=\bigoplus_{X\in h(\mathcal D)} W_X$. We denote $W_{X_l}$ also by $W_l$.

Now, we continue with the construction of affine brane varieties. This is a variety which is assigned to a brane diagram $\mathcal D$ in the following way:
for any red line $V$, define the variety
\[
\mathbb M_V:= \operatorname{Hom}(W_{V^+},W_{V^-})\oplus  \operatorname{Hom}(W_{V^-},W_{V^+}).
\]
The elements of $\mathbb M_V$ are denoted by $y_V=(C_V,D_V)$ and we equip $\mathbb M_V$ with the usual $(\mathrm{GL}(W_{V^-})\times \mathrm{GL}(W_{V^+}))$-action
\[
(g_-,g_+).(C_V,D_V)= (g_-C_Vg_+^{-1},g_+D_Vg_-^{-1}),\quad g_-\in\mathrm{GL}(W_{V^-}),g_+\in \mathrm{GL}(W_{V^+}).
\]
The variety $\mathbb M_V$ admits the algebraic $(\mathrm{GL}(W_{V^-})\times \mathrm{GL}(W_{V^+}))$-equivariant symplectic form $dC_V \wedge dD_V$
%\[
%\mathbb M_V\times\mathbb M_V\rightarrow \mathbb C,\quad (C_V,D_V,C'_V,D'_V)\mapsto \operatorname{tr}(D'_VC_V-D_VC'_V)
%\]
with associated moment map
\[
m_V:\mathbb M_V \rightarrow \operatorname{End}(V^-)\oplus\operatorname{End}(V^+),\quad (C_V,D_V)\mapsto (-C_VD_V,D_VC_V).
\]
To any blue line $U$, we attach the variety $\mathbb M_U:=\mathrm{tri}(W_{U^-},W_{U^+})$. We denote the elements of $\mathbb M_U$ by $x_U=(A_U,B_U^-,B_U^+,a_U,b_U)$ and by $m_U$ the moment map from Proposition~\ref{prop:MomentMapTrianglePart}. 
\begin{definition} The \textit{affine brane variety associated to $\mathcal D$} is defined as
\begin{equation}\label{eq:affinevar}
\widetilde{\mathcal M}(\mathcal D):= \prod_{U\in b(\mathcal D)} \mathbb M_U \times \prod_{V\in r(\mathcal D)}\mathbb M_V.
\end{equation}
\end{definition}
We denote points of $\widetilde{\mathcal M}(\mathcal D)$ as tuples $((A_U,B_U^+,B_U^-,a_U,b_U)_U,(C_V,D_V)_V)$.
%Finally, we define the variety
%\[
%\widetilde{\mathcal M}(\mathcal D):= \prod_{\textup{$U$ D5}} \mathbb M_U \times \prod_{\textup{$V$ NS5}}\mathbb M_V
%\]
%which is called the \textit{prequotient} associated to $\mathcal D$.
By construction, $\widetilde{\mathcal M}(\mathcal D)$ is a smooth affine variety endowed with an algebraic (base change) action of the group 
\[
\mathcal G:=\prod_{X\in h(\mathcal D)} \mathrm{GL}(W_X).
\]
Since it is a product of algebraic symplectic varieties, $\widetilde M(\mathcal D)$ admits an algebraic symplectic form which is $\mathcal G$-equivariant. Furthermore, $\widetilde M(\mathcal D)$ admits the moment map
\begin{equation}\label{m}
m: \widetilde{M}(\mathcal D)\rightarrow \bigoplus_{\textup{$X\in h(\mathcal D)$}}\operatorname{End}(W_X),
\quad ((x_U)_U,(y_V)_V)\mapsto \sum_{\textup{$U\in b(\mathcal D)$}}m_U(x_U) +  \sum_{\textup{$V\in r(\mathcal D)$}}m_V(y_V).
\end{equation}
Explicitly, for a black line $X$, the corresponding component $m((x_U)_U,(y_V)_V)_X$ is given by
\begin{equation}\label{eq:MomentMapEquations}
m((x_U)_U,(y_V)_V)_X=
\begin{cases}
			   B_{X^+}^- - B_{X^-}^+ &\textup{if $X^+,X^-$ are both blue,} \\
			D_{X^-}C_{X^-}-C_{X^+}D_{X^+} &\textup{if $X^+,X^-$ are both red,}\\
			D_{X^-}C_{X^-} + B_{X^+}^- &\textup{if $X^+$ is blue and $X^-$ is red,}\\
			-C_{X^+}D_{X^+}-B_{X^-}^+ &\textup{if $X^+$ is red and $X^-$ is blue.}
\end{cases}
\end{equation}
The conditions \ref{item:S1} and \ref{item:S2} imply that the points of $m^{-1}(0)$ satisfy the following injectivity resp. surjectivity properties:
\begin{prop}\label{prop:ZeroLocusMomentMapInjectivityAndSurjectivity}
For each $x=((A_U,B_U^+,B_U^-,a_U,b_U)_U,(C_V,D_V)_V)\in m^{-1}(0)$ the following holds:
\begin{enumerate}[label=(\roman*)]
\item\label{item:ZeroLocusMomentMapInjectivity}
Given a local configuration in $\mathcal D$ of the form:
\begin{equation*}
\begin{tikzpicture}
[scale=.5]
\draw[thick] (0,0)--(2,0);
\draw[thick] (3,0)--(5,0);
\draw[thick] (6,0)--(8,0);

\draw [thick,red] (5,-1) --(6,1);
\draw [thick,blue] (3,-1) --(2,1);

\node[red] at (5,-1.5) {$V$};
\node[blue] at (3,-1.5) {$U$};
\node at (1,0.5){$d_{j-1}$};
\node at (4,0.5){$d_j$};
\node at (7,0.5){$d_{j+1}$};

\end{tikzpicture}
\end{equation*}
Then, the map $\alpha:W_j\rightarrow W_{j-1}\oplus W_{j+1}\oplus\mathbb C$, $v\mapsto (A_U(v),D_V(v), b_U(v))$ is injective.
\item\label{item:ZeroLocusMomentMapSurjectivity}
Given a local configuration in $\mathcal D$ of the form:
\begin{equation*}
\begin{tikzpicture}
[scale=.5]
\draw[thick] (0,0)--(2,0);
\draw[thick] (3,0)--(5,0);
\draw[thick] (6,0)--(8,0);

\draw [thick,blue] (6,-1) --(5,1);
\draw [thick,red] (2,-1) --(3,1);

\node[red] at (2,-1.5) {$V$};
\node[blue] at (6,-1.5) {$U$};
\node at (1,0.5){$d_{j-1}$};
\node at (4,0.5){$d_j$};
\node at (7,0.5){$d_{j+1}$};

\end{tikzpicture}
\end{equation*}
Then, the map $\beta:W_{j-1}\oplus W_{j+1}\oplus\mathbb C \rightarrow W_j$, $(v_1,v_2,v_3) \mapsto (D_V(v_1),A_U(v_2), a_U(v_3))$ is surjective.
\end{enumerate}
\end{prop}

\begin{proof} 
For \ref{item:ZeroLocusMomentMapInjectivity}, note that $\mathrm{ker}(D_V)\subset \mathrm{ker}(B_U^+)$ as $C_VD_V=-B_U^+$ by the moment map equations \eqref{eq:MomentMapEquations}. Thus, $\mathrm{ker}(\alpha)$ is $B_U^+$-invariant and therefore satisfies the conditions of \ref{item:S1}. Hence, $\mathrm{ker}(\alpha)=0$ which gives the injectivity of $\alpha$.
The proof of \ref{item:ZeroLocusMomentMapSurjectivity} is analogous.
\end{proof}

\begin{example}\label{example:AffineBraneVariety}
Let $\mathcal D=0\textcolor{red}{\slash}1\textcolor{blue}{\backslash}1\textcolor{blue}{\backslash}1\textcolor{red}{\slash}0$.
The data specifying an element of $\widetilde{\mathcal M}(\mathcal D)$ is encoded in the following diagram:
\begin{equation}\label{eq:ExampleAffineBraneVarietyQuiver}
\begin{tikzcd}
0\arrow[r, in=-150,out=-30,"D_2", swap]&
\mathbb C \arrow[l, in=-330,out=-210,"C_2", swap]\arrow[in=60,out=120,"B_1^-",loop,looseness=5]
&&
\mathbb C \arrow[ll,"A_1"]\arrow[in=100,out=150,"B_1^+",loop,looseness=5]\arrow[in=30,out=80,"B_2^-",loop,looseness=5]\arrow[dl,"b_1"]
&&
\mathbb C \arrow[in=60,out=120,"B_2^+",loop,looseness=5] \arrow[ll,"A_2"] \arrow[r, in=-150,out=-30,"D_1", swap] \arrow[dl,"b_2"]
&
0\arrow[l, in=-330,out=-210,"C_1", swap] \\
&&
\mathbb C \arrow[ul,"a_1"]&&
\mathbb C \arrow[ul,"a_2"]
\end{tikzcd}
\end{equation}
The conditions~\ref{item:S1} and~\ref{item:S2} are equivalent to $A_1,A_2\ne 0$. Thus, we have an isomorphism of varieties
$
\widetilde{\mathcal M}(\mathcal D)\xrightarrow{\sim} (\mathbb C^\ast \times \mathbb C^3)^2$
given by
\[ ((A_i,B_i^-,B_i^+,a_i,b_i)_{i=1,2},(C_i,D_i)) \mapsto (A_1,B_1^+,a_1,b_1,A_2,B_2^+,a_2,b_2).
\]
\end{example}

We next construct bow varieties as hamiltonian reductions of affine brane varieties.

\subsection{Bow varieties}\label{subsection:bowVarieties}
Let $\mathcal D$ be a brane diagram and $\widetilde{\mathcal M}(\mathcal D)$ its associated affine brane diagram. Moreover, let $\mathcal G$ and $m$ be as in the previous subsection. 
\begin{definition}
The \textit{bow variety associated to $\mathcal D$} is defined as 
\[
\mathcal C(\mathcal D):= m^{-1}(0)/\!\!/_{\!\!\chi}\; \mathcal G,
\]
where \[
\chi:\mathcal G\rightarrow\mathbb C^\ast,\quad (g_X)_X\mapsto \prod_{X\in M_{\mathcal D}}\operatorname{det}(g_X)
\]
and  $M_{\mathcal D}$ denote the set of black lines $X$ such that $X^-$ is red.
\end{definition}
%Let $M_{\mathcal D}$ denote the set of D3 branes $X$ such that $X^-$ is an NS5 brane.
%We define the cocharacter
%\[
%\chi:\mathcal G\rightarrow\mathbb C^\ast,\quad (g_X)_X\mapsto \prod_{X\in M_{\mathcal D}}\operatorname{det}(g_X).
%\]
%The \textit{bow variety} corresponding to the brane diagram $\mathcal D$ is defined as the GIT quotient \[\mathcal C(\mathcal D):= m^{-1}(0)/\!/_\chi \mathcal G.\]
%The points of $\mathcal C(\mathcal D)$ can be characterized as follows:
Recall from Theorem~\ref{thm:PointsOfGITQuotients}, that the points of $\mathcal C(\mathcal D)$ are characterized by the $\chi$-semistability condition.
Using Mumford's numerical criterion, Nakajima and Takayama, {\cite[Proposition~2.8]{nakajima2017cherkis}}, proved the following criterion for $\chi$-(semi)stability:

\begin{prop}[(Semi-)Stability for bow varieties]\label{prop:nakajimaSemistabilityCriterion} Let \[x=((A_U,B_U^+,B_U^-,a_U,b_U)_U,(C_V,D_V)_V))\in m^{-1}(0).\] Then the following holds:
\begin{enumerate}[label=(\roman*)]
\item The point $x$ is $\chi$-semistable if and only if $x$ satisfies the following condition:
%\begin{enumerate}[label=($\nu$)]
%\item\label{item:nakajimaSemistabilityViaMumford}
for all graded subspaces $T=\bigoplus_{\textup{$X\in h(\mathcal D)$}}T_X \subset W$ such that $\mathrm{Im}(a_U)+\mathrm{A_U}\subset T_{U^-}$ and $A_U$ induces an isomorphism $W_{U^+}/T_{U^+}\rightarrow W_{U^-}/T_{U^-}$ for all $U\in b(\mathcal D)$, we have
\begin{equation}\label{eq:SemistabilityCodimensionInequality}
\sum_{X\in M_{\mathcal D}} \operatorname{codim}(T_X) \le 0.
\end{equation}
%\end{enumerate}
\item The point $x$ is $\chi$-stable if and only if we have an strict inequality in~\eqref{eq:SemistabilityCodimensionInequality} unless $T=W$.
\end{enumerate}
\end{prop}
We immediately obtain that the semistable and the stable locus of $m^{-1}(0)$ coincide:
\begin{cor}[Semistable=stable]\label{cor:SemistabilityEqualsStability}
We have $m^{-1}(0)^{\mathrm{ss}}=m^{-1}(0)^{\mathrm{s}}$.
\end{cor}
By Theorem~\ref{thm:PointsOfGITQuotients} and Corollary~\ref{cor:SemistabilityEqualsStability}, we have that $\mathcal C(\mathcal D)=m^{-1}(0)^{\mathrm s}/\mathcal G$ is a geometric quotient. Moreover, Nakajima and Takayama proved that the $\chi$-stable points of $m^{-1}(0)$ have trivial stabilizers~\cite[Lemma~2.10]{nakajima2017cherkis} and that $m^{-1}(0)^{\mathrm s}$ is actually a smooth variety~\cite[Proposition~2.13]{nakajima2017cherkis}. 
Thus, by Luna's slice theorem, $\mathcal C(\mathcal D)$ is a smooth variety and the quotient morphism $\pi:m^{-1}(0)^{\mathrm s}\rightarrow \mathcal C(\mathcal D)$ is an \'etale princial $\mathcal G$-bundle. Since $\mathcal G$ is a special group, see e.g. \cite[Theorem~11.4]{milne1998lectures}, the quotient morphism is actually a principal $\mathcal G$-bundle in the Zariski topology.
Moreover, since $\widetilde M(\mathcal D)$ is algebraic symplectic with moment map $m$ the algebraic version of the Marsden--Weinstein theorem (see e.g.~\cite[Theorem~9.53]{kirillov2016quiver}) gives that also $\mathcal C(\mathcal D)$ is algebraic symplectic. We denote the symplectic form on $\mathcal C(\mathcal D)$ by $\omega_{\mathcal C(\mathcal D)}$.

The Example~\ref{example:AffineBraneVariety} in fact reproduces a very familiar quasi-projective variety:
\begin{example} \label{THEexample} We show that $\mathcal C(\mathcal D)$, for $\mathcal D$ as in Example~\ref{example:AffineBraneVariety}, is isomorphic to the cotangent bundle of the projective line $T^\ast\mathbb P^1$. For this, let $S$ denote the tautological bundle and $Q$ the universal quotient bundle of $\mathbb P^1$. Then, $T^\ast\mathbb P^1$ is isomorphic to the total space of the vector bundle $\operatorname{Hom}(S,Q)$. So the points of $T^\ast\mathbb P^1$ are given by
\[
T^\ast\mathbb P^1 = \{(V,f)\mid V\in\mathbb P^1, f\in \operatorname{End}(\mathbb C^2), \mathrm{Im}(f)\subset V,V\subset \mathrm{ker}(f)\}.
\]
Recall the data to specify elements of $\mathcal C(\mathcal D)$ from~\eqref{eq:ExampleAffineBraneVarietyQuiver}, and that the pair~\ref{item:S1} and~\ref{item:S2} is equivalent to $A_1,A_2\ne 0$. By Proposition~\ref{prop:nakajimaSemistabilityCriterion}, the $\chi$-semistability condition is equivalent to $(a_1,a_2)\ne (0,0)$.
Moreover, by definition, see ~\eqref{eq:nonMomentMap} and~\eqref{eq:MomentMapEquations}, a tuple
\[((A_i,B_i^+,B_i^-,a_i,b_i)_{i=1,2},(C_i,D_i)_{i=1,2})\] is contained in $m^{-1}(0)$ if and only the following equations are satisfied:
\begin{gather*}
B_1^-A_1-A_1B_1^++a_1b_1 =0,\quad 
B_2^-A_2-A_2B_2^++a_2b_2 =0 ,\quad 
B_1^-=0, \quad B_1^+=B_2^-,\quad B_2^+=0.
\end{gather*}
%We denote the elements of $\mathcal C(\mathcal D)$ according to the following diagram:
%\[
%\begin{tikzcd}
%0\arrow[r, in=150,out=30,"D_1"]&
%\mathbb C \arrow[l, in=330,out=210,"C_1"]\arrow[in=60,out=120,"B_1^-",loop]
%&&
%\mathbb C \arrow[ll,"A_1"]\arrow[in=100,out=150,"B_1^+",loop]\arrow[in=30,out=80,"B_2^-",loop]\arrow[dl,"b_1"]
%&&
%\mathbb C \arrow[in=60,out=120,"B_2^+",loop] \arrow[ll,"A_2"] \arrow[r, in=150,out=30,"D_1"] \arrow[dl,"b_1"]
%&
%0\arrow[l, in=330,out=210,"C_2"] \\
%&&
%\mathbb C \arrow[ul,"a_1"]&&
%\mathbb C \arrow[ul,"a_2"]
%\end{tikzcd}
%\]
%The conditions~\ref{item:S1} and~\ref{item:S2} are equivalent to $A_1,A_2\ne 0$ and the $\chi$-stability is equivalent to $(a_1,a_2)\ne (0,0)$. Moreover, the above operators satisfy the following equations
%\begin{gather*}
%B_1^-A_1-A_1B_1^++a_1b_1 =0,\quad 
%B_1^-A_1-A_1B_1^++a_1b_1 =0 \\
%B_1^-=0, \quad B_1^+=B_2^-,\quad B_2^+=0.
%\end{gather*}
It follows that $\mathrm{ker}(a_1,a_2)$ is of dimension $1$ and we have
\[
\frac{a_1b_1}{A_1} + \frac{a_2b_2}{A_2} =0.
\]
Thus, there exists an isomorphism of varieties $\mathcal C(\mathcal D)\xrightarrow{\sim}T^\ast\mathbb P^1$ given by
\[
[(A_i,B_i^-,B_i^+,a_i,b_i)_{i=1,2},(C_i,D_i)_{i=1,2}]\mapsto \Big(\mathrm{ker}(a_1A_1^{-1},a_2), \begin{pmatrix} b_1 \\ b_2A_2^{-1} \end{pmatrix} \begin{pmatrix} a_1A_1^{-1}& a_2
\end{pmatrix} \Big).
\]
\end{example}
This example shows a very special instance of the general fact that each Nakajima quiver variety of type $A$ is isomorphic to a bow variety,~\cite[Theorem~2.15]{nakajima2017cherkis}. In particular, the family of bow varieties includes the cotangent bundles of partial flag varieties.

A further interesting property of bow varieties is that the injectivity and surjectivity constrains  yield that $\mathcal C(\mathcal D)$ is empty unless $\mathcal D$ satisfies the following properties:
\begin{cor}\label{cor:EmptyBowVarieties} 
If $\mathcal C(\mathcal D)\ne \emptyset$ then
$
d_j\le d_{j-1}+d_{j+1}+1
$
for all local configurations $d_{j-1}\textcolor{red}{\slash}d_{j}\textcolor{blue}{\backslash} d_{j+1}$ and $d_{j-1}\textcolor{blue}{\backslash} d_j\textcolor{red}{\slash} d_{j+1}$ in $\mathcal D$.
\end{cor}
Consequently, we restrict our attention to bow varieties corresponding to brane diagrams satisfying the inequalities from Corollary~\ref{cor:EmptyBowVarieties}:
\begin{definition} We call a brane diagram $\mathcal D$ \textit{admissible} if 
$
d_j\le d_{j-1}+d_{j+1}+1
$
for all local configurations $d_{j-1}\textcolor{red}{\slash}d_{j}\textcolor{blue}{\backslash} d_{j+1}$ and $d_{j-1}\textcolor{blue}{\backslash} d_j\textcolor{red}{\slash} d_{j+1}$ in $\mathcal D$.
\end{definition}

\begin{assumption} From now on we assume that each brane diagram $\mathcal D$ is admissible.
\end{assumption}

\begin{remark} Nakajima and Takayama gave a more general definition of  bow varieties depending on more stability parameters
$\nu_{\sigma}^{\mathbb C}$ and $\nu_{\sigma}^{\mathbb R}$. For simplicity, we only consider bow varieties corresponding to the specializations $\nu_{\sigma}^{\mathbb C}=0$ and $\nu_{\sigma}^{\mathbb R}=-1$. One nice feature of this family of bow varieties is that they are smooth, which is not true in general. 
\end{remark}

Bow varieties share many properties with Nakajima quiver varieties; they for instance also admit a family of tautological vector bundles. Given a black line $X$, we have a free diagonal $\mathcal G$-action on the product $W_X\times\widetilde{\mathcal M}(\mathcal D)^{\mathrm s}$, and we can define the following vector bundles. 
\begin{definition}\label{definition:TautologicalBundles}
The \textit{tautological (vector) bundle} $\xi_X$ corresponding to $X$ is defined as the geometric quotient
$$
\xi_X:= W_X\times m^{-1}(0)^{\mathrm s}/\mathcal G.
$$ Furthermore, we call $\xi:=\bigoplus_{\textup{$X\in h(\mathcal D)$}}\xi_X$ the \textit{full tautological bundle over $\mathcal C(\mathcal D)$}.
\end{definition}
Note that as the projection $m^{-1}(0)^{\mathrm s}\rightarrow\mathcal C(\mathcal D)$ is a principal $\mathcal G$-bundle in the Zariski topology, $\xi_X$ is indeed a vector bundle over $\mathcal C(\mathcal D)$ in the Zariski topology.

\subsection{Torus action}
\label{subsection:torusAction}
We will use two kind of torus actions on bow varieties. The first one is pretty obvious from the construction of the varieties and leaves the symplectic form on $\mathcal C(\mathcal D)$ invariant. We refer to this action as the \textit{obvious action}. The second one was introduced in~\cite[Section~6.9.3]{nakajima2017cherkis}. Since it scales the symplectic form, we refer to it as the \textit{scaling action}. We follow here the conventions from~\cite[Section~3.1]{rimanyi2020bow}, for the precise connection to the definition of Nakajima and Takayama see~\cite[Section~3.4]{rimanyi2020bow}.

The following two tori and the combination $\mathbb T=\mathbb A\times \mathbb C^\ast_h$ will be used:
\begin{itemize}
\item $\mathbb A=(\mathbb C^\ast)^N$, and its elements are denoted by $(t_1,\ldots,t_N)$ or $(t_U)_{U\in b(\mathcal D)}$ or just by $(t_U)_U$.
\item $\mathbb C^\ast_h=\mathbb C^\ast$, and its elements are usually called $h$.
\end{itemize}
Recall from Definition~\ref{uglyDef} that $N=|b(\mathcal D)|$.

\subsubsection*{The obvious action} For any brane diagram $\mathcal D$, the torus $\mathbb A$ acts algebraically on $\widetilde{\mathcal M}(\mathcal D)$ via
\[
(t_U)_{U}.((A_U,B_U^+,B_U^-,a_U,b_U)_U,(C_V,D_V)_V)=
((A_U,B_U^+,B_U^-,a_Ut_U^{-1},t_Ub_U)_U,(C_V,D_V)_V).
\]
By construction, the $\mathbb A$-action commutes with the $\mathcal G$-action, and $m$ from \eqref{m} is $\mathbb A$-equivariant. From the explicit description, \cite[Proposition~5.7]{nakajima2017cherkis}, of the symplectic form on $\widetilde{\mathcal M}(\mathcal D)$ it follows that the symplectic form on $\widetilde{\mathcal M}(\mathcal D)$ is $\mathbb A$-invariant. Thus, we obtain an induced $\mathbb A$-action on $\mathcal C(\mathcal D)$, explicitly given by
\[
(t_U)_{U}.[(A_U,B_U^+,B_U^-,a_U,b_U)_U,(C_V,D_V)_V]=
[(A_U,B_U^+,B_U^-,a_Ut_U^{-1},t_Ub_U)_U,(C_V,D_V)_V].
\]
As the symplectic form on $\widetilde{\mathcal M}(\mathcal D)$ is $\mathbb A$-invariant, so is the symplectic form $\omega_{\mathcal C(\mathcal D)}$ on $\mathcal C(\mathcal D)$.

\subsubsection*{The scaling action}
%There is a further $\mathbb C^\ast$-action on $\mathcal C(\mathcal D)$ introduced in~\cite[Section~6.9.3]{nakajima2017cherkis}. In our exposition of this action, we follow the conventions from~\cite[Section~3.1]{rimanyi2020bow}. The relation between the two torus actions is explained in~\cite[Section~3.4]{rimanyi2020bow}.
There is an algebraic $\mathbb C^\ast_h$-action on $\widetilde{\mathcal M}(\mathcal D)$ given by
\[
h.((A_U,B_U^+,B_U^-,a_U,b_U)_U,(C_V,D_V)_V)=((A_U,hB_U^+,hB_U^-,a_U,hb_U)_U,(hC_V,D_V)_V).
\]
This action also commutes with the $\mathcal G$-action and $m^{-1}(0)^{\mathrm s}$ is a $\mathbb C^\ast_h$-invariant locally closed subvariety of $\widetilde{\mathcal M}(\mathcal D)$. It follows from \cite[Proposition~5.7]{nakajima2017cherkis} that the symplectic form on $\widetilde{\mathcal M}(\mathcal D)$ is scaled by this action.
Hence, we get an induced $\mathbb C^\ast_h$-action on $\mathcal C(\mathcal D)$ given by
\[
h.[(A_U,B_U^+,B_U^-,a_U,b_U)_U,(C_V,D_V)_V]=[(A_U,hB_U^+,hB_U^-,a_U,hb_U)_U,(hC_V,D_V)_V].
\]
By construction, $\omega_{\mathcal C(\mathcal D)}$ gets also scaled by the $\mathbb C^\ast_h$-action. 

By construction, the obvious and the scaling action commute and give rise to an action of $\mathbb T:=\mathbb A\times\mathbb C_h^\ast$. The $\mathbb T$-equivariant cohomology ring of $\mathcal C(\mathcal D)$ will be one of the main players in our study of the theory of stable envelopes in the context of bow varieties.

\subsection{Hanany\textbf{--}Witten transition}
The family of bow varieties comes with an interesting collection of isomorphisms between bow varieties, called Hanany--Witten isomorphisms. We will use here some of their properties and the underlying combinatorics. For the explicit construction see \cite[Section~7]{nakajima2017cherkis} and also the exposition in~\cite[Section~3.3]{rimanyi2020bow}. 

We start be describing the combinatorics underlying the Hanany--Witten isomorphisms. 
\begin{definition}
Let $\mathcal D$ and $\tilde{\mathcal D}$ be brane diagrams. We say that $\tilde{\mathcal D}$ \textit{is obtained from $\mathcal D$ via a Hanany--Witten transition} if $\tilde{\mathcal D}$ differs from $\mathcal D$ by performing a local move of the form
%\[
%\begin{tikzpicture}
%[scale=.5]
%\draw[thick] (0,0)--(2,0);
%\draw[thick] (3,0)--(5,0);
%\draw[thick] (6,0)--(8,0);
%
%\draw [thick,red] (2,-1) --(3,1);
%\draw [thick,blue] (6,-1) --(5,1);
%
%\draw[->, line width=0.5mm](10,0) -- (12,0);
%
%\draw[thick] (14,0)--(16,0);
%\draw[thick] (17,0)--(19,0);
%\draw[thick] (20,0)--(22,0);
%
%\draw [thick,red] (19,-1) --(20,1);
%\draw [thick,blue] (17,-1) --(16,1);
%
%\node at (1,0.5){$d_1$};
%\node at (4,0.5){$d_2$};
%\node at (7,0.5){$d_3$};
%\node at (15,0.5){$d_1$};
%\node at (18,0.5){$\tilde{d}_2$};
%\node at (21,0.5){$d_3$};
%\end{tikzpicture}
%\]
%or

\begin{equation*}
\begin{tikzpicture}
[scale=.5]
\draw[thick] (0,0)--(2,0);
\draw[thick] (3,0)--(5,0);
\draw[thick] (6,0)--(8,0);

\draw [thick,red] (5,-1) --(6,1);
\draw [thick,blue] (3,-1) --(2,1);

\node[red] at (5,-1.5) {$V_j$};
\node[blue] at (3,-1.5) {$U_i$};
\node at (1,0.5){$d_{k-1}$};
\node at (4,0.5){$d_k$};
\node at (7,0.5){$d_{k+1}$};
\node at (11,0.5){\tiny{HW}};
\draw[->, decorate, decoration=zigzag, line width=0.3mm](10,0) -- (12,0);
\draw[thick] (14,0)--(16,0);
\draw[thick] (17,0)--(19,0);
\draw[thick] (20,0)--(22,0);
\draw [thick,red] (16,-1) --(17,1);
\draw [thick,blue] (20,-1) --(19,1);
\node[red] at (16,-1.5) {$V_j$};
\node[blue] at (20,-1.5) {$U_i$};
\node at (15,0.5){$d_k$};
\node at (18,0.5){$\tilde{d}_{k+1}$};
\node at (21,0.5){$d_k$};
\end{tikzpicture}
\end{equation*}
where $d_{k-1}+d_{k+1}+1=d_k+\tilde d_k$, and write short $\HW$.
\end{definition}
The following result, \cite[Proposition~7.1]{nakajima2017cherkis} and \cite[Theorem~3.9]{rimanyi2020bow}, indicates the importance of the Hanany--Witten transforms:
\begin{prop}[HW-isomorphisms]\label{prop:HWTransition}
Assume $\HW$ exchanging the blue line $U_i$ with a red line. 
Then, there exists a $\mathbb T$-equivariant isomorphism of symplectic varieties \[\Psi:\mathcal C(\mathcal D)\rightarrow \mathcal C(\tilde{\mathcal D}),\]
where the $\mathbb T$-action $\mathcal C(\tilde{\mathcal D})$ is twisted by the algebraic group automorphism 
\begin{equation*}\label{HWiso}
\rho_i:\mathbb T\rightarrow \mathbb T,\quad (t_1,\ldots,t_N,h)\mapsto (t_1,\ldots, t_{i-1},ht_i,t_{i+1},\ldots,t_N,h).
\end{equation*}
That is, $
\Psi(t.x)=\rho_i(t).\Psi(x)$, for all $x\in\mathcal C(\mathcal D)$, $t\in\mathbb T$.
\end{prop}

These \textit{Hanany--Witten isomorphisms} will be crucial later. They allow for instance to move  all red lines in a brane diagram to the left of all blue lines while not changing the isomorphism type of the bow variety.  This significantly simplifies e.g. the calculation of $\mathbb T$-fixed points.
%For our purposes, Hanany--Witten isomorphism are useful \ref{subsection:SeparatedBraneDiagrams}

%Suppose now that ${\mathcal D}\leadsto_{\textup{HW}} \tilde{\mathcal D}$ and let $U_i$ be the blue line that is exchanged with an red line. We define the algebraic group automorphism 
%\[
%\rho_i:\mathbb T\rightarrow \mathbb T,\quad (t_1,\ldots,t_N,h)\mapsto (t_1,\ldots, t_{i-1},ht_i,t_{i+1},\ldots,t_N,h).
%\]
%Then, we have the following result (see~\cite[Proposition~7.1]{nakajima2017cherkis} and \cite[Theorem~3.9]{rimanyi2020bow}):
%\begin{prop} There exists an $\rho_i$-equivariant isomorphism of symplectic varieties $\mathcal C(\mathcal D)\rightarrow \mathcal C(\tilde{\mathcal D})$.
%\end{prop}
\section{Torus fixed points}
\label{section:torusFixedPoints}
The localization theorem in equivariant cohomology allows to describe the equivariant cohomology ring $H_{\mathbb T}^\ast(\mathcal C(\mathcal D))$ via the equivariant cohomology rings of all the fixed points and their interaction.  A good understanding of the set of torus fixed points is essential for concrete calculations related to stable envelopes.% and stable envelope bases.
 
Nakajima proved in \cite{nakajima2021geometric} that each bow variety $\mathcal C(\mathcal D)$ admits only finitely many $\mathbb A$-fixed points and the $\mathbb A$-fixed points can be classified 
\cite[Theorem A.5]{nakajima2021geometric} in terms of of certain combinatorial data attached to $\mathcal D$ which can be seen as variants of Maya diagrams. As $\mathcal C(\mathcal D)^{\mathbb A}$ is finite, it follows that $\mathcal C(\mathcal D)^{\mathbb A}=\mathcal C(\mathcal D)^{\mathbb T}$. 
In this section, we briefly recall this parametrization of $\mathcal C(\mathcal D)^{\mathbb T}$ using the language of Rim\'anyi and Shou from \cite{rimanyi2020bow}. Then we follow \cite{nakajima2021geometric} to prove the crucial Theorem~\ref{thm:fixedPointsGeneric1ParameterTorus}, which we call the 
\textit{Generic Co\-cha\-rac\-ter The\-o\-rem}: the $\mathbb T$-fixed locus of any bow variety coincides with the fixed locus corresponding to a generic one-parameter subgroup of $\mathbb A$. This result will later be used to set up the theory of of stable envelopes using attracting cells of torus fixed points.

\subsection{Tie and butterfly diagrams}\label{subsection:TiesAndButterflies}

We start by associating to a brane diagram $\mathcal D$ further combinatorial objects. Up to slight reformulations all this can be found in \cite{rimanyi2020bow}.

Let $Y_1,Y_2$ be lines in $\mathcal D$. We write $Y_1 \triangleleft Y_2$ if $Y_1$ is to the left of $Y_2$.
Assume we are given a pair $(Y_1,Y_2)$ of colored lines and a black line $X$ in $\mathcal D$ with $Y_1\triangleleft Y_2$. We say that the pair $(Y_1,Y_2)$ \textit{covers} $X$ if $Y_1\triangleleft X\triangleleft Y_2$.% is to the left of $X$ and $Y_2$ is to the right of $X$. 

%Tie diagrams were introduced by Rimanyi and Shou in~\cite[Section~4.1]{rimanyi2020bow}. Fix a brane diagram $\mathcal D$ and let 
%$(Y_1,Y_2)$ be a pair of 5-branes such that $Y_1$ is on the left of $Y_2$. Given a D3 brane $X$, we say that the pair $(Y_1,Y_2)$ \textit{covers} $X$ if $Y_1$ is to the left of $X$ and $Y_2$ is to the right of $X$.

\begin{definition}\label{definition:TieDiagrams}
A \textit{tie data with underlying brane diagram $\mathcal D$} is the data of $\mathcal D$ together with a set $D$ of pairs of colored lines of $\mathcal D$ such that the following holds:
\begin{itemize}
\item If $(Y_1,Y_2)\in D$ then $Y_1\triangleleft Y_2$.
\item If $(Y_1,Y_2)\in D$ then either $Y_1$ is blue and $Y_2$ is red, or $Y_1$ is red and $Y_2$ is blue.
\item For all black lines $X$ of $\mathcal D$, the number of pairs in $D$ covering $X$ is equal to $d_X$.
\end{itemize}
\end{definition}

Usually we fix a brane diagram $\mathcal D$ and denote by $D$ any tie data associated with it and call it just \textit{a tie data} without referring to $\mathcal D$. 

One can easily observe, that not every brane diagrams $\mathcal D$ admits a tie data.
% the following inequalities on the labels in $\mathcal D$ are necessary for the existence of a tie data:
%\begin{equation*}%\label{eq:TieDiagramAllowsHW}
%d_{X_i} \le d_{X_i-1}+d_{X_{i+1}} +1, \quad\textup{for all black lines $X_i$.}
%\end{equation*}

Given  $\mathcal D$, we visualize a tie data $D$ by attaching to the brane diagram $\mathcal D$ dotted curves  connecting a red line with a blue line according to the following algorithm. 

We consider all pairs $(Y_1,Y_2)\in D$ of one red and one blue line. 
\begin{itemize}
\item If $Y_1$ is blue and $Y_2$ is red, we draw a dotted connection \textit{below} the diagram $\mathcal D$. 
\item If  $Y_1$ is red and $Y_2$ is blue, we draw a dotted connection \textit{above} the diagram $\mathcal D$. 
\end{itemize}
The resulting diagram is called the \textit{tie diagram} of $D$ and the dotted curves are called \textit{ties}. Conversely a diagram with connections between red and blue lines, drawn from red to blue at the top and from blue to red at the bottom, is a visualization of a tie data, if each black line $X$ is covered from the top and bottom by a total number of $d_X$ arcs. 
%{\color{red}Habe diesen Satz reingemacht. Muss checken, ob er stimmt}

\begin{example}
Let for example $\mathcal D$ be the brane diagram
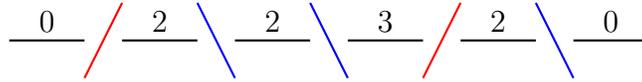
\[
\begin{tikzpicture}
[scale=.5]
\draw[thick] (0,0)--(2,0);
\draw[thick] (3,0)--(5,0);
\draw[thick] (6,0)--(8,0);
\draw[thick] (9,0)--(11,0);
\draw[thick] (12,0)--(14,0);
\draw[thick] (15,0)--(17,0);

\draw [thick,red] (2,-1) --(3,1);
\draw [thick,red] (11,-1) --(12,1);
\draw [thick,blue] (6,-1) --(5,1);
\draw [thick,blue] (9,-1) --(8,1);
\draw [thick,blue] (15,-1) --(14,1);

\node at (1,0.5){$0$};
\node at (4,0.5){$2$};
\node at (7,0.5){$2$};
\node at (10,0.5){$3$};
\node at (13,0.5){$2$};
\node at (16,0.5){$0$};
\end{tikzpicture}
\]
Then the pairs  $D=\{(V_2,U_1),(V_2,U_3),(U_1,V_1),(U_2,V_1),(V_1,U_3)\}$ give the data of a  tie diagram. Its visualization is given by  
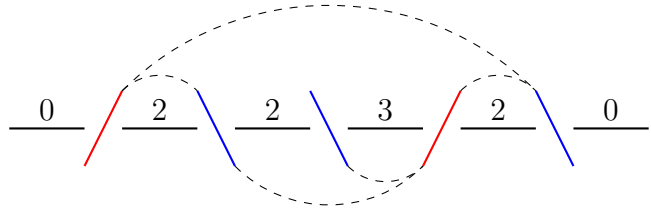
\[
\begin{tikzpicture}
[scale=.5]
\draw[thick] (0,0)--(2,0);
\draw[thick] (3,0)--(5,0);
\draw[thick] (6,0)--(8,0);
\draw[thick] (9,0)--(11,0);
\draw[thick] (12,0)--(14,0);
\draw[thick] (15,0)--(17,0);

\draw [thick,red] (2,-1) --(3,1);
\draw [thick,red] (11,-1) --(12,1);
\draw [thick,blue] (6,-1) --(5,1);
\draw [thick,blue] (9,-1) --(8,1);
\draw [thick,blue] (15,-1) --(14,1);

\draw[dashed] (3,1) to[out=45,in=135] (5,1);
\draw[dashed] (3,1) to[out=45,in=135] (14,1);
\draw[dashed] (6,-1) to[out=315,in=225] (11,-1);
\draw[dashed] (9,-1) to[out=315,in=225]  (11,-1);
\draw[dashed] (12,1) to[out=45,in=135] (14,1);

\node at (1,0.5){$0$};
\node at (4,0.5){$2$};
\node at (7,0.5){$2$};
\node at (10,0.5){$3$};
\node at (13,0.5){$2$};
\node at (16,0.5){$0$};
\end{tikzpicture}
\]
Note that for instance for the first label 2 the two curves run above the brane diagram, whereas for the second 2 one curve runs above and one below the diagram.
\end{example}
As tie data and their corresponding tie diagrams are in obvious one-to-one correspondence, we do not distinguish between them.

\subsection{Torus fixed points of bow varieties} Next we assign to each tie diagram $D$ a $\mathbb T$-fixed point $x_D\in\mathcal C(\mathcal D)$. By \cite{nakajima2021geometric}, the assignment $D\mapsto x_D$ then provides the following 

\begin{theorem}[Classification Theorem of $\mathbb T$-fixed points]\label{thm:torusFixedPointsOfBowVarieties}
There is a bijection
%\[
%\{\textit{Tie diagrams of $\mathcal D$}\} \xleftrightarrow{\phantom{x}1:1\phantom{x}}
%\mathcal C(\mathcal D)^{\mathbb T}.
%\]
\[
\begin{tikzcd}
\{\textit{Tie diagrams of $\mathcal D$}\}\arrow[r,leftrightarrow, "1:1"] &\mathcal C(\mathcal D)^{\mathbb T}.
\end{tikzcd}
\]
\end{theorem}
In particular, $\mathcal C(\mathcal D)$ admits a $\mathbb T$-fixed point if and only if we can extend $\mathcal D$ to a tie diagram. For the detailed combinatorics see \cite[Appendix]{nakajima2021geometric} and \cite[Theorem~4.8]{rimanyi2020bow}. 

Our next goal is to illustrate the construction of $x_D$. This requires some preparation.

\subsection{Butterfly diagrams}

Given a tie diagram $D$, we first assign to $D$ a family of colored graphs which are called butterfly diagrams. Based on the structure of these butterfly diagrams, we then define in the subsequent subsection the $\mathbb T$-fixed point $x_D$ in terms of matrices.

We first define the vertex set of the butterfly diagrams. Recall notation from Definition~\ref{uglyDef}.

\begin{definition}\label{defi:ButterflyVertices} Let $D$ be a tie diagram and $U$ a blue line in $\mathcal D$. Let $J\in\{1,\ldots, M+N\}$ such that $U^-=X_J$. The set $V(D,U)$ of \textit{butterfly vertices corresponding to $\mathcal D$ and $U$} is a finite subset of $\mathbb Z^2$  where a point $(j_1,j_2)\in \mathbb Z^2$ is contained in $V(D,U)$ if and only if the following conditions (i)--(iii) are satisfied
\begin{enumerate}[label=(\roman*)]
\item $2-J\le j_1\le M+N-J$, 
\item $c_{D,U,X_{j_1+J}}\le j_2 < c_{D,U,X_{j_1+J}}+d_{D,U,X_{j_1+J}}$,
\item $d_{D,U,X_{j_1+J}}\ne 0$.
\end{enumerate}
Here, $c_{D,U,X}$ and $d_{D,U,X}$ are integers, depending on a black line $X=X_j$,  defined as follows: 
\[
d_{D,U,X} :=
\begin{cases} |\{\textup{$V\in r(\mathcal D)$}\mid (V, U)\in D, V\triangleleft X\}| & \textup{if $X\;\triangleleft \;U$,}\\
 |\{\textup{$V\in r(\mathcal D)$}\mid (U,V)\in D, V\triangleright X\}| & \textup{if $X\triangleright U$}.
\end{cases} 
\]
For $X_j\triangleleft U$ we define $c_{D,U,X_j}$ recursively via $c_{D,U,X_J}=0$ and for $2\le j <J$ as
\[
c_{D, U,X_j}:= \begin{cases}
c_{D, U,X_{j+1}} &\textup{if $X_j^+$ is blue},\\
c_{D, U,X_{j+1}} &\textup{if $X_j^+$ is red and $d_{D, U,X_j}+1=d_{D, U,X_{j+1}}$},\\
c_{D, U,X_{j+1}}-1 &\textup{if $X_j^+$ is red and $d_{D, U,X_j}=d_{D,U,X_{j+1}}$}.
\end{cases}
\]
In case  $ X_j\triangleright U$, we set
\[
c_{D,U,X_j}:= \begin{cases} d_{D,U,X_{J+1}}- d_{D,U,X_j} +1 &\textup{if $d_{D,U,X_J}=0$,}\\
d_{D,U,X_{J+1}}- d_{D,U,X_j} &\textup{if $d_{D,U,X_J}\ne 0$.}
\end{cases}
\]
%
%Let $c_{D,U,X_j}:= d_{D,U,U^+}- d_{D,U,X_j}$ 
%is defined as
%\[
%in case  $ X_j\triangleright U$. 
%\color{red}{In the formulas is something wrong. Need to check! War beim naechsten Beispiel verwirrt und habe es geaendert}
We call the elements of $V(D,U)$ the \textit{butterfly vertices of $D$ and $U$}
and the integers $c_{D,U,X_j}$ the \textit{column bottom indices of $D$ and $U$}.
\end{definition}

\begin{example}\label{example:ButterflyVertices} Let $D$ be the following tie diagram:
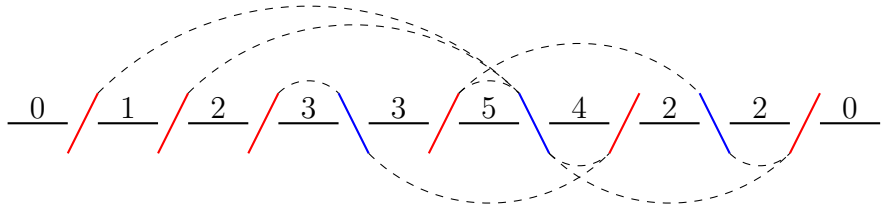
\[
\begin{tikzpicture}
[scale=.4]
\draw[thick] (-3,0)--(-1,0);
\draw[thick] (0,0)--(2,0);
\draw[thick] (3,0)--(5,0);
\draw[thick] (6,0)--(8,0);
\draw[thick] (9,0)--(11,0);
\draw[thick] (12,0)--(14,0);
\draw[thick] (15,0)--(17,0);
\draw[thick] (18,0)--(20,0);
\draw[thick] (21,0)--(23,0);
\draw[thick] (24,0)--(26,0);

\draw [thick,red] (-1,-1) -- (0,1);
\draw [thick,red] (2,-1) --(3,1);
\draw [thick,red] (5,-1) --(6,1);
\draw [thick,blue] (8,1) --(9,-1);
\draw [thick,red] (11,-1) --(12,1);
\draw [thick,blue] (14,1) --(15,-1);
\draw [thick,red] (17,-1) --(18,1);
\draw [thick,blue] (20,1) --(21,-1);
\draw [thick,red] (23,-1) --(24,1);

\node at (-2,0.5){$0$};
\node at (1,0.5){$1$};
\node at (4,0.5){$2$};
\node at (7,0.5){$3$};
\node at (10,0.5){$3$};
\node at (13,0.5){$5$};
\node at (16,0.5){$4$};
\node at (19,0.5){$2$};
\node at (22,0.5){$2$};
\node at (25,0.5){$0$};

\draw[dashed] (3,1) to[out=45,in=135] (14,1);
\draw[dashed] (6,1) to[out=45,in=135] (8,1);
\draw[dashed] (0,1) to[out=45,in=135] (14,1);
\draw[dashed] (12,1) to[out=45,in=135] (14,1);
\draw[dashed] (12,1) to[out=45,in=135] (20,1);

\draw[dashed] (9,-1) to[out=315,in=225] (17,-1);
\draw[dashed] (15,-1) to[out=315,in=225]  (17,-1);
\draw[dashed] (15,-1) to[out=315,in=225]  (23,-1);
\draw[dashed] (21,-1) to[out=315,in=225]  (23,-1);
\end{tikzpicture}
\]
 We pick $U=U_2$. In order to determine the integers $d_{D, U,X}$, we  remove all ties \textit{not} connected to $U_2$ and count for each black line $X$ the number of ties which cover $X$:
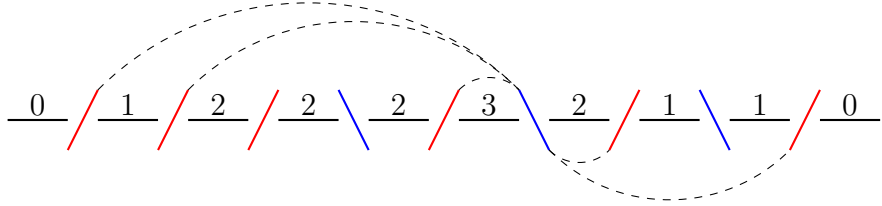
\[
\begin{tikzpicture}
[scale=.4]
\draw[thick] (-3,0)--(-1,0);
\draw[thick] (0,0)--(2,0);
\draw[thick] (3,0)--(5,0);
\draw[thick] (6,0)--(8,0);
\draw[thick] (9,0)--(11,0);
\draw[thick] (12,0)--(14,0);
\draw[thick] (15,0)--(17,0);
\draw[thick] (18,0)--(20,0);
\draw[thick] (21,0)--(23,0);
\draw[thick] (24,0)--(26,0);

\draw [thick,red] (-1,-1) -- (0,1);
\draw [thick,red] (2,-1) --(3,1);
\draw [thick,red] (5,-1) --(6,1);
\draw [thick,blue] (8,1) --(9,-1);
\draw [thick,red] (11,-1) --(12,1);
\draw [thick,blue] (14,1) --(15,-1);
\draw [thick,red] (17,-1) --(18,1);
\draw [thick,blue] (20,1) --(21,-1);
\draw [thick,red] (23,-1) --(24,1);

\node at (-2,0.5){$0$};
\node at (1,0.5){$1$};
\node at (4,0.5){$2$};
\node at (7,0.5){$2$};
\node at (10,0.5){$2$};
\node at (13,0.5){$3$};
\node at (16,0.5){$2$};
\node at (19,0.5){$1$};
\node at (22,0.5){$1$};
\node at (25,0.5){$0$};

\draw[dashed] (3,1) to[out=45,in=135] (14,1);
\draw[dashed] (0,1) to[out=45,in=135] (14,1);
\draw[dashed] (12,1) to[out=45,in=135] (14,1);

\draw[dashed] (15,-1) to[out=315,in=225]  (17,-1);
\draw[dashed] (15,-1) to[out=315,in=225]  (23,-1);
\end{tikzpicture}
\]
The resulting numbers $d_{D,U,X_j}$ are the new labels. We denote the underlying brane diagram by $\mathcal D_{D,U}$. The column bottom indices $c_{D,U,X_j}$ are %recorded in the following table:
\begin{center}
\begin{tabular}{| c | c | c | c | c | c | c | c | c |}
  \hline			
  $j$ & $2$ & $3$ & $4$ & $5$ & $6$ & $7$ & $8$ & $9$ \\ \hline
  $c_{D, U,X_j}$ & $-1$ & $-1$ & $0$ & $0$ & $0$ & $0$ & $1$ & $1$ \\
  \hline  
\end{tabular}
\vspace{2mm}
\end{center}
Following Definition~\ref{defi:ButterflyVertices}, we draw the elements of $V(D,U)$ as dots into the coordinate plane. For better illustration, we draw the coordinate plane below the brane diagram $\mathcal D_{D,U}$.
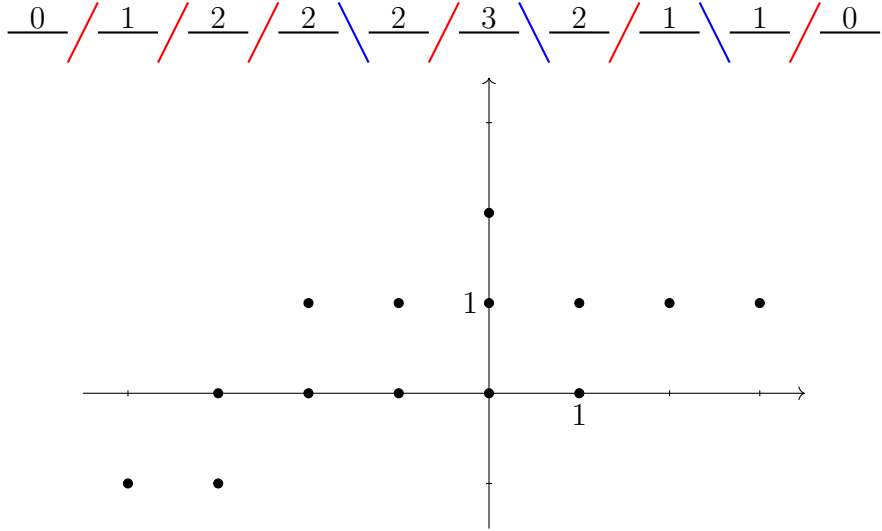
\[
\begin{tikzpicture}
[scale=.4]
\draw[thick] (-3,0)--(-1,0);
\draw[thick] (0,0)--(2,0);
\draw[thick] (3,0)--(5,0);
\draw[thick] (6,0)--(8,0);
\draw[thick] (9,0)--(11,0);
\draw[thick] (12,0)--(14,0);
\draw[thick] (15,0)--(17,0);
\draw[thick] (18,0)--(20,0);
\draw[thick] (21,0)--(23,0);
\draw[thick] (24,0)--(26,0);

\draw [thick,red] (-1,-1) -- (0,1);
\draw [thick,red] (2,-1) --(3,1);
\draw [thick,red] (5,-1) --(6,1);
\draw [thick,blue] (8,1) --(9,-1);
\draw [thick,red] (11,-1) --(12,1);
\draw [thick,blue] (14,1) --(15,-1);
\draw [thick,red] (17,-1) --(18,1);
\draw [thick,blue] (20,1) --(21,-1);
\draw [thick,red] (23,-1) --(24,1);

\node at (-2,0.5){$0$};
\node at (1,0.5){$1$};
\node at (4,0.5){$2$};
\node at (7,0.5){$2$};
\node at (10,0.5){$2$};
\node at (13,0.5){$3$};
\node at (16,0.5){$2$};
\node at (19,0.5){$1$};
\node at (22,0.5){$1$};
\node at (25,0.5){$0$};

%\draw[fill] (14.5,-3) circle [radius=.15];
\draw[fill] (13,-6) circle [radius=.15];
\draw[fill] (13,-9) circle [radius=.15];
\draw[fill] (13,-12) circle [radius=.15];

\draw[fill] (10,-9) circle [radius=.15];
\draw[fill] (10,-12) circle [radius=.15];

\draw[fill] (7,-9) circle [radius=.15];
\draw[fill] (7,-12) circle [radius=.15];

\draw[fill] (4,-15) circle [radius=.15];
\draw[fill] (4,-12) circle [radius=.15];

\draw[fill] (1,-15) circle [radius=.15];

\draw[fill] (16,-9) circle [radius=.15];
\draw[fill] (16,-12) circle [radius=.15];

\draw[fill] (19,-9) circle [radius=.15];
\draw[fill] (22,-9) circle [radius=.15];

\draw[->] (-0.5, -12) -- (23.5, -12);
  \draw[->] (13, -16.5) -- (13, -1.5);
  \node[anchor=north] at (16,-12) {$1$};
  \draw (13,-11.9) -- (13,-12.1);
  \draw (16,-11.9) -- (16,-12.1);
  \draw (19,-11.9) -- (19,-12.1);
  \draw (22,-11.9) -- (22,-12.1);
  \draw (7,-11.9) -- (7,-12.1);
  \draw (4,-11.9) -- (4,-12.1);
  \draw (1,-11.9) -- (1,-12.1);
  
  \draw (12.9,-15) -- (13.1,-15);
  \draw (12.9,-9) -- (13.1,-9);
  \draw (12.9,-6) -- (13.1,-6);
  \draw (12.9,-3) -- (13.1,-3);
  
  \node[anchor=east] at (13,-9) {$1$};

\end{tikzpicture}
\]
\end{example}
Let still $D$ be a tie diagram and $U$ a fixed blue line in  $\mathcal D$.
\begin{definition} A \textit{butterfly diagram for $( D,U)$} is a finite, directed, colored graph with colors black, blue, red, violet and green with vertex set $V({ D},U)$.
\end{definition}

%Following~\cite[Section~4.3]{rimanyi2020bow},
We assign to each pair $(D,U)$ a butterfly diagram $\operatorname{b}(D,U)$. To encode the vertices in the diagram first define subsets of $V(D,U)$:
\begin{eqnarray*}
V_b^+=\{(i,j)\in V(D,U)\mid X_{i+J}^+ \in b(\mathcal D)\},&\quad& V_b^-=\{(i,j)\in V(D,U)\mid X_{i+J}^- \in b(\mathcal D)\},\\
V_r^+=\{(i,j)\in V(D,U)\mid X_{i+J}^+ \in r(\mathcal D)\},&\quad &V_r^-=\{(i,j)\in V(D,U)\mid X_{i+J}^- \in r(\mathcal D)\}.
\end{eqnarray*}
In addition, we set $V_b=V_b^+\cup V_b^-$ and  $V_r=V_r^+\cup V_r^-$.
The  colored arrows of $\operatorname{b}(D,U)$ are recorded in 
%\begin{figure}
\begin{center}
	\begin{table}[h]
\begin{tabular}{ | c | c  l|}
 \hline 
 Color & \multicolumn{2}{|c|}{Arrows of $\operatorname{b}(D,U)$} \\
 \hline  
 \hline
 black & $(i,j-1)\leftarrow (i,j)$ \quad & $(i,j),(i,j-1)\in V_b$ \\
 \hline
 \textcolor{blue}{blue} & $(i-1,j)\textcolor{blue}{\leftarrow} (i,j)$ & $(i,j)\in V_b^-$, $(i-1,j)\in V_b^+$  \\
 \hline
 \textcolor{red}{red} & $(i+1,j)\textcolor{red}{\leftarrow} (i,j)$ & $(i,j)\in V_r^+$, $(i+1,j)\in V_b^-$\\
 \hline
 \textcolor{violet}{violet} & $(i-1,j-1)\textcolor{violet}{\leftarrow} (i,j)$ & $(i,j)\in V_r^+$, $(i+1,j)\in V_b^-$\\
 \hline
 \multirow{2}{*}{\textcolor{green}{green}} & $(0,d_{D,U,U^-}) \textcolor{green}{\leftarrow} \ast$ & if $d_{D,U,U^-}\ne0$ \\
 \cline{2-3}
 & $\ast\textcolor{green}{\leftarrow} (1,d_{D,U,U^-}+1)$ & if $d_{D,U,U^-}<d_{D,U,U^+}$\\
 \hline
\end{tabular}
\caption{Arrows of the butterfly diagram $\operatorname{b}(D,U)$.}
    \label{table:TiesAndButterflies} 
\end{table}
\end{center}
\begin{example} The butterfly diagram $\operatorname{b}(D,U)$ for  $(D,U)$ as in Example~\ref{example:ButterflyVertices} is the following:
\[
\begin{tikzpicture}
[scale=.4]
\draw[thick] (-3,0)--(-1,0);
\draw[thick] (0,0)--(2,0);
\draw[thick] (3,0)--(5,0);
\draw[thick] (6,0)--(8,0);
\draw[thick] (9,0)--(11,0);
\draw[thick] (12,0)--(14,0);
\draw[thick] (15,0)--(17,0);
\draw[thick] (18,0)--(20,0);
\draw[thick] (21,0)--(23,0);
\draw[thick] (24,0)--(26,0);

\draw [thick,red] (-1,-1) -- (0,1);
\draw [thick,red] (2,-1) --(3,1);
\draw [thick,red] (5,-1) --(6,1);
\draw [thick,blue] (8,1) --(9,-1);
\draw [thick,red] (11,-1) --(12,1);
\draw [thick,blue] (14,1) --(15,-1);
\draw [thick,red] (17,-1) --(18,1);
\draw [thick,blue] (20,1) --(21,-1);
\draw [thick,red] (23,-1) --(24,1);

\node at (-2,0.5){$0$};
\node at (1,0.5){$1$};
\node at (4,0.5){$2$};
\node at (7,0.5){$2$};
\node at (10,0.5){$2$};
\node at (13,0.5){$3$};
\node at (16,0.5){$2$};
\node at (19,0.5){$1$};
\node at (22,0.5){$1$};
\node at (25,0.5){$0$};

\node at (14.5,-3){$\ast$};
\draw[fill] (13,-6) circle [radius=.15];
\draw[fill] (13,-9) circle [radius=.15];
\draw[fill] (13,-12) circle [radius=.15];

\draw[fill] (10,-9) circle [radius=.15];
\draw[fill] (10,-12) circle [radius=.15];

\draw[fill] (7,-9) circle [radius=.15];
\draw[fill] (7,-12) circle [radius=.15];

\draw[fill] (4,-15) circle [radius=.15];
\draw[fill] (4,-12) circle [radius=.15];

\draw[fill] (1,-15) circle [radius=.15];

\draw[fill] (16,-9) circle [radius=.15];
\draw[fill] (16,-12) circle [radius=.15];

\draw[fill] (19,-9) circle [radius=.15];
\draw[fill] (22,-9) circle [radius=.15];

\draw [blue, ->] (15.8, -9) -- (13.2,-9);
\draw [blue, ->] (15.8, -12) -- (13.2,-12);

\draw [blue, ->] (9.8, -9) -- (7.2,-9);
\draw [blue, ->] (9.8, -12) -- (7.2,-12);

\draw [blue, ->] (21.8, -9) -- (19.2,-9);

\draw [red, ->] (16.2, -9) -- (18.8,-9);

\draw [red, ->] (10.2, -9) -- (12.8,-9);
\draw [red, ->] (10.2, -12) -- (12.8,-12);

\draw [red, ->] (4.2, -12) -- (6.8,-12);

\draw [red, ->] (1.2, -15) -- (3.8,-15);

\draw [->] (13, -6.2) -- (13,-8.8);
\draw [->] (13, -9.2) -- (13,-11.8);
\draw [->] (16, -9.2) -- (16,-11.8);
\draw [->] (10, -9.2) -- (10,-11.8);
\draw [->] (7, -9.2) -- (7,-11.8);

\draw [violet, ->] (3.8, -12.2) -- (1.2,-14.8);
\draw [violet, ->] (6.8, -12.2) -- (4.2,-14.8);
\draw [violet, ->] (6.8, -9.2) -- (4.2,-11.8);

\draw [violet, ->] (12.8, -6.2) -- (10.2,-8.8);
\draw [violet, ->] (12.8, -9.2) -- (10.2,-11.8);

\draw[->, green] (14.5,-3.2) to [out=-70,in=120] (13,-5.8);

\end{tikzpicture}
\]
For further examples of butterfly diagrams see~\cite[Section~4.6]{rimanyi2020bow}.
\end{example}

\subsection{Explicit construction of $x_D$}\label{subsection:ExplicitConstructionTorusFixedPoints}

Let $D$ be a tie diagram. To construct the corresponding $\mathbb T$-fixed point $x_D\in\mathcal C(\mathcal D)$ we first assign to each butterfly diagram $\operatorname{b}(D,U)$ a family of linear operators as follows:
Let $F_{D,U}=\bigoplus_{i,j\in\mathbb Z}\mathbb Ce_{U,i,j}$ and let $\mathbb C_{D,U}=\mathbb C$. Assume $a$ is an arrow in $\operatorname{b}(D,U)$ which is not green and let $(i_1,j_1)$ be the source of $a$ and $(i_2,j_2)$ be the target of $a$. Then,
we assign to $a$ the vector space endomorphism
\[
\varphi_a:F_{D,U}\rightarrow F_{D,U},\quad \varphi_a(e_{U,i,j})=\begin{cases}
e_{U,i_2,j_2} &\textup{if }i=i_1,j=j_1,\\
0&\textup{else.}
\end{cases}
\]
By construction, $\operatorname{b}(D,U)$ admits at most one green arrow with starting in $\ast$ and at most one green arrow ending in $\ast$.
If $\operatorname{b}(D,U)$ admits a green arrow $a$ starting in $\ast$ and ending in $(i,j)$ we assign to $a$ the vector space homomorphism 
\[
\psi_a: \mathbb C_{D,U}\rightarrow F_{D,U},\quad \psi_a(1)=e_{U,i,j}.
\]
If $\operatorname{b}(D,U)$ admits a green arrow $b$ starting in $(i_1,j_1)$ and ending in $\ast$, we assign to $b$ the vector space homomorphism
\[
\psi_b':F_{D,U}\rightarrow \mathbb C_{D,U},\quad \psi_b'(e_{U,i,j})=\begin{cases}
1 &\textup{if }i=i_1,j=j_1,\\
0&\textup{else.}
\end{cases}
\]
The column indices of the butterfly vertices define finite dimensional subspaces of $F_{D,U}$:
\[
F_{D,U,X_i}:=\mathrm{span}_{\mathbb C}(e_{U,i-J,j}\mid (i-J,j)\in V(D,U)),\quad\textup{for all $X_i\in h(\mathcal D)$.}
\]
Let $U'$ be a blue line of $\mathcal D$ and $J'\in\{1,\ldots,M+N\}$ such that $X_{J'}=(U')^-$. Using the colored arrows of $\operatorname{b}(D,U)$, we define linear operators
\[
A_{D,U,U'}\in\operatorname{Hom}(F_{D,U,X_{J'+1}},F_{D,U,X_{J'}}),\quad B_{D,U,U'}^+\in\operatorname{End}(F_{D,U,X_{J'+1}}),\quad B_{D,U,U'}^-\in\operatorname{End}(F_{D,U,X_{J'}}) 
\]
as
\begin{align*}
A_{D,U,U'}(e_{U,J'-J+1,j}) &= \sum_{\mathclap{a\in \textcolor{blue}{\operatorname{blue}}(D,U,J')}}\varphi_a(e_{U,J'-J+1,j}),\\
B_{D,U,U'}^+(e_{U,J'+1-J,j}) &= \sum_{\mathclap{a\in \textcolor{black}{\operatorname{black}}(D,U,J'+1)}}(-1)\varphi_a(e_{U,J'+1-J,j}),\\
B_{D,U,U'}^-(e_{U,J'-J,j}) &= \sum_{\mathclap{a\in \textcolor{black}{\operatorname{black}}(D,U,J')}}(-1)\varphi_a(e_{U,J'-J,j}).
\end{align*}
Here and in the following, for any color $c$, we denote by $\operatorname{c}(D,U,j)$ the set of arrows colored $c$ in $\operatorname{b}(D,U)$ with first coordinate of the target equal to $j$.

%$\textcolor{blue}{\operatorname{blue}}(D,U,J')$ is the set of blue arrows in $\operatorname{b}(D,U)$ with first coordinate of the target equal to $J'$. Likewise, $\textcolor{black}{\operatorname{black}}(D,U,J'+1)$ respectively $\textcolor{black}{\operatorname{black}}(D,U,J')$ is the set of black arrows in $\operatorname{b}(D,U)$ with first coordinate of the target equal to $J'+1$ resp. $J'$.

Next, we analogously construct linear operators for each red line. Given a red line $V$ in $\mathcal D$ and $I\in\{1,\ldots,M+N\}$ such that $X_I=V^-$, we define linear operators:
\begin{equation*}
%\begin{array}[t]{lcrcl}
C_{D,U,V}\in\operatorname{Hom}(F_{D,U,X_{I+1}},F_{D,U,X_I})\quad\textup{and}\quad
D_{D,U,V}\in \operatorname{Hom}(F_{D,U,X_I},F_{D,U,X_{I+1}})
\end{equation*}
via the formulas
\[
C_{D,U,V}(e_{U,I-J+1,j}) =\sum_{\mathclap{a\in \textcolor{violet}{\operatorname{violet}}(D,U,I-J)}}\varphi_a(e_{U,I-J+1,j})
,\quad
D_{D,U,V}(e_{U,I-J,j})=
\sum_{\mathclap{a\in \textcolor{red}{\operatorname{red}}(D,U,I-J+1)}}\varphi_a(e_{U,I-J,j}).
\]
%\begin{align*}
%C_{D,U,V}(e_{U,I-J+1,j})&= \sum_{a\in \textcolor{violet}{\operatorname{violet}}(D,U,I-J)}\varphi_a(e_{U,I-J+1,j}),\\
%D_{D,U,V}(e_{U,I-J,j}) &= \sum_{a\in \textcolor{red}{\operatorname{red}}(D,U,I-J+1)}\varphi_a(e_{U,I-J,j}).
%\end{align*}
Finally, we also define homomorphisms
\begin{equation}\label{eq:DefinitionaUbU}
%\begin{array}[t]{lclcrcl}
a_{D,U}\in \operatorname{Hom}(\mathbb C_U,F_{D,U,U^-})\quad\text{and}\quad b_{D,U}\in \operatorname{Hom}(F_{D,U,U^+},\mathbb C_U),
\end{equation}
\begin{align*}
a_{D,U}(1) &= \begin{cases}
\psi_a(1) &\textup{if $\textcolor{green}{\operatorname{green}}^{\mathrm{out}}(D,U)=\{a\}$}, \\
0 &\textup{if $\textcolor{green}{\operatorname{green}}^{\mathrm{out}}(D,U)=\emptyset $},
\end{cases} \\
b_{D,U}(e_{U,i,j})&= \begin{cases}
\psi_b'(e_{U,i,j}) &\textup{if $\textcolor{green}{\operatorname{green}}^{\mathrm{in}}(D,U)=\{b\}$}, \\
0 &\textup{if $\textcolor{green}{\operatorname{green}}^{\mathrm{in}}(D,U)=\emptyset $},
\end{cases}
\end{align*} 
%\begin{eqnarray*}
%a_{D,U}(1) = \begin{cases}
%\psi_a(1) &\textup{if $\textcolor{green}{\operatorname{green}}^{\mathrm{out}}(D,U)=\{a\}$}, \\
%0 &\textup{if $\textcolor{green}{\operatorname{green}}^{\mathrm{out}}(D,U)=\emptyset $},
%\end{cases} 
%&&
%b_{D,U}(e_{U,i,j})= \begin{cases}
%\psi_b'(e_{U,i,j}) &\textup{if $\textcolor{green}{\operatorname{green}}^{\mathrm{in}}(D,U)=\{b\}$}, \\
%0 &\textup{if $\textcolor{green}{\operatorname{green}}^{\mathrm{in}}(D,U)=\emptyset $},
%\end{cases} 
%\end{eqnarray*}
where $\textcolor{green}{\operatorname{green}}^{\mathrm{in}}(D,U)$ and  $\textcolor{green}{\operatorname{green}}^{\mathrm{out}}(D,U)$ are the sets of green arrows starting respectively ending in the additional vertex $\ast$. Write $F_X=\bigoplus_{U\in b(\mathcal D)}F_{D,U,X}$ for each black line $X$. 

Combining the above pieces, we now define the point $x_D$:
 \begin{definition}[Fixed points]
% \textcolor{red}{Check asumptions}
 Let $\mathcal D$ be a brane diagram and $D$ a tie diagram of $\mathcal D$. 
 Set
\[
x_D:=[(A_{D,U},B_{D,U}^+,B_{D,U}^-,a_{D,U},b_{D,U})_U,(C_{D,V},D_{D,V})_V]\in\mathcal C(\mathcal D),
\]
where 
\begin{eqnarray*}
%\begin{align*}
&A_{D,U}= \bigoplus_{U'\in b(\mathcal D)} A_{D,U',U}, \quad
B_{D,U}^+= \bigoplus_{U'\in b(\mathcal D)}B_{D,U',U}^+,  \quad
B_{D,U}^- = \bigoplus_{U'\in b(\mathcal D)}B_{D,U',U}^-,  &\\
&C_{D,V} = \bigoplus_{U'\in b(\mathcal D)} C_{D,U',V}, \quad
D_{D,V}= \bigoplus_{U'\in b(\mathcal D)} D_{D,U',V} &
%\end{align*}
\end{eqnarray*}
and $a_U,b_U$ are defined as in \eqref{eq:DefinitionaUbU}.
\end{definition}

Using the explicit construction of the operators from above one can directly show that
\[
y_D:=((A_{D,U},B_{D,U}^+,B_{D,U}^-,a_{D,U},b_{D,U})_U,(C_{D,V},D_{D,V})_V)
\]
is contained in $m^{-1}(0)^{\operatorname{s}}$. In addition, given $t=(t_U)_U\in\mathbb A$, we have
$
(t_U)_U.y_D = g_t.y_D
$,
where $g_t=(g_{t,X})_X\in\mathcal G$ and $g_{t,X}$ acts on each $F_{D,U,X}$ via scalar multiplication by $t_U$. Likewise, for $h\in\mathbb C^\ast_h$, we have
$
h.y_D = g_h.y_D,
$
where $g_h=(g_{h,X})_X\in\mathcal G$ and $g_{h,X}$ acts on each $F_{D,U,X}$ via
$g_h(e_{U,i,j})=h^{j}e_{U,i,j}$. Thus, $x_D$ is indeed a $\mathbb T$-fixed point of $\mathcal C(\mathcal D)$.

\subsection{Fixed point matching under Hanany\textbf{--}Witten transition} By Proposition~\ref{prop:HWTransition}, Hanany--Witten isomorphisms induce bijections in torus fixed point sets. This bijection can explicitly be characterized as follows: let $U_i$ be a blue line in $\mathcal D$ and suppose $V_j$ is a red line directly to the right of $U_i$. Let $\HW$ using a Hanany--Witten transition involving $U_i$ and $V_j$ and let $\Psi:\mathcal C(\mathcal D)\xrightarrow\sim \mathcal C(\mathcal D')$ be the corresponding Hanany--Witten isomorphism.
\begin{definition}[Combinatorial HW-transforms]
We define an isomorphism of sets
\[
\psi:\{\textup{Tie diagrams of $\mathcal D$}\} \xrightarrow{\phantom{x}\sim\phantom{x}} \{\textup{Tie diagrams of $\mathcal D'$}\}
\]
via
\[
\psi(D)= \begin{cases}
D\setminus\{(U_i,V_j)\}&\textup{if $(U_i,V_j)\in D$},\\
D\cup \{(U_i,V_j)\}&\textup{if $(U_i,V_j)\notin D$}.
\end{cases}
\]
\end{definition}
The following proposition from {\cite[Section~4.7]{rimanyi2020bow}} relates this to Theorem~\ref{thm:torusFixedPointsOfBowVarieties}:

\begin{prop}[Fixed point matching]\label{prop:HWFixedPointMatching}
The following diagram commutes:
\[
\begin{tikzcd}
\{\textup{Tie diagrams of $\mathcal D$}\} \arrow[r, "\sim"]\arrow[d,"\psi"]& \mathcal C(\mathcal D)^{\mathbb T}
\arrow[d,"\Psi"]\\
\{\textup{Tie diagrams of $\mathcal D'$}\}\arrow[r, "\sim"] & \mathcal C(\mathcal D')^{\mathbb T}
\end{tikzcd}
\]
Here, the horizontal maps are the bijections from the Classification Theorem of $\mathbb T$-fixed points.
\end{prop}

Consequently,  \textit{the tie diagram $\psi(D)$ is obtained from $D$ via a Hanany--Witten transition.}

%We close this subsection with describing the $\mathbb T$-fixed point matching under Hanany--Witten isomorphisms. For this, let $U_i$ be D5 and $V_j$ NS5 such that $U_i$ is to the left of $V_j$. Let $\mathcal D'$ be the brane diagram obtained from $\mathcal D$ by Hanany--Witten transition between $U_i$ and $V_j$ and let $\Psi:\mathcal C(\mathcal D)\rightarrow \mathcal C(\mathcal D')$ be the corresponding Hanany--Witten isomorphism. We have that $\Psi$ induces a bijection of the $\mathbb T$-fixed points of $\mathcal C(\mathcal D)$ and $\mathcal C(\mathcal D')$.

%Given a tie diagram $D$ of $\mathcal D$, we construct a tie diagram of $\mathcal D'$ as follows
%\[
%D'=\begin{cases}
%D\setminus\{(U_i,V_j)\}&\textup{if $(U_i,V_j)\in D$},\\
%D\cup \{(U_i,V_j)\}&\textup{if $(U_i,V_j)\notin D$}.
%\end{cases}
%\]
%Consequently, say that $D'$ is obtained from $D$ by Hanany--Witten transition.
%Using the explicit construction of the Hanany--Witten isomorphism $\Psi$ one can check that $\Psi$ maps the $\mathbb T$-fixed point $x_D$ to $x_{D'}$.
%For more details, see~\cite[Section~4.7]{rimanyi2020bow}.

\subsection{Separated brane diagrams}\label{subsection:SeparatedBraneDiagrams}

In some proofs appearing later, we will use a reduction to certain nice brain diagrams which we introduce now. 

\begin{definition}
For a given brane diagram $\mathcal D$ the \textit{separation degree of $\mathcal D$} is defined as
\[
\operatorname{sdeg}(\mathcal D):=|\{(U,V)\in b(\mathcal D)\times r(\mathcal D)| U\triangleleft V\}|.
\]
We call $\mathcal D$ \textit{separated} if $\operatorname{sdeg}(\mathcal D)=0$, i.e. all red lines are in $\mathcal D$ to the left of all blue lines.
\end{definition}

%A brane diagram $\mathcal D$ is called \textit{separated} if all $M$ red lines are to the left of all $N$ blue lines, in formulas.

We can deduce now that each bow variety is isomorphic to a bow variety corresponding to a separated brane diagram:

\begin{prop}[Reduction argument]\label{prop:separatedBowVarieties} %Given a brane diagram $\mathcal D$,
There exists a separated brane diagram $\tilde{\mathcal D}$ such that $\HW$, i.e. they are Hanany--Witten equivalent. 
% There exists a separated brane diagram $\tilde{\mathcal D}$ such that $\mathcal C(\mathcal D)\xrightarrow\sim \mathcal C(\tilde{\mathcal D})$. The isomorphism can be chosen to be a composition of Hanany--Witten isomorphisms.
\end{prop}
\begin{proof} 
Suppose $\operatorname{sdeg}(\mathcal D)>0$. Then, there exist $U\in b(\mathcal D),V\in r(\mathcal D)$ such that $U$ is directly to the left of $V$. Since $\mathcal D$ is admissible, we can apply a Hanany--Witten transition reducing the separatedness degree by $1$. Now just repeat this argument.
%We define \[
%r(\mathcal D):=|\{(i,j)| \textup{$U_i$ left of $V_j$, $1\le i\le N$, $1\le j\le M$}\}|
%\]
%and prove the proposition by induction on $r(\mathcal D)$. The condition $r(\mathcal D)=0$ is equivalent to $\mathcal D$ being separated. For the induction step, suppose $r(\mathcal D)>0$. Thus, there exist  $U_i$ D5 and $V_j$ NS5 such that $U_i$ is to the left of $V_j$. Since $\mathcal D$ admits a tie diagram, we can apply by \eqref{eq:TieDiagramAllowsHW} a Hanany--Witten transition exchanging $U_i$ and $V_j$. The resulting brane diagram $\mathcal D'$ satisfies $r(\mathcal D')=r(\mathcal D)-1$. Hence, applying the induction hypothesis to $\mathcal D'$ proves the proposition.
\end{proof}
%%%%%%%%%%%%%%%%%%%%Ungleichungen geloescht, da wir sie nicht brauchen
%Note that for $\mathcal D$ separated, the condition that $\mathcal D$ admits a tie diagram implies the inequalities
%\[
%d_{V_M^-}\le d_{V_{M-1}^-} \le \ldots \le d_{V_1^-}\le  d_{V_1^+},\quad d_{U_1^-}\geq d_{U_1^+}\geq d_{U_2^+} \geq \ldots \geq d_{U_N^+}.
%\]
For $\mathcal D$ separated the operators defining points of $\mathcal C(\mathcal D)$ satisfy the following nilpotency conditions, recall the definition of the moment map $m$ from \eqref{m}:

\begin{prop}[Nilpotency]\label{prop:nilpotency}
Suppose $\mathcal D$ is separated and let \[((A_U,B_U^-,B_U^+,a_U,b_U)_{U},(C_V,D_V)_{V})\in m^{-1}(0).\] Then, the following holds:
\begin{enumerate}[label=(\roman*)]
\item\label{item:nilpotencyNS5} We have $(C_{V_j}D_{V_j})^{M-j}=0$ and $(D_{V_j}C_{V_j})^{M-j+1}=0$ for $j=1,\ldots,M-1$.
\item\label{item:nilpotencyD5} We have $(B_{U_1}^-)^{M}=0$.
\end{enumerate}
\end{prop}

\begin{proof} We prove~\ref{item:nilpotencyNS5} by induction on $j$.
By the moment map equation, we have the equalities
%The case $j=M$ is immediate from $C_{V_M}=0,D_{V_M}=0$ and the moment map equation. The induction step follows also from the moment map equation:
\begin{align*}
(C_{V_j}D_{V_j})^{M-j} &= C_{V_{j}}(C_{V_{j+1}}D_{V_{j+1}})^{M-j-1}D_{V_{j}}, \\
(D_{V_j}C_{V_j})^{M-j+1} &= D_{V_{j}}(D_{V_{j+1}}C_{V_{j+1}})^{M-j}D_{V_{j}}.
\end{align*}
Thus,~\ref{item:nilpotencyNS5} follows directly from $C_{V_M}=0,D_{V_M}=0$ via induction.
%We show by induction $(D_{V_j}C_{V_j})^j=0$.
%By definition, $C_{V_1}=0$ and $D_{V_1}=0$. We show by induction that $(C_{V_j}D_{V_j})^{j-1}=0$ and $(D_{V_j}C_{V_j})^j=0$
% Clearly, $C_{V_1}D_{V_1}=0$ and $D_{V_1}C_{V_1}=0$. We show by induction that
%$C_{V_1}D_{V_1}=0$ and $D_{V_1}C_{V_1}=0$
%
% Assertion \ref{item:nilpotencyNS5} follows from
%\[
%C_{V_i}D_{V_i}=D_{V_{i+1}}C_{V_{i+1}},\quad \textup{for }i=1,\ldots,M-1
%\]
%and $C_{V_1}=0$ and $D_{V_1}=0$. 
The assertion~\ref{item:nilpotencyD5} follows from~\ref{item:nilpotencyNS5} since $B_{U_1}^-=-C_{V_1}D_{V_1}$.
\end{proof}

\subsection{The Generic Cocharacter Theorem}
To formulate the Cocharacter Theorem, let 
\[
\sigma: \mathbb C^\ast\rightarrow \mathbb A,\quad t\mapsto (\sigma_U(t))_{U},
\]
be a cocharacter. We call $\sigma$ \textit{generic} if $\sigma_U\ne\sigma_{U^\prime}$ for all $U,U'\in b(\mathcal D)$.
In addition, we set
\[
\mathcal C(\mathcal D)^\sigma :=\{x\in\mathcal C(\mathcal D)\mid \sigma(t).x=x\textup{ for all }t\in\mathbb C^\ast\}.
\]
The Generic Cocharacter Theorem states as follows:

\begin{theorem}[Generic Cocharacter Theorem]\label{thm:fixedPointsGeneric1ParameterTorus} Let $\sigma$ be generic. Then, $\mathcal C(\mathcal D)^\sigma = \mathcal C(\mathcal D)^{\mathbb T}$.
\end{theorem}

\begin{remark}
An analogue of Theorem~\ref{thm:fixedPointsGeneric1ParameterTorus}  in the affine setting is discussed in  \cite[Section 4]{nakajima2021geometric} in case of balanced brane diagrams. For simplicity we restrict ourselves to generic cocharacters. More general cocharacters can be dealt with as in  \cite[Theorem 4.14.]{nakajima2021geometric}. 

\end{remark}

\begin{remark}
By slightly adapting the statement of \cite[Theorem~3.1]{botta2023mirror} and its proof it is in fact possible to obtain the Generic Cocharacter Theorem. The main step hereby is a decomposition,  with respect to the action of subtori of $\mathbb A$, of the fixed point set of $\mathcal C(\mathcal D)$ into a product of smaller bow varieties. This product then has to be realised as a set of isolated points by carefully following the underlying combinatorics. Our argument packages these two steps into the action of a one dimensional torus via a generic cocharacter.
\end{remark}

The upcoming four subsections are devoted to the proof of Theorem~\ref{thm:fixedPointsGeneric1ParameterTorus}. The proof is based on a diagrammatic study of the weight spaces of the fiber of the full tautological bundle at torus fixed points. Before we go into the details, we prove a simple consequence of the Generic Cocharacter Theorem about tangent weights of bow varieties:

\begin{cor}[Tangent weights]\label{cor:tangentweights} Let $p\in\mathcal C(\mathcal D)^{\mathbb T}$ and $\tau$ be a $\mathbb T$-weight of $T_p\mathcal C(\mathcal D)$. Then, there exist $i,j\in \{1,\ldots, N\}$ with $i\ne j$ and $m\in\mathbb Z$ such that $\tau=t_i-t_j+mh$.
\end{cor}

\begin{proof} According to~\cite[Section~3.2]{rimanyi2020bow}, all $\mathbb T$-weights of $T_p\mathcal C(\mathcal D)$ are of the form $t_i-t_j+mh$, with $i,j\in \{1,\ldots, N\}$ and $m\in\mathbb Z$. By Theorem~\ref{thm:fixedPointsGeneric1ParameterTorus}, $p$ is an isolated $\mathbb A$-fixed point. Since $\mathcal C(\mathcal D)$ is smooth, the equivariant slice theorem, see e.g. \cite[Theorem~I.1.2]{audin2004torus}, implies that no $\mathbb A$-weight of $T_p\mathcal C(\mathcal D)$ is trivial. Thus, no $\mathbb T$-weight of $T_p\mathcal C(\mathcal D)$ is of the form $mh$ for $m\in\mathbb Z$ which proves the corollary.
\end{proof}

\subsection{Weight spaces}
Let $\sigma$ be a generic cocharacter of $\mathbb A$. By Proposition~\ref{prop:separatedBowVarieties}, it suffices to prove Theorem~\ref{thm:fixedPointsGeneric1ParameterTorus} in the case where $\mathcal D$ is separated. As we already saw in Proposition~\ref{prop:nilpotency}, the operators describing points of bow varieties associated to separated brane diagrams satisfy useful nilpotency conditions. Hence,  we assume from now on until the end of this section that $\mathcal D$ is separated.

The cocharacter $\sigma$ induces a $\mathbb C^\ast$-action on the full tautological bundle $\xi$ on $\mathcal C(\mathcal D)$, see Definition~\ref{definition:TautologicalBundles}. Let \[p=[(A_U,B_U^-,B_U^+,a_U,b_U)_{U},(C_V,D_V)_V]\in\mathcal C(\mathcal D)^\sigma,\] and consider the fiber $\xi_p$ which we identify with $W:=\bigoplus_{X\in h(\mathcal D)} W_X $. This induces
\begin{equation} \label{eq:torusactionD3}
\rho:\mathbb C^\ast \rightarrow \mathcal G,\quad t\mapsto (\rho_X(t))_X
\end{equation}
and thus a $\mathbb C^\ast$-action on $W$ which satisfies the following \textit{action identity} in $\widetilde{\mathcal M}(\mathcal D)$ for all $t\in\mathbb C^\ast$:
\begin{equation*}
\rho(t).((A_U,B_{U}^-,B_{U}^+,a_U,b_U)_{U},(C_V,D_V)_{V})
= ((A_U,B_{U}^-,B_{U}^+,a_U\sigma_U(t)^{-1},\sigma_U(t)b_U)_{U},(C_V,D_V)_{V}).
\end{equation*}
 %The above equality is in $\widetilde{\mathcal M}(\mathcal D)$.
Given a character $\tau:\mathbb C^\ast\rightarrow\mathbb C^\ast$ and  a black line $X$ in $\mathcal D$, let $W_\tau$ and  $W_{\tau,X}$ be the corresponding weight space of $W$ and  $W_{X}$ respectively. Then the finite-dimensionality of $W$ implies
\[
 W=\bigoplus_{\tau} W_\tau =\bigoplus_{\tau}\bigoplus_{X\in h(\mathcal D)} W_{\tau,X}
\]
and these weight spaces satisfy the following invariance property:
\begin{lemma}[Invariance property]\label{lemma:abcdinvariance} Let $ W_\tau\subset  W$ be a $\mathbb C^\ast$-weight space. Then, $ W_\tau$ is invariant under all operators $A_U,B_U^-,B_U^+,C_V,D_V$.
\end{lemma}
\begin{proof} We only show that $ W_\tau$ is $A_U$-invariant, since the proof for the other operators is analogous. Using the action identity formulated after~\eqref{eq:torusactionD3}, we get
$\rho_{U^-}(t)A_U(w)=A_U\rho_{U^+}(t)(w)=\tau(t)A_U(w)$, for $t\in \mathbb C^\ast,w\in  W_{\tau,U^+}$. Hence, $W_\tau$ is $A_U$-invariant.
\end{proof}

We now study the nontrivial weight spaces of $W$ and provide a diagrammatic description of the actions of the operators $A_U,B_U^-,B_U^+,C_V,D_V$ on them.

\begin{prop}[Weight spaces]\label{prop:weightab} Given $U\in b(\mathcal D)$, the following holds:
\begin{enumerate}[label=(\roman*)] 
\item\label{item:aweight} We have  $\operatorname{im}(a_U)\subset W_{\sigma_U,U^{-}}$.
\item\label{item:bweight} We have $\bigoplus_{\tau\ne\sigma_U}{W}_{\tau,U^+}\subset \operatorname{ker}(b_U)$.
\item\label{item:weightisos} The operator $A_U$ induces a $\mathbb C$-linear isomorphism 
$ W_{\tau,U^+}\xrightarrow\sim  W_{\tau,U^-}$ for all $\tau\ne \sigma_U$.
\end{enumerate}
\end{prop}
\begin{proof}
By the action identity formulated after~\eqref{eq:torusactionD3}, 
$\rho(t)_{U^+}a_U( \sigma_U(t)^{-1}1)=a_U(1)$ and
\[
\sigma_U(t) b_U(\rho(t)^{-1}_{U^-}w)=b_U(w), \quad \textup{for all $w\in W_{U^+}$, $t\in\mathbb C^\ast$}
\]
which implies \ref{item:aweight},\ref{item:bweight}. %We continue with proving \ref{item:weightisos}.
 By \ref{item:aweight} (or \ref{item:bweight}), we now know that $a_Ub_U$ vanishes on $W_{\tau,U^-}$. Hence, \eqref{eq:nonMomentMap} gives
\begin{equation}\label{eq:weightisoproofmoment}
B_U^-A_U(w)=A_UB_U^+(w),\quad\textup{for all }w\in {W}_{\tau,U^+}.
\end{equation}
In particular, $\operatorname{ker}(A_{U\mid{W}_{\tau,U^-}})$ is $B_U^{+}$-invariant and thus, $\operatorname{ker}(A_{U\mid W_{\tau,U^-}})=0$ by~\ref{item:S1}.
Next, we show that $A_{U\mid W_{\tau, U^+}}$ surjects onto $ W_{\tau, U^-}$. By~\eqref{eq:weightisoproofmoment}, $\operatorname{im}(A_{U\mid W_{\tau, U^-}})$ is stable under the $B_U^-$-action. By~Lemma~\ref{lemma:abcdinvariance} and~\ref{item:aweight} the subspace
\[
\operatorname{im}(A_{U\mid W_{\tau, U^-}}) \oplus \bigoplus_{\nu\ne\tau}  W_{\nu,U^+} \subset  W_{U^+}
\]
satisfies~\ref{item:S2} and thus equals $W_{U^+}$ which proves~\ref{item:weightisos}.
%
%. Hence, the above inclusion is an equality which implies $\operatorname{im}(A_{U\mid W_{\tau, U^-}}) =  W_{\tau,U^+}$. This completes the proof of \ref{item:weightisos}. 
\end{proof}
%As a consequence, we get that the only possible non-zero weight spaces are the following:

\begin{cor} We have
 $ W = \bigoplus_{U\in b(\mathcal D)} W_{\sigma_U}$.
%\item $\textbf W_{\sigma_U}=0$ if and only if $\operatorname{im}(a_U)\ne 0$.
\end{cor}
\begin{proof} We have to show $ W_{\tau}=0$ for each $\tau$ with  $\tau\ne\sigma_U$ for all $U\in b(\mathcal D)$.
But by Proposition~\ref{prop:weightab} and Lemma~\ref{lemma:abcdinvariance}, the direct sum $\bigoplus_{\nu\ne \tau}{W}_\nu \subset {W}$ satifies the conditions of Proposition~\ref{prop:nakajimaSemistabilityCriterion} and hence equals $W$. Thus, its complement is zero and $ W_{\tau}=0$.
%Given a weight space $ W_{\tau}\subset {W}$ with $\tau\ne\sigma_U$ for all $U$ D5. We have to show $ W_{\tau}=0$. 
%By Proposition~\ref{prop:weightab}.\ref{item:weightisos}, $A_U$ induces an isomorphism $ W_{\tau,U^+}\rightarrow  W_{\tau,U^-}$ for all $U$ D5. Hence, it follows from Lemma~\ref{lemma:abcdinvariance} and Proposition~\ref{prop:weightab}.(\ref{item:aweight}) that the subspace
%$
%\bigoplus_{\nu\ne \tau}{W}_\nu \subset {W}
%$
%satisfies the conditions of Proposition~\ref{prop:nakajimaSemistabilityCriterion}. Thus, $\bigoplus_{\nu\ne \tau}{W}_\nu = {W}$ which implies ${W}_\tau=0$.
%By Proposition~\ref{prop:weightab}.\ref{item:weightisos}
\end{proof}

\subsection{Bases and diagrammatics for the blue part}

Let $U_i$ be a blue line in $b(\mathcal D)$. Next, we employ Proposition~\ref{prop:weightab} and the stability condition~\ref{item:S2} to determine bases of the spaces $W_{\sigma_{U_i},X_j}$ for $j=M+1,\ldots,M+N+1$ and describe the restrictions of the operators $A_U,B_U^-,B_U^+$ with respect to these bases for all $U\in b(\mathcal D)$.

%Our next aim is to give a diagrammatic description of the restrictions of the operators $A_U,B_U^-,B_U^+$ to $W_{\sigma_{U_i}}$ for every $U$ D5.

\begin{cor}\label{cor:weightspaceoperatorsD5}  The following holds:
\begin{enumerate}[label=(\roman*)]
\item\label{item:D5vanishing} We have $ W_{\sigma_{U_i},X_{M+i+1+j}}=0$ for $j\geq 1$.
\item\label{item:D5generators} The $\mathbb C$-vector space $ W_{\sigma_{U_i},U_i^-}$ is generated by $\{(B_{U_i}^-)^ia_U(1)|i\geq0 \}$.
\item\label{item:D5isos} The operator $A_{U_j}$ induces an isomorphism of vector spaces $ W_{\sigma_{U_i},X_j+1}\xrightarrow\sim  W_{\sigma_{U_i},X_j} $ for $M+1\le j \le M+i-1$.
\end{enumerate}
\end{cor}

\begin{proof} According to Proposition~\ref{prop:weightab}.\ref{item:weightisos}, the subspaces ${W}_{\sigma_{U_i},X_{M+i+1+j}}$, are mutually isomorphic for $j\geq 1$.
Hence, \ref{item:D5vanishing} follows from $W_{X_{M+N+1}}=0$. For \ref{item:D5generators}, let $E:=\operatorname{span}_{\mathbb C}((B_{U_i}^-)^ia_U(1)|i\geq0 )$. Since $ W_{\sigma_{U_i},{U_i}^+}=0$, the subspace
$
E\oplus \bigoplus_{\tau\ne \sigma_{U_i}}{W}_{\tau}\subset  W
$
equals $W$ by~\ref{item:S2}, which implies $E= W_{\sigma_{U_i},U_i^-}$.
Statement \ref{item:D5isos} is immediate from Proposition~\ref{prop:weightab}.\ref{item:weightisos}.
\end{proof}

%From Lemma~\ref{lemma:weightspaceoperatorsD5}, we get the following diagrammatic description:
Using Corollary~\ref{cor:weightspaceoperatorsD5} we give a basis for each  $W_{\sigma_{U_i},X_j}$ where $M+1\le j \le M+i$ as follows: Let $r=\operatorname{dim}(W_{\sigma_{U_i},X_{M+i}})$ and we set $y_{M+i}:=a_{U_i}(1)\in W_{\sigma_{U_i},X_{M+i}}$. In addition, we define recursively
$y_{M+i-k}:=A_{U_{i-k}}\dots A_{U_{i-1}}y_{M+i}\in W_{\sigma_{U_i},X_{M+i-k}}$  for $1\le k <i$ and we set $y_{M+l,k}:=(-B_{U_{l}}^-)^k y_{M+l}$ for $l=1,\ldots,i,k\geq0$.
\begin{cor}\label{cor:D5Basis} Let $l\in\{l=1,\ldots,i\}$. Then, 
$(y_{M+1,0},\ldots,y_{M+1,r-1})$ is a basis of $W_{\sigma_{U_i},X_{M+l}}$.
\end{cor}
\begin{proof} By \eqref{eq:nonMomentMap} and Proposition~\ref{prop:weightab}.\ref{item:bweight}, we have
\begin{equation}\label{eq:nonMomentmapOnWeightSpace}
A_{U_{j-1}} B_{U_{j}}^-w = B_{j-1}^-A_{U_{j-1}}w \quad\textup{for $j=2,\ldots,i$ and $w\in W_{\sigma_{U_i},U_{j}^-}$.}
\end{equation}
Thus, the corollary follows from Corollary~\ref{cor:weightspaceoperatorsD5}.\ref{item:D5generators} and Corollary~\ref{cor:weightspaceoperatorsD5}.\ref{item:D5isos}.
\end{proof}

We denote the basis $(y_{M+l,0},\ldots,y_{M+l,r-1})$ of $W_{\sigma_{U_i},X_{M+l}}$ by $\mathfrak B_{M+l}$ for $l=1,\ldots,i$.
Our previous considerations lead to the following diagrammatic description of the operators $A_U,B_U^-,B_U^+$ with respect to this choice of bases:

\begin{cor}[Blue operators]\label{cor:D5part} The restrictions of the operators $(A_U,-B_U^-,-B_U^+)_U$ and $a_{U_i},b_{U_i}$ to ${ W}_{\sigma_{U_i}}$ with respect to the bases $\mathfrak B_{M+1},\ldots,\mathfrak B_{M+i}$ are illustrated by the diagram in Figure~\ref{fig:D5part} where each column contains $r$ dots.
\end{cor}

\begin{proof} 
By Corollary~\ref{cor:weightspaceoperatorsD5}.\ref{item:D5isos}, the dimension of the vector spaces $W_{\sigma_{U_i},X_j}$ match with the diagram. Proposition~\ref{prop:nilpotency} gives that the operators $B_U^-$ are nilpotent. Thus, by Corollary~\ref{cor:D5Basis}, the operators $-B_U^-$ (and equivalently $-B_U^+$) act on the chosen basis as in the diagram. It follows from~\eqref{eq:nonMomentmapOnWeightSpace} that also the operators $A_U$ act as illustrated in the diagram. By definition, we have $a_{U_i} (1)=y_{M+i,0}$ and since $W_{\sigma_{U_i},X_{M+i+1}}=0$ we also have $b_{U_i}=0$ by Proposition~\ref{prop:weightab}.\ref{item:bweight}.
\end{proof}

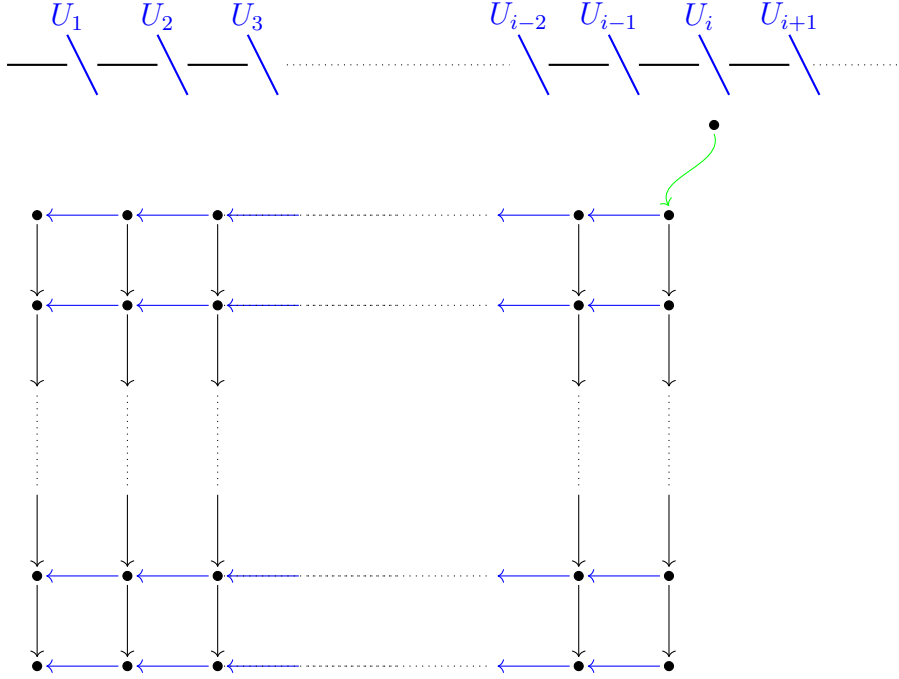
\begin{figure}
\centering
\begin{tikzpicture}
[scale=.4]
\draw[thick] (-1.5,1)--(.5,1);
\draw [thick,blue] (1.5,0) --(0.5,2); 
\node[blue] at (.5,2.6) {$U_1$};
\draw[thick] (1.5,1)--(3.5,1);
\draw [thick,blue](4.5,0) --(3.5,2);  
\node[blue] at (3.5,2.6) {$U_2$};
\draw [thick](4.5,1)--(6.5,1);
\draw [thick,blue](7.5,0) --(6.5,2);  
\node[blue] at (6.5,2.6) {$U_3$};

\draw [thick,blue](16.5,0) --(15.5,2);  
\node[blue] at (15.5,2.6) {$U_{i-2}$};
\draw [thick,blue](19.5,0) --(18.5,2);
\draw[thick] (16.5,1)--(18.5,1);  
\draw[thick] (19.5,1)--(21.5,1); 
\draw[thick] (22.5,1)--(24.5,1); 
\node[blue] at (18.5,2.6) {$U_{i-1}$};
\draw [thick,blue](22.5,0) --(21.5,2);  
\node[blue] at (21.5,2.6) {$U_{i}$};
\draw [thick,blue](25.5,0) --(24.5,2);  
\node[blue] at (24.5,2.6) {$U_{i+1}$};

\draw[dotted] (25.3,1) -- (28.3,1);

\draw[fill] (-0.5,-4) circle [radius=.15];
\draw[fill] (-0.5,-7) circle [radius=.15];
%\draw[fill] (-0.5,-10) circle [radius=.15];
\draw[fill] (-0.5,-16) circle [radius=.15];
\draw[fill] (-0.5,-19) circle [radius=.15];

\draw[dotted] (-0.5,-10) -- (-0.5,-13);

\draw [->] (-0.5,-4.3) -- (-0.5,-6.7);
\draw [->] (-0.5,-7.3) -- (-0.5,-9.7);
%\draw [->] (-0.5,-10.3) -- (-0.5,-12.7);
\draw [->] (-0.5,-13.3) -- (-0.5,-15.7);
\draw [->] (-0.5,-16.3) -- (-0.5,-18.7);

\draw[fill] (5.5,-4) circle [radius=.15];
\draw[fill] (5.5,-7) circle [radius=.15];
%\draw[fill] (5.5,-10) circle [radius=.15];
\draw[fill] (5.5,-16) circle [radius=.15];
\draw[fill] (5.5,-19) circle [radius=.15];

\draw[dotted] (5.5,-10) -- (5.5,-13);

\draw [->] (5.5,-4.3) -- (5.5,-6.7);
\draw [->] (5.5,-7.3) -- (5.5,-9.7);
%\draw [->] (5.5,-10.3) -- (5.5,-12.7);
\draw [->] (5.5,-13.3) -- (5.5,-15.7);
\draw [->] (5.5,-16.3) -- (5.5,-18.7);

\draw[fill] (2.5,-4) circle [radius=.15];
\draw[fill] (2.5,-7) circle [radius=.15];
%\draw[fill] (2.5,-10) circle [radius=.15];
\draw[fill] (2.5,-16) circle [radius=.15];
\draw[fill] (2.5,-19) circle [radius=.15];

\draw[dotted] (2.5,-10) -- (2.5,-13);

\draw [->] (2.5,-4.3) -- (2.5,-6.7);
\draw [->] (2.5,-7.3) -- (2.5,-9.7);
%\draw [->] (2.5,-10.3) -- (2.5,-12.7);
\draw [->] (2.5,-13.3) -- (2.5,-15.7);
\draw [->] (2.5,-16.3) -- (2.5,-18.7);

\draw [blue, ->] (2.2, -4) -- (-0.2,-4);
\draw [blue,->] (2.2, -7) -- (-0.2,-7);
%\draw [blue,->] (2.2, -10) -- (-0.2,-10);
\draw [blue,->] (2.2, -16) -- (-0.2,-16);
\draw [blue,->] (2.2, -19) -- (-0.2,-19);

\draw [blue,->] (5.2, -4) -- (2.8,-4);
\draw [blue,->] (5.2, -7) -- (2.8,-7);
%\draw [blue,->] (5.2, -10) -- (2.8,-10);
\draw [blue,->] (5.2, -16) -- (2.8,-16);
\draw [blue,->] (5.2, -19) -- (2.8,-19);

\draw [blue,->] (8.2, -4) -- (5.8,-4);
\draw [blue,->] (8.2, -7) -- (5.8,-7);
%\draw [blue,->] (8.2, -10) -- (5.8,-10);
\draw [blue,->] (8.2, -16) -- (5.8,-16);
\draw [blue,->] (8.2, -19) -- (5.8,-19);

\draw [blue,->] (20.2, -4) -- (17.8,-4);
\draw [blue,->] (20.2, -7) -- (17.8,-7);
%\draw [blue,->] (20.2, -10) -- (17.8,-10);
\draw [blue,->] (20.2, -16) -- (17.8,-16);
\draw [blue,->] (20.2, -19) -- (17.8,-19);

\draw[dotted] (7.8,1) -- (15.2,1);
\draw[dotted] (8.5,-4) -- (14.5,-4);
\draw[dotted] (8.5,-7) -- (14.5,-7);
%\draw[dotted] (8.5,-10) -- (14.5,-10);
\draw[dotted] (8.5,-16) -- (14.5,-16);
\draw[dotted] (8.5,-19) -- (14.5,-19);

\draw[dotted] (5.5,-4) -- (11.5,-4);
\draw[dotted] (5.5,-7) -- (11.5,-7);
%\draw[dotted] (5.5,-10) -- (11.5,-10);
\draw[dotted] (5.5,-16) -- (11.5,-16);
\draw[dotted] (5.5,-19) -- (11.5,-19);

\draw [blue,->] (17.2, -4) -- (14.8,-4);
\draw [blue,->] (17.2, -7) -- (14.8,-7);
%\draw [blue,->] (17.2, -10) -- (14.8,-10);
\draw [blue,->] (17.2, -16) -- (14.8,-16);
\draw [blue,->] (17.2, -19) -- (14.8,-19);

\draw[fill] (17.5,-4) circle [radius=.15];
\draw[fill] (17.5,-7) circle [radius=.15];
%\draw[fill] (17.5,-10) circle [radius=.15];
\draw[fill] (17.5,-16) circle [radius=.15];
\draw[fill] (17.5,-19) circle [radius=.15];

\draw[dotted] (17.5,-10) -- (17.5,-13);

\draw [->] (17.5,-4.3) -- (17.5,-6.7);
\draw [->] (17.5,-7.3) -- (17.5,-9.7);
%\draw [->] (17.5,-10.3) -- (17.5,-12.7);
\draw [->] (17.5,-13.3) -- (17.5,-15.7);
\draw [->] (17.5,-16.3) -- (17.5,-18.7);

\draw[fill] (20.5,-4) circle [radius=.15];
\draw[fill] (20.5,-7) circle [radius=.15];
%\draw[fill] (20.5,-10) circle [radius=.15];
\draw[fill] (20.5,-16) circle [radius=.15];
\draw[fill] (20.5,-19) circle [radius=.15];

\draw[dotted] (20.5,-10) -- (20.5,-13);

\draw [->] (20.5,-4.3) -- (20.5,-6.7);
\draw [->] (20.5,-7.3) -- (20.5,-9.7);
%\draw [->] (20.5,-10.3) -- (20.5,-12.7);
\draw [->] (20.5,-13.3) -- (20.5,-15.7);
\draw [->] (20.5,-16.3) -- (20.5,-18.7);

\draw[fill] (22,-1) circle [radius=.15];
\draw[->, green] (22,-1.3) to [out=-70,in=120] (20.5,-3.7);
%\draw[->, green] (21.7,-1.3) to (20.2,-3.7);
\end{tikzpicture}
\caption{Diagrammatic description of the restrictions of the operators $A_U,-B_U^-,-B_U^+$ for $U\in b(\mathcal D)$ and $a_{U_i},b_{U_i}$ to $W_{\sigma_{U_i}}$.}\label{fig:D5part}
\end{figure}

\subsection{Bases and diagrammatics for the red part}

Similar to the previous subsection, we now continue with characterizing bases for the weight spaces $W_{\sigma_{U_i},X_{j}}$ for $1\le j \le M$ and give diagrammatic descriptions of the restriction of the operators $C_V,D_V$ with respect to these particular bases for all red lines $V$.

At first, we set up some notation. Set \[z_{M+1}:=y_{M+1} \in {W}_{\sigma_{U_i},X_{M+1}} \]
and define $z_{M+1-j}\in W_{\sigma_{U_i},X_{M+1-j}}$ recursively as
$z_{M+1-j}=C_{V_{j}} z_{M+2-j}$,
for $j=1,\ldots,M$. Let $z_{l,k}:=(D_{V_{M+2-l}}C_{V_{M+2-l}})^k z_{l}$, for $l=2,\ldots, M+1,k\geq 0$. We set $ E_{l}:=\operatorname{span}_{\mathbb C}(z_{l,k}|k\geq 0)$ and
\[
E := \bigoplus_{l=2}^M E_{l}   \oplus \bigoplus_{l=1}^{M+N+1} W_{\sigma_{U_i},X_{M+1+l} } \subset W_{\sigma_{U_i}}.
\]
Note that by the moment map equation, we have $z_{M+1,k}=y_{M+1,k}$ for all $k\geq0$.
%In addition, we set $E_0:=  W_{\sigma_{U_i},X_{M+1}}$ and $z_{0,k}:=(-B_{U_1}^-)^k(z_0)$ for $k\geq 0$. 

\begin{prop}\label{prop:weightspaceNS5} We have $E =W_{\sigma_{U_i}} $.
\end{prop}

At first, we investigate how the operators $C_V,D_V$ act on the elements $z_{k,l}$:

\begin{lemma}\label{lemma:NS5invariance} With the above notation, the following holds:
\begin{enumerate}[label=(\roman*)]
	\item\label{item:NS5actionC} We have $C_{V_{M+1-l}}z_{l+1,k}=z_{l,k}$ for $k\geq 0,l=1,\ldots,M$.
	\item\label{item:NS5actionD} We have $D_{V_{M+1-l}}z_{l,k}= z_{l+1,k+1}$ for $k\geq 0,l=1,\ldots,M$.
\end{enumerate}
Moreover, $E$ is invariant under all $A_U,B_{U}^-,B_{U}^+,C_V,D_V$.
\end{lemma}

\begin{proof} The assertions \ref{item:NS5actionC} and \ref{item:NS5actionD} are immediate from the moment map equations. The invariance of $E$ under all $A_U,B_{U}^-,B_{U}^+$ follows directly from Lemma~\ref{lemma:abcdinvariance}. Furthermore, \ref{item:NS5actionC} and \ref{item:NS5actionD} imply that $ E$ is invariant under all $C_V,D_V$.
%, as each $\textbf{E}_{-j}$ is generated by the elements $x_{-j,k}$ for $k\geq 0$.
\end{proof}

The proof of Proposition~\ref{prop:weightspaceNS5} follows now from the stability criterion for bow varieties:

\begin{proof}[Proof of Proposition~\ref{prop:weightspaceNS5}] By Lemma~\ref{lemma:NS5invariance}, the subspace
$
 E' =  E\oplus \bigoplus_{\tau\ne\sigma_{U_i}} W_{\tau} \subset W.
$
satisfies the conditions of Proposition~\ref{prop:nakajimaSemistabilityCriterion} and hence equals $W$ which implies $ E =W_{\sigma_{U_i}}$.
\end{proof}

The proof of Proposition~\ref{prop:weightspaceNS5} leads to the following useful observation:

\begin{cor}\label{cor:NS5dimdrop} For given $V\in r(\mathcal D)$ the following holds:
\begin{enumerate}[label=(\roman*)]
\item\label{item:NS5onto} The operator $C_V$ induces a surjection $W_{\sigma_{U_i},V^+}\rightarrow  W_{\sigma_{U_i},V^-}$.
\item\label{item:NS5dimdrop} We have either $\operatorname{dim}(W_{\sigma_{U_i},V^+})= \operatorname{dim}( W_{\sigma_{U_i},V^-})$ or $\operatorname{dim}( W_{\sigma_{U_i},V^+})= \operatorname{dim}( W_{\sigma_{U_i},V^-})+1$.
\end{enumerate}
\end{cor}

\begin{proof} According to Lemma~\ref{lemma:NS5invariance}.\ref{item:NS5actionD} and Proposition~\ref{prop:weightspaceNS5}, the image of $C_V$ contains a generating system of $W_{\sigma_{U_i},V^-}$ which gives~\ref{item:NS5onto}.
For~\ref{item:NS5dimdrop}, write $V=V_{l}$ where $l=1,\ldots,M$. By Lemma~\ref{lemma:NS5invariance}.\ref{item:NS5actionC}, $C_{V_{l}}$ surjects onto the span of $\{z_{M+1-l,k}|k\geq 1\}$  which is a subspace of $W_{\sigma_{U_i},V^+}$ of codimension $1$ by Proposition~\ref{prop:nilpotency}. Combining this with \ref{item:NS5onto}, we obtain the following inequalities $
\operatorname{dim}(W_{\sigma_{U_i},V^+})-1 \le \operatorname{dim}(W_{\sigma_{U_i},V^-}) \le\operatorname{dim}(W_{\sigma_{U_i},X_{V^+}})$.
Thus, we conclude \ref{item:NS5dimdrop}.
\end{proof}

Now, by Corollary~\ref{cor:NS5dimdrop}, there exist $k_0:=1\le k_1 < \ldots < k_r \le k_{r+1}:=M$ such that
$\operatorname{dim}(W_{\sigma_{U_i},V_{k_j}^+ })=\operatorname{dim}(W_{\sigma_{U_i},V_{k_j}^-})+1$ for $j=1,\ldots,r$
and $\operatorname{dim}(W_{\sigma_{U_i},V_l^+})=\operatorname{dim}(W_{\sigma_{U_i},V_l^-})$ in case $k_j<l<k_{j+1}$ with  $j=0,\ldots,r$.
The following corollary is immediate from Proposition~\ref{prop:weightspaceNS5}:

\begin{cor}[Combinatorial bases]\label{cor:redBasis} Let $X_l\in h(\mathcal D)$ with $V_{k_{j+1}}\triangleleft X_l \triangleleft V_{k_j}$. Then, the vector space $W_{\sigma_{U_i},X}$ has basis $(z_{l,0},\ldots,z_{l,r-j-1})$. We denote this basis by $\mathfrak B_{U_i,l}$.
\end{cor}

With the above notation, we now give a diagrammatic description of the operators $C_V,D_V$:
%by Corollary~\ref{cor:NS5dimdrop} there exist $1=k_0\le k_1 \le \ldots \le k_r \le M=k_{r+1}$ such that
%\[
%\operatorname{dim}(W_{\sigma_{U_i},V_{k_j}^+ })=\operatorname{dim}(W_{\sigma_{U_i},V_{k_j}^-})+1,\quad j=1,\ldots,r,
%\]
%and
%\[
%\operatorname{dim}(W_{\sigma_{U_i},V_l^+})=\operatorname{dim}(W_{\sigma_{U_i},V_l^-}),\quad k_j<l<k_{j+1},j=0,\ldots,r.
%\]
%According to Proposition~\ref{prop:weightspaceNS5}, $W_{\sigma_{U_i},X_{l}}$ has basis $(z_{l,0},\ldots,z_{j-1})$ where $k_j+1\le l \le k_{j+1}$. We denote this basis by $\mathfrak B_{U_i,l}$.

\begin{cor}[Red operators]\label{cor:NS5part} The operators $C_{V_l},D_{V_l}$, where $l=k_j\ldots,k_{j+1}-1$, with respect to the bases $\mathfrak B_{k_j},\ldots \mathfrak B_{k_{j+1}}$ are illustrated by the diagram in Figure~\ref{fig:NS5part}.
\end{cor}

\begin{proof} By Proposition~\ref{prop:nilpotency}, we have $z_{l,k}=0$ for $k\geq \operatorname{dim}(W_{\sigma_{U_i},X_l})$. Hence, Lemma~\ref{lemma:NS5invariance} gives that the stated operators act on the given bases exactly as illustrated in the diagram.
\end{proof}

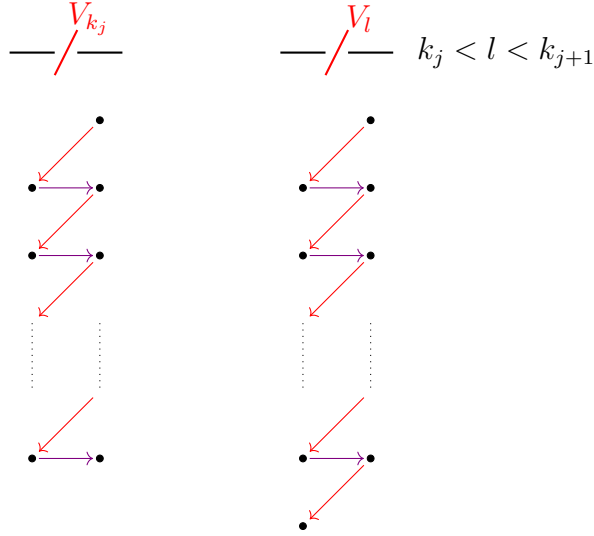
\begin{figure}
\centering
\begin{tikzpicture}
[scale=.3]
\draw[thick] (0,0)--(2,0);
\draw[thick] (3,0)--(5,0);
\draw [thick,red] (2,-1) --(3,1); 
\node[red] at (3.5,1.5) {$V_{k_{j}}$};
\draw[fill] (4,-3) circle [radius=.15];
\draw[fill] (4,-6) circle [radius=.15];
\draw[fill] (4,-9) circle [radius=.15];
\draw[fill] (4,-18) circle [radius=.15];
\draw[fill] (1,-6) circle [radius=.15];
\draw[fill] (1,-9) circle [radius=.15];
\draw[fill] (1,-18) circle [radius=.15];

\draw[dotted] (1,-12) -- (1,-15);
\draw[dotted] (4,-12) -- (4,-15);

\draw [red, ->] (3.7, -3.3) -- (1.3,-5.7);
\draw [red, ->] (3.7, -6.3) -- (1.3,-8.7);
\draw [red, ->] (3.7, -9.3) -- (1.3,-11.7);
\draw [red, ->] (3.7, -15.3) -- (1.3,-17.7);
\draw [violet, ->] (1.3,-6) -- (3.7,-6);
\draw [violet, ->] (1.3,-9) -- (3.7,-9);
\draw [violet, ->] (1.3,-18) -- (3.7,-18);

\draw[thick] (12,0)--(14,0);
\draw[thick] (15,0)--(17,0);
\draw [thick,red] (14,-1) --(15,1); 
\node[red] at (15.5,1.5) {$V_{l}$};
\node at (22,0) {$k_j<l<k_{j+1}$};
\draw[fill] (16,-3) circle [radius=.15];
\draw[fill] (16,-6) circle [radius=.15];
\draw[fill] (16,-9) circle [radius=.15];
\draw[fill] (16,-18) circle [radius=.15];
\draw[fill] (13,-6) circle [radius=.15];
\draw[fill] (13,-9) circle [radius=.15];
\draw[fill] (13,-18) circle [radius=.15];
\draw[fill] (13,-21) circle [radius=.15];

\draw[dotted] (13,-12) -- (13,-15);
\draw[dotted] (16,-12) -- (16,-15);

\draw [red, ->] (15.7, -3.3) -- (13.3,-5.7);
\draw [red, ->] (15.7, -6.3) -- (13.3,-8.7);
\draw [red, ->] (15.7, -9.3) -- (13.3,-11.7);
\draw [red, ->] (15.7, -15.3) -- (13.3,-17.7);
\draw [red, ->] (15.7, -18.3) -- (13.3,-20.7);
\draw [violet, ->] (13.3,-6) -- (15.7,-6);
\draw [violet, ->] (13.3,-9) -- (15.7,-9);
\draw [violet, ->] (13.3,-18) -- (15.7,-18);
\end{tikzpicture}
\caption{Diagrammatic description of $C_{V_l},D_{V_l}$ in case $l=k_j\ldots,k_{j+1}-1$.}\label{fig:NS5part}
\end{figure}

\subsection{Proof of the Generic Cocharacter Theorem}
The proof of the Generic Cocharacter Theorem is essentially a consequence of the diagrammatic description of the operators $A_U,B_U^+,B_U^-,C_V,D_V$ from Corollary~\ref{cor:D5part} and Corollary~\ref{cor:NS5part}:

\begin{proof}[Proof of Theorem~\ref{thm:fixedPointsGeneric1ParameterTorus}]
For a given $p\in\mathcal C(\mathcal D)^\sigma$ we define a tie diagram $D$ via
\[
(V,U)\in D \Leftrightarrow \operatorname{dim}(W_{\sigma_{U},V^+}) = \operatorname{dim}(W_{\sigma_{U},V^-})+1.
\]
Given $U_i\in b(\mathcal D)$, let $r,k_0,\ldots,k_{r+1}$ be defined as in the previous subsection. We have $d_{D,U_i,X_l}=0$ for $l>M+i$ and $d_{D,U_i,X_l}=r$ for $M+1\le l\le M+i$. Moreover,
$d_{D,U_i,X_l} = r-j$ for $j=0,\ldots,r$ and $V_{k_j+1}\triangleleft X_l \triangleleft V_{k_{j}}$. Thus, $d_{D,U_i,X_l}=\operatorname{dim}(W_{\sigma_{U_i},X_l})$ for all $X_l\in h(\mathcal D)$ which implies that $D$ is indeed a tie diagram of $\mathcal D$.
The corresponding column bottom indices are given by $c_{D,U_i,X_l}=0$ for $M+1\le l\le M+i$ and $c_{D,U_i,X_l}=l-M-1+j,$ for $V_{k_{j+1}}\triangleleft X_l\triangleleft V_{k_j}$, $j=0,\ldots,r$. 
Therefore, the vector spaces $F_{D,U_i,X_l}$ have bases $( e_{U_i,l-M-i,0},\ldots, e_{U_i,l-M-i,r-1})$ for $M+1\le l\le M+i$ and $(\tilde e_{U_i,X_l,0},\ldots,\tilde e_{U_i,X_l,r-1-j})$ for $V_{k_{j+1}}\triangleleft X_l\triangleleft V_{k_j}$, $j=0,\ldots,r$, where we set $\tilde e_{U_i,X_l,k}:=e_{U_i,l-M-i,l-M-2+r-k}$.
Consequently,  we can define isomorphisms of vector spaces 
$\phi_{U_i,X_l}:W_{\sigma_{U_i},X_l}\rightarrow F_{D,U,X_l}$
via
\[
\phi_{U_i,X_l}(y_{l,k})= e_{U_i,l-M-i,r-k-1},\quad M+1\le l\le M+i, k=0,\ldots,r-1
\]
and
\[
\phi_{U_i,X_l}(z_{l,k})= \tilde e_{U_i,X_l,k}\quad V_{k_{j+1}}\triangleleft X_l\triangleleft V_{k_j}, k=0,\ldots,j-1.
\]
For the other $X_l$, we have $W_{\sigma_{U_i},X_l}=0$, so we set $\phi_{U_i,X_l}=0$ for $X_l\triangleleft V_{k_M}$ and $X_l\triangleright U_i$. By Corollary~\ref{cor:D5part}, 
\begin{gather*}
\phi_{U_i,U^-}A_{U}(w_+)= A_{D,U_i,U}\phi_{U_i,U^+}(w_+),\quad \phi_{U_i,U^-}B_{U}^-(w_-)=B_{D,U_i,U}^-\phi_{U_i,U^-}(w_-),\\ 
\phi_{U_i,U^+}B_{U}^+(w_+)=B_{D,U_i,U}^+\phi_{U_i,U^+}(w_+),
\end{gather*}
for all $U\in b(\mathcal D)$, $w_-\in W_{\sigma_{U_i},U^-}, w_+\in W_{\sigma_{U_i},U^+}$. In addition,
$\phi_{U_i,U_i^-}a_{U_i}=a_{D,U_i}$ and $b_{U_i}=b_{D,U_i}=0$.
Likewise, Corollary~\ref{cor:NS5part} gives
\[
\phi_{U_i,V^-}C_{V}(v_+)= C_{D,U_i,V}\phi_{U_i,V^+}(v_+),\quad \phi_{U_i,V^+}D_{V}(v_-)= D_{D,U_i,V}\phi_{U_i,V^-}(v_-),
\]
for all $V\in r(\mathcal D), v_-\in W_{\sigma_{U_i},V^-}, v_+\in W_{\sigma_{U_i},V^+}$.
Thus, we proved that $p$ equals the $\mathbb T$-fixed point $x_D$ and hence $\mathcal C(\mathcal D)^\sigma=\mathcal C(\mathcal D)^{\mathbb T}$.
\end{proof}

%We close this section with a simple consequence of Theorem~\ref{thm:fixedPointsGeneric1ParameterTorus}. For this, let $T_x\mathcal C(\mathcal D)$ be the tangent space of $\mathcal C(\mathcal D)$ at a point $x$. If $x$ is a $\mathbb T$-fixed point, then $T_x\mathcal C(\mathcal D)$ is also naturally a module over $\mathbb T$.
%
%\begin{cor}\label{cor:tangentweights} Let $p\in\mathcal C(\mathcal D)^{\mathbb T}$ and $\tau$ be a $\mathbb T$-weight of $T_p\mathcal C(\mathcal D)$. Then, there exist $i,j\in \{1,\ldots, N\}$ with $i\ne j$ and $m\in\mathbb Z$ such that $\tau=t_i-t_j+mh$.
%\end{cor}
%
%\begin{proof} According to~\cite[Section~3.2]{rimanyi2020bow}, all $\mathbb T$-weights of $T_p\mathcal C(\mathcal D)$ are of the form $t_i-t_j+mh$, with $i,j\in \{1,\ldots, N\}$ and $m\in\mathbb Z$. By Theorem~\ref{thm:fixedPointsGeneric1ParameterTorus}, $p$ is an isolated $\mathbb A$-fixed point. Since $\mathcal C(\mathcal D)$ is smooth, the equivariant slice theorem, see e.g. \cite[Theorem~I.1.2]{audin2004torus}, implies that no $\mathbb A$-weight of $T_p\mathcal C(\mathcal D)$ is trivial. Thus, no $\mathbb T$-weight of $T_p\mathcal C(\mathcal D)$ is of the form $mh$ for $m\in\mathbb Z$ which proves the proposition.
%\end{proof}

\section{Attracting cells}
\label{section:attraction}

	The theory of attracting cells lays the foundation for the theory of stable envelopes.  Maulik and Okounkov consider in~\cite[Section~3.2]{maulik2019quantum} attracting cells for smooth and quasi-projective varieties $X$ with an algebraic torus action such that $X$ admits a proper equivariant morphism to an affine variety. As we discussed in Section~\ref{section:bowVarieties}, bow varieties are smooth and quasi-projective. Moreover, by their construction as GIT quotients, bow varieties also admit an equivariant proper morphism to their respective categorical quotients which is affine. 
	
	Therefore, bow varieties fit into the Maulik--Okounkov setup and all results in~\cite[Section~3.2]{maulik2019quantum} apply. The Generic Cocharacter Theorem \ref{thm:fixedPointsGeneric1ParameterTorus} implies that we are in the preferable situation of finitely many isolated torus fixed points which moreover can be described combinatorially, see Section~\ref{section:torusFixedPoints}. We now describe their attracting cells in more details. 
	
	The theory of attracting cells requires non-empty torus fixed locus of generic one-parameter subgroups. Thus, we restrict our attention to bow varieties $\mathcal C(\mathcal D)$ with $\mathcal C(\mathcal D)^\sigma\ne \emptyset$ where $\sigma$ is a generic cocharater of $\mathbb A$. By the Generic Cocharcter Theorem this condition is equivalent to the following combinatoric assumption:
	\begin{assumption} From now on we assume that the brane diagram $\mathcal D$ can be extended to a tie diagram.
	\end{assumption}
	
%%%%%%%%%%%%%%Notation should come earlier!!!!!!!!!!! ---What about prelims?
%Before we go into details, we set up some notation. 
%Given a $\mathbb C$-vector space $V$ of finite dimension with an algebraic $\mathbb C^\ast$-action, we denote by $V^+$ (resp. $V^-$) the subspace generated by all strictly positive (resp. negative) weight spaces. The subspace of $\mathbb C^\ast$-invariant vectors is denoted by $V^0$. Moreover, we set $V^{\geq 0}:=V^0\oplus V^+$ and $V^{\le 0}=V^0\oplus V^-$.

\subsection{Affine structure}\label{subsection:affinestr}
	Given a generic cocharacter $\sigma:\mathbb C^\ast \rightarrow \mathbb A$ and $p\in \mathcal C(\mathcal D)^{\mathbb T}$, we obtain by Corollary~\ref{cor:tangentweights} a splitting
	$T_p\mathcal C(\mathcal D)=T_p\mathcal C(\mathcal D)_\sigma^+\oplus T_p\mathcal C(\mathcal D)_\sigma^-$ in the subspace of strictly positive respectively strictly negative weights corresponding to $\sigma$. As the symplectic form on $\mathcal C(\mathcal D)$ is $\mathbb A$-invariant, the $\mathbb C$-vector spaces $T_p\mathcal C(\mathcal D)_\sigma^+$ and $
	T_p\mathcal C(\mathcal D)_\sigma^-$ have the same dimension.	
	\begin{definition}
	The \textit{attracting cell of $p$ with respect to $\sigma$} is defined as
	\[
	\operatorname{Attr}_{\sigma}(p):= \{ z\in\mathcal C(\mathcal D)\mid \lim_{t\to0}\sigma(t).z=p \}.
	\]
	\end{definition}
	By definition, $\operatorname{Attr}_{\sigma}(p)$ is just a $\mathbb T$-invariant subset of $\mathcal C(\mathcal D)$. The following proposition shows that it actually carries the structure of a locally closed affine subvariety.
	
	\begin{prop}\label{prop:affinestructure} The attracting cell $\operatorname{Attr}_{\sigma}(p)$ is a locally closed $\mathbb T$-invariant subvariety of $\mathcal C(\mathcal D)$ which is $\mathbb T$-equivariantly isomorphic to the affine space $T_p\mathcal C(\mathcal D)_\sigma^+$.
	\end{prop}
	\begin{proof} 
	Since $\mathcal C(\mathcal D)$ is smooth and quasi-projective, there exists by \cite[Theorem~1]{sumihiro1974equivariant} a projective variety $X$ with $\mathbb T$-action such that there is an open dense $\mathbb T$-equivariant embedding $\mathcal C(\mathcal D)\hookrightarrow X$.
	Thanks to the equivariant Hironaka theorem, see e.g.~\cite[Theorem 1.0.3]{wlodarczyk2005simple}, we can assume that $X$ is smooth. The classical Bialynicki-Birula theorem \cite[Theorem~4.3]{bialynicki1973some} gives that
	\[
		X_{p,\sigma}^+:=\{ x\in X\mid \lim_{t\to0}\sigma(t).x=p \}	
	\]
	is a locally closed $\mathbb T$-invariant subvariety of $X$ which is isomorphic as variety to $(T_pX)_\sigma^+$. By~\cite[Remark]{konarski1996bb}, this isomorphism can be chosen to be $\mathbb T$-equivariant. Since the $\mathbb T$-equivariant embedding of $\mathcal C(\mathcal D)$ in $X$ is open dense, we have $(T_pX)_\sigma^+=(T_p\mathcal C(\mathcal D))_\sigma^+$ and $\operatorname{Attr}_{\sigma}(p)=X_{p,\sigma}^+ \cap \mathcal C(\mathcal D)$ is a locally closed subvariety of $\mathcal C(\mathcal D)$. To conclude that $\operatorname{Attr}_{\sigma}(p)$
	is $\mathbb T$-equivariantly isomorphic to $T_p\mathcal C(\mathcal D)_\sigma^+$, we show that $\operatorname{Attr}_{\sigma}(p)=X_{p,\sigma}^+ $. For this, note that $\operatorname{Attr}_{\sigma}(p)$ is an open, $\mathbb T$-invariant subvariety of $X_{p,\sigma}^+$ containing $p$. Hence, Lemma~\ref{lemma:openorigin} below implies $\operatorname{Attr}_{\sigma}(p)=X_{p,\sigma}^+$ which completes the proof.
%	 there exists a smooth projective variety $X$ together with a $\mathbb C^\ast$-action on $X$ such that there is a $\sigma$-equivariant open dense embedding $\mathcal C(\mathcal D)\hookrightarrow X$. Let $Z(p)\subset X$ denote the attracting cell of $p$ corresponding to the $\mathbb C^\ast$-action on $X$. By the classical Bialynicki-Birula theorem \cite[Theorem 4.3]{bialynicki1973some}, $Z(p)$ is a locally closed subvariety of $X$ which is isomorphic to the affine space $N_p^+$. According to Proposition~\ref{prop:tangentweights}, we have $\operatorname{dim}(N_p^+)=\frac12 \operatorname{dim}(\mathcal C(\mathcal D))=:m$. Fix an isomorphism of varieties $Z(p)\cong\mathbb A^m=\operatorname{Spec}(\mathbb C[x_1,\ldots,x_m])$ such that $p$ is mapped to the origin in $\mathbb A^m$. Since every point in $Z(p)$ is attracted by $p$, the induced $\mathbb C^\ast$-action on $\mathbb A^m$ corresponds to an $\mathbb N_0$-grading on $\mathbb C[x_1,\ldots,x_n]$ such that
%	\[
%	(x_1,\ldots,x_n)\subset \bigoplus_{i\geq 1} \mathbb C[x_1,\ldots,x_n]_i.
%	\]
%	As $\operatorname{Attr}_{\sigma}(p)=Z(p)\cap \mathcal C(\mathcal D)$ is an open $\mathbb C^\ast$-invariant subset of $Z(p)$ containing the origin, we can apply Lemma~\ref{lemma:openorigindense} which implies $\operatorname{Attr}_{\sigma}(p)=Z(p)$ as sets. This proves that $\operatorname{Attr}_{\sigma}(p)$ is a locally closed subvariety of $\mathcal C(\mathcal D)$ isomorphic to the affine space $N_p^+$. %This completes completes the proof of the proposition.
	\end{proof}
	
	\begin{lemma}\label{lemma:openorigin} Let $W$ be a finite dimensional vector space with a linear $\mathbb C^{\ast}$-action such that all $\mathbb C^\ast$-weights of $W$ are strictly positive\footnote{note that they are integers}. Let $U\subset W$ be an open $\mathbb C$-invariant subvariety containing the origin. Then, $U=W$.
	\end{lemma}
	\begin{proof} Let $w\in W$ and $\overline{\mathbb C^\ast.w}$ be the Zariski closure of the $\mathbb C^\ast$-orbit of $\mathbb C^\ast.w$ in $W$. As all $\mathbb C^\ast$-weights of $W$ are strictly positive, $\overline{\mathbb C^\ast.w}$ contains the origin. So $U\cap \overline{\mathbb C^\ast.w}$ is a non-empty, open, $\mathbb C^{\ast}$-invariant subvariety of $\overline{\mathbb C^\ast.w}$ which implies $w\in U$.
	\end{proof}
	
	\subsection{Attracting cells in a concrete example} \label{subsection:ExampleAttractingCells}
	Let $\mathcal D$ be the brane diagram
	
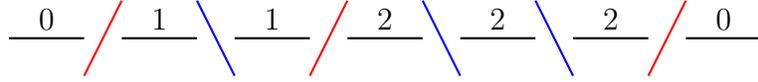
\[
	\begin{tikzpicture}
	[scale=.5]
	\draw[thick] (0,0)--(2,0);
	\draw[thick] (3,0)--(5,0);
	\draw[thick] (6,0)--(8,0);
	\draw[thick] (9,0)--(11,0);
	\draw[thick] (12,0)--(14,0);
	\draw[thick] (15,0)--(17,0);
	\draw[thick] (18,0)--(20,0);

	\draw [thick,red] (2,-1) --(3,1);
	\draw [thick,red] (8,-1) --(9,1);
	\draw [thick,red] (17,-1) --(18,1);
	\draw [thick,blue] (6,-1) --(5,1);
	\draw [thick,blue] (12,-1) --(11,1);
	\draw [thick,blue] (15,-1) --(14,1);

	\node at (1,0.5){$0$};
	\node at (4,0.5){$1$};
	\node at (7,0.5){$1$};
	\node at (10,0.5){$2$};
	\node at (13,0.5){$2$};
	\node at (16,0.5){$2$};
	\node at (19,0.5){$0$};
	\end{tikzpicture}
	\]
	We encode the elements of $\widetilde{\mathcal M}(\mathcal D)$, of $m^{-1}(0)$ and of the bow variety $\mathcal C(\mathcal D)$ again as tuples of endomorphisms with the notation given by the following diagram:
	
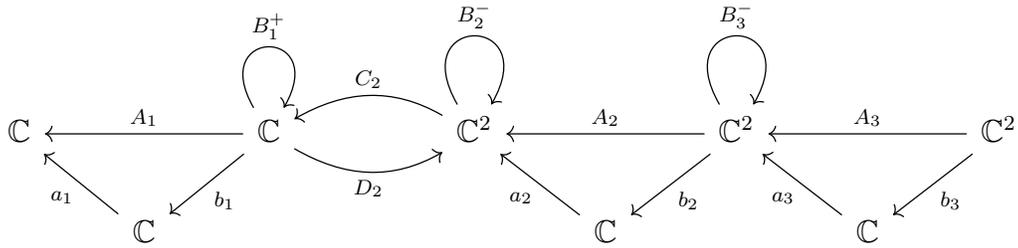
\[
	\begin{tikzcd}
\mathbb C && \mathbb C \arrow[ll,"A_1",swap]\arrow[ld,"b_1"] \arrow[in=60,out=120,"B_1^+",loop] \arrow[rr, in=210,out=330,"D_2",swap] && \mathbb C^2 \arrow[ll, in=30,out=150,"C_2" ,swap] \arrow[in=60,out=120,"B_2^-",loop] &&\arrow[ll,"A_2",swap] \mathbb C^2 \arrow[in=60,out=120,"B_3^-",loop]\arrow[ld,"b_2"]&&\arrow[ll,"A_3",swap] \mathbb C^2\arrow[ld,"b_3"] \\
		&\mathbb C\arrow[lu,"a_1"] &&&& \mathbb C\arrow[lu,"a_2"] && \mathbb C \arrow[lu,"a_3"]
	\end{tikzcd}
	\]
	Here, we dropped the operators $C_1,C_3,D_1,D_3,B_1^-$ and $B_3^+$ from the picture as they always vanish. We also identified $B_{2}^{+}$ and $B_{3}^{-}$ according to the moment map equation. Moreover, note that $A_1,A_2,A_3$ are isomorphisms by Proposition~\ref{prop:takayamaS1S2implications}.
	
	One can easily check that $\mathcal D$ can be extended to exactly five different tie diagrams:
	
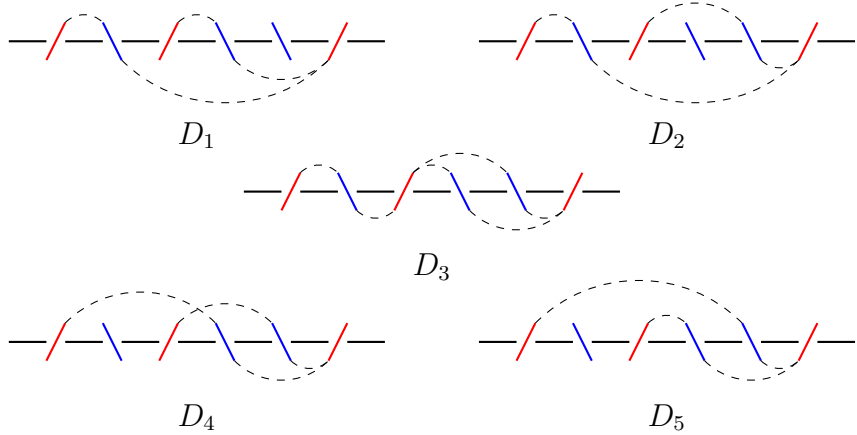
\[
\begin{tikzpicture}
[scale=.25]
\draw[thick] (0,0)--(2,0);
\draw[thick] (3,0)--(5,0);
\draw[thick] (6,0)--(8,0);
\draw[thick] (9,0)--(11,0);
\draw[thick] (12,0)--(14,0);
\draw[thick] (15,0)--(17,0);
\draw[thick] (18,0)--(20,0);

\draw [thick,red] (2,-1) --(3,1);
\draw [thick,red] (8,-1) --(9,1);
\draw [thick,red] (17,-1) --(18,1);
\draw [thick,blue] (6,-1) --(5,1);
\draw [thick,blue] (12,-1) --(11,1);
\draw [thick,blue] (15,-1) --(14,1);

\draw[dashed] (3,1) to[out=45,in=135] (5,1);
\draw[dashed] (9,1) to[out=45,in=135] (11,1);
\draw[dashed] (6,-1) to[out=315,in=225] (17,-1);
\draw[dashed] (12,-1) to[out=315,in=225]  (17,-1);

\node at (10,-5){$D_1$};

\draw[thick] (25,0)--(27,0);
\draw[thick] (28,0)--(30,0);
\draw[thick] (31,0)--(33,0);
\draw[thick] (34,0)--(36,0);
\draw[thick] (37,0)--(39,0);
\draw[thick] (40,0)--(42,0);
\draw[thick] (43,0)--(45,0);

\draw [thick,red] (27,-1) --(28,1);
\draw [thick,red] (33,-1) --(34,1);
\draw [thick,red] (42,-1) --(43,1);
\draw [thick,blue] (31,-1) --(30,1);
\draw [thick,blue] (37,-1) --(36,1);
\draw [thick,blue] (40,-1) --(39,1);

\draw[dashed] (28,1) to[out=45,in=135] (30,1);
\draw[dashed] (34,1) to[out=45,in=135] (39,1);
\draw[dashed] (31,-1) to[out=315,in=225] (42,-1);
\draw[dashed] (40,-1) to[out=315,in=225]  (42,-1);

\node at (35,-5){$D_2$};

\draw[thick] (12.5,-8)--(14.5,-8);
\draw[thick] (15.5,-8)--(17.5,-8);
\draw[thick] (18.5,-8)--(20.5,-8);
\draw[thick] (21.5,-8)--(23.5,-8);
\draw[thick] (24.5,-8)--(26.5,-8);
\draw[thick] (27.5,-8)--(29.5,-8);
\draw[thick] (30.5,-8)--(32.5,-8);

\draw [thick,red] (14.5,-9) --(15.5,-7);
\draw [thick,red] (20.5,-9) --(21.5,-7);
\draw [thick,red] (29.5,-9) --(30.5,-7);
\draw [thick,blue] (18.5,-9) --(17.5,-7);
\draw [thick,blue] (24.5,-9) --(23.5,-7);
\draw [thick,blue] (27.5,-9) --(26.5,-7);

\draw[dashed] (15.5,-7) to[out=45,in=135] (17.5,-7);
\draw[dashed] (21.5,-7) to[out=45,in=135] (23.5,-7);
\draw[dashed] (21.5,-7) to[out=45,in=135] (26.5,-7);
\draw[dashed] (18.5,-9) to[out=315,in=225] (20.5,-9);
\draw[dashed] (24.5,-9) to[out=315,in=225]  (29.5,-9);
\draw[dashed] (27.5,-9) to[out=315,in=225]  (29.5,-9);

\node at (22.5,-12){$D_3$};

\draw[thick] (0,-16)--(2,-16);
\draw[thick] (3,-16)--(5,-16);
\draw[thick] (6,-16)--(8,-16);
\draw[thick] (9,-16)--(11,-16);
\draw[thick] (12,-16)--(14,-16);
\draw[thick] (15,-16)--(17,-16);
\draw[thick] (18,-16)--(20,-16);

\draw [thick,red] (2,-17) --(3,-15);
\draw [thick,red] (8,-17) --(9,-15);
\draw [thick,red] (17,-17) --(18,-15);
\draw [thick,blue] (6,-17) --(5,-15);
\draw [thick,blue] (12,-17) --(11,-15);
\draw [thick,blue] (15,-17) --(14,-15);

\draw[dashed] (3,-15) to[out=45,in=135] (11,-15);
\draw[dashed] (9,-15) to[out=45,in=135] (14,-15);
\draw[dashed] (12,-17) to[out=315,in=225] (17,-17);
\draw[dashed] (15,-17) to[out=315,in=225]  (17,-17);

\node at (10,-20){$D_4$};

\draw[thick] (25,-16)--(27,-16);
\draw[thick] (28,-16)--(30,-16);
\draw[thick] (31,-16)--(33,-16);
\draw[thick] (34,-16)--(36,-16);
\draw[thick] (37,-16)--(39,-16);
\draw[thick] (40,-16)--(42,-16);
\draw[thick] (43,-16)--(45,-16);

\draw [thick,red] (27,-17) --(28,-15);
\draw [thick,red] (33,-17) --(34,-15);
\draw [thick,red] (42,-17) --(43,-15);
\draw [thick,blue] (31,-17) --(30,-15);
\draw [thick,blue] (37,-17) --(36,-15);
\draw [thick,blue] (40,-17) --(39,-15);

\draw[dashed] (28,-15) to[out=45,in=135] (39,-15);
\draw[dashed] (34,-15) to[out=45,in=135] (36,-15);
\draw[dashed] (37,-17) to[out=315,in=225] (42,-17);
\draw[dashed] (40,-17) to[out=315,in=225]  (42,-17);

\node at (35,-20){$D_5$};

\end{tikzpicture}
	\]
	Consequently, the fixed point locus is $\mathcal C(\mathcal D)^{\mathbb T}=\{x_{D_1},x_{D_2},x_{D_3},x_{D_4},x_{D_5}\}$, see Section~\ref{section:torusFixedPoints}. In order to determine the attracting cells of these $\mathbb T$-fixed points, we first describe a covering of $\mathcal C(\mathcal D)$ by open affine $\mathbb T$-invariant subvarieties. For this, note that by Proposition~\ref{prop:nakajimaSemistabilityCriterion} a point
	\[
	x=(A_1,A_2,A_3,B_1^+,B_2^-,B_3^-,C_2,D_2,a_1,a_2,a_3)\in m^{-1}(0) 
	\]
	is semistable if and only if the following equalities hold
	\[
	\mathrm{Im}(a_1)+\mathrm{Im}(A_1C_2a_2)+\mathrm{Im}(A_1C_2a_3)=\mathbb C,\quad  \mathrm{Im}(a_2)+\mathrm{Im}(A_2a_3)+\mathrm{Im}(D_2A_1^{-1}a_1)=\mathbb C^2.
	\]
	It follows that $x\in m^{-1}(0)^{\mathrm{s}}$ if and only if one of the following five conditions is satisfied:
	\begin{enumerate}[label=(cov-\,\arabic*), wide=20pt]
		\item $a_1\ne0$ and $\operatorname{det}( a_2 \; D_2 ) \ne 0$,
		\item $a_1\ne0$ and $\operatorname{det}( A_2a_3 \; D_2 ) \ne 0$,
		\item $a_1\ne0$ and $\operatorname{det}( a_2 \; A_2a_3 ) \ne 0$,
		\item $C_2a_2\ne0$ and $\operatorname{det}( a_2 \; A_2a_3 ) \ne 0$,
		\item $C_2A_2a_3\ne0$ and $\operatorname{det}( a_2 \; A_2a_3 ) \ne 0$.
	\end{enumerate}
	We show now that there is a covering of $\mathcal C(\mathcal D)$
	by open affine $\mathbb T$-invariant subvarieties
	\begin{equation}\label{eq:affcov} 
	\mathcal C(\mathcal D)=\bigcup_{i=1}^5W_i.
	\end{equation}
	
To see this consider the decomposition of $m^{-1}(0)^{\mathrm{s}}$ by open affine $\mathbb T$-invariant and $\mathcal G$-invariant subvarieties
	\[
	m^{-1}(0)^{\mathrm{s}}=\widetilde W_1 \cup \widetilde W_2 \cup  \widetilde W_3 \cup  \widetilde W_4 \cup \widetilde W_5, 
	\]
	where $\widetilde W_i = \{x\in m^{-1}(0)\mid \text{such that (cov-i)  holds}\}$ for $i=1,\ldots,5$. 
%	\begin{align*}
%	\widetilde W_1 &= \{x\in m^{-1}(0)\mid a_1\ne0\textup{ and }\operatorname{det}( a_2 \; D_2 ) \ne 0 \}, \\
%	\widetilde W_2 &= \{x\in m^{-1}(0)\mid a_1\ne0\textup{ and }\operatorname{det}( A_2a_3 \; D_2 ) \ne 0 \}, \\
%	\widetilde W_3 &= \{x\in m^{-1}(0)\mid a_1\ne0\textup{ and }\operatorname{det}( a_2 \; A_2a_3 ) \ne 0 \}, \\
%	\widetilde W_4 &= \{x\in m^{-1}(0)\mid C_2a_2\ne0\textup{ and }\operatorname{det}( a_2 \; A_2a_3 ) \ne 0 \}, \\
%	\widetilde W_5 &= \{x\in m^{-1}(0)\mid C_2A_2a_3\ne0\textup{ and }\operatorname{det}( a_2 \; A_2a_3 ) \ne 0 \}.
%	\end{align*}
	Setting $W_i:=\widetilde W_i/\mathcal G\subset \mathcal C(\mathcal D)$,  provides a covering \eqref{eq:affcov} by open affine $\mathbb T$-invariant subvarieties.  Note that $x_{D_i}\in W_i$ for each $i$. It remains to show that the $W_i$ are isomorphic to affine spaces:
	\begin{claim} The parametrization from Figure~\ref{fig:TInvariantNeighborhoods} gives, for any $i=1,\ldots,5$,  an isomorphism of varieties $\eta_i:\mathbb C^4\xrightarrow\sim W_i$ with $\eta_i(0)=x_{D_i}$. In particular, \eqref{eq:affcov} is a covering by affine spaces.
	\end{claim}
	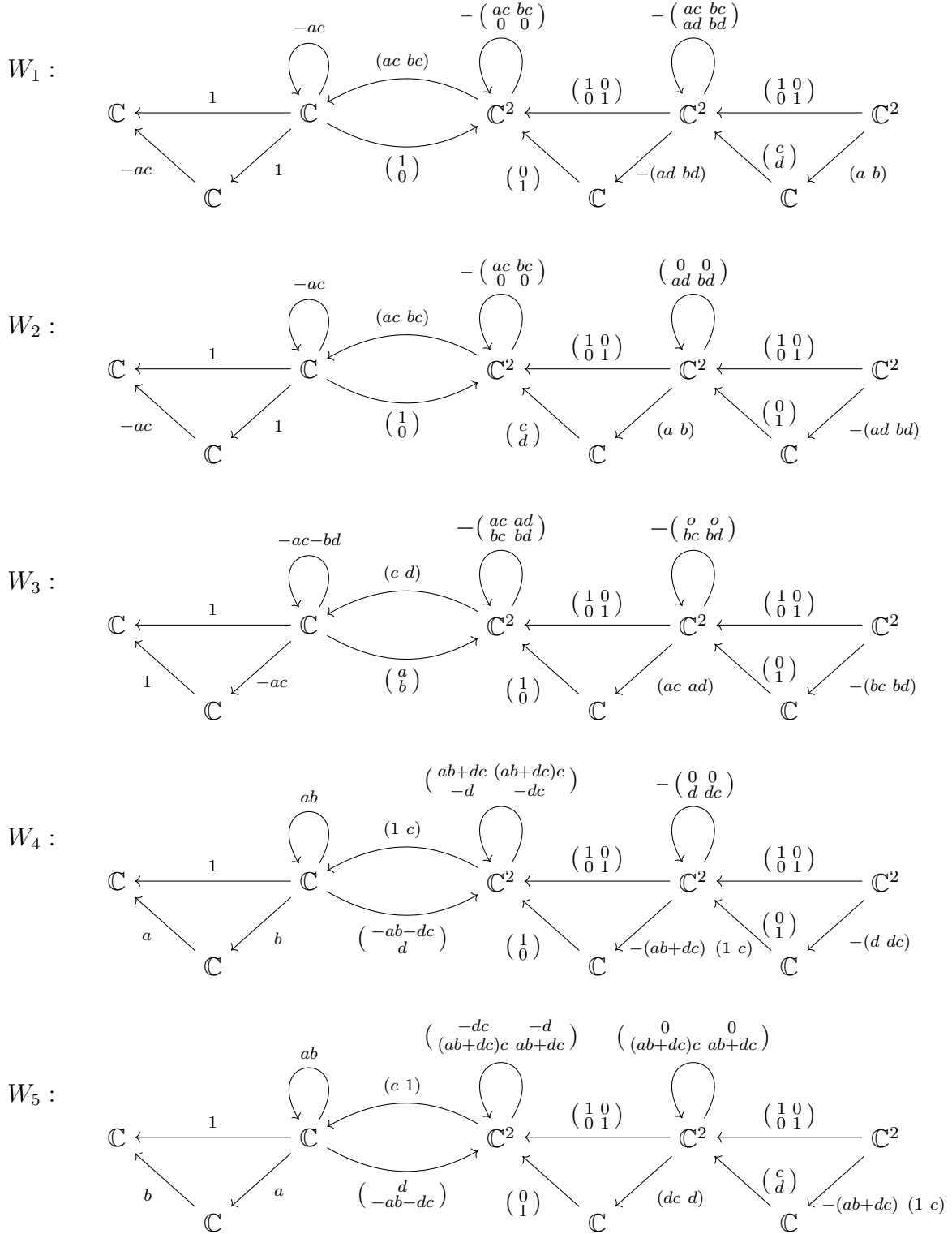
\begin{figure}
	\centering
	\begin{tikzpicture}[scale=.7]
	\node (W1) at (0,2) {$W_1:$};
	\node (W1Q1) at (2,1) {$\mathbb C$};
	\node (W1Q2) at (6.5,1) {$\mathbb C$};
	\node (W1Q3) at (11,1) {$\mathbb C^2$};
	\node (W1Q4) at (15.5,1) {$\mathbb C^2$};
	\node (W1Q5) at (20,1) {$\mathbb C^2$};
	\node (W1F1) at (4.25,-1) {$\mathbb C$};
	\node (W1F2) at (13.25,-1) {$\mathbb C$};
	\node (W1F3) at (17.75,-1) {$\mathbb C$};
	
	%B_1^+
	\draw[->] (W1Q2) to[in=120, out=60,looseness=10] node[midway,above] {$\begin{smallmatrix} -ac
	\end{smallmatrix}$} (W1Q2);
	%B_2^-	
	\draw[->] (W1Q3) to[in=120, out=60,looseness=10] node[midway,above] {$\begin{smallmatrix}-\end{smallmatrix}(\begin{smallmatrix} ac & bc \\ 0& 0 \end{smallmatrix})$} (W1Q3);
	%B_3^-
	\draw[->] (W1Q4) to[in=120, out=60,looseness=10] node[midway,above] {$\begin{smallmatrix}-\end{smallmatrix}(\begin{smallmatrix} ac & bc \\ ad & bd \end{smallmatrix})$} (W1Q4);
	%D_2
	\draw[->] (W1Q2) to[in=210, out=330] node[midway,below] {$(\begin{smallmatrix} 1 \\ 0 \end{smallmatrix})$} (W1Q3);
	%C_2
	\draw[->] (W1Q3) to[in=30, out=150] node[midway,above] {$\begin{smallmatrix} (ac & bc) \end{smallmatrix}$}(W1Q2);
	%A_1
	\draw[->] (W1Q2) to[in=0,out=180] node[midway,above] {$\begin{smallmatrix} 1
	\end{smallmatrix}$} (W1Q1);
	%A_2
	\draw[->] (W1Q4) to[in=0,out=180] node[midway,above] {$(\begin{smallmatrix} 1 & 0 \\ 0 & 1 \end{smallmatrix})$} (W1Q3);
	%A_3
	\draw[->] (W1Q5) to[in=0,out=180] node[midway,above] {$(\begin{smallmatrix} 1 & 0 \\ 0 & 1 \end{smallmatrix})$} (W1Q4);
	
	%a_1
	\draw[->] (W1Q2) to node[midway, below right]{$\begin{smallmatrix} 1
	\end{smallmatrix}$} (W1F1);
	%b_1
	\draw[->] (W1F1) to node[midway, below left]{$\begin{smallmatrix} -ac
	\end{smallmatrix}$} (W1Q1);
	%b_2
	\draw[->] (W1Q4) to node[midway, below right=0 and -0.35]{$\begin{smallmatrix} -(ad & bd) \end{smallmatrix}$} (W1F2);
	%a_2
	\draw[->] (W1F2) to node[midway, below left]{$(\begin{smallmatrix} 0 \\ 1 \end{smallmatrix})$} (W1Q3);
	%b_3
	\draw[->] (W1Q5) to node[midway, below right]{$\begin{smallmatrix}  (a & b) \end{smallmatrix}$} (W1F3);
	%a_3
	\draw[->] (W1F3) to node[midway, right=0.15]{$(\begin{smallmatrix} c \\ d \end{smallmatrix})$} (W1Q4);
	
	\node (W2) at (0,-4) {$W_2:$};
	\node (W2Q1) at (2,-5) {$\mathbb C$};
	\node (W2Q2) at (6.5,-5) {$\mathbb C$};
	\node (W2Q3) at (11,-5) {$\mathbb C^2$};
	\node (W2Q4) at (15.5,-5) {$\mathbb C^2$};
	\node (W2Q5) at (20,-5) {$\mathbb C^2$};
	\node (W2F1) at (4.25,-7) {$\mathbb C$};
	\node (W2F2) at (13.25,-7) {$\mathbb C$};
	\node (W2F3) at (17.75,-7) {$\mathbb C$};
	
	\draw[->] (W2Q2) to[in=120, out=60,looseness=10] node[midway,above] {$\begin{smallmatrix}-ac\end{smallmatrix}$} (W2Q2);
	\draw[->] (W2Q3) to[in=120, out=60,looseness=10] node[midway,above] {$ \begin{smallmatrix}-\end{smallmatrix}(\begin{smallmatrix} ac & bc \\ 0& 0 \end{smallmatrix})$} (W2Q3);
	\draw[->] (W2Q4) to[in=120, out=60,looseness=10] node[midway,above] {$(\begin{smallmatrix} 0 & 0 \\ ad & bd \end{smallmatrix})$} (W2Q4);
	\draw[->] (W2Q2) to[in=210, out=330] node[midway,below] {$(\begin{smallmatrix} 1 \\ 0 \end{smallmatrix})$} (W2Q3);
	\draw[->] (W2Q3) to[in=30, out=150] node[midway,above] {$\begin{smallmatrix}( ac & bc) \end{smallmatrix}$}(W2Q2);
	\draw[->] (W2Q2) to[in=0,out=180] node[midway,above] {$\begin{smallmatrix}1\end{smallmatrix}$} (W2Q1);
	\draw[->] (W2Q4) to[in=0,out=180] node[midway,above] {$(\begin{smallmatrix} 1 & 0 \\ 0 & 1 \end{smallmatrix})$} (W2Q3);
	\draw[->] (W2Q5) to[in=0,out=180] node[midway,above] {$(\begin{smallmatrix} 1 & 0 \\ 0 & 1 \end{smallmatrix})$} (W2Q4);
	
	%b_1
	\draw[->] (W2Q2) to node[midway, below right]{$\begin{smallmatrix}1\end{smallmatrix}$} (W2F1);
	%a_1
	\draw[->] (W2F1) to node[midway, below left]{$ \begin{smallmatrix}-ac\end{smallmatrix}$} (W2Q1);
	%b_2
	\draw[->] (W2Q4) to node[midway,below right]{$\begin{smallmatrix} (a & b) \end{smallmatrix}$} (W2F2);
	%a_2
	\draw[->] (W2F2) to node[midway, below left]{$(\begin{smallmatrix} c \\ d \end{smallmatrix})$} (W2Q3);
	%b_3
	\draw[->] (W2Q5) to node[midway, below right]{$\begin{smallmatrix}  -(ad & bd) \end{smallmatrix}$} (W2F3);
	%a_3
	\draw[->] (W2F3) to node[midway, right=0.15]{$(\begin{smallmatrix} 0 \\ 1 \end{smallmatrix})$} (W2Q4);
	
	\node (W3) at (0,-10) {$W_3:$};
	\node (W3Q1) at (2,-11) {$\mathbb C$};
	\node (W3Q2) at (6.5,-11) {$\mathbb C$};
	\node (W3Q3) at (11,-11) {$\mathbb C^2$};
	\node (W3Q4) at (15.5,-11) {$\mathbb C^2$};
	\node (W3Q5) at (20,-11) {$\mathbb C^2$};
	\node (W3F1) at (4.25,-13) {$\mathbb C$};
	\node (W3F2) at (13.25,-13) {$\mathbb C$};
	\node (W3F3) at (17.75,-13) {$\mathbb C$};
	
	%B_1^+
	\draw[->] (W3Q2) to[in=120, out=60,looseness=10] node[midway,above] {$\begin{smallmatrix}-ac-bd\end{smallmatrix}$} (W3Q2);
	%B_2^-	
	\draw[->] (W3Q3) to[in=120, out=60,looseness=10] node[midway,above] {$-(\begin{smallmatrix} ac & ad \\ bc & bd \end{smallmatrix})$} (W3Q3);
	%B_3^-
	\draw[->] (W3Q4) to[in=120, out=60,looseness=10] node[midway,above] {$-(\begin{smallmatrix} o & o \\ bc & bd \end{smallmatrix})$} (W3Q4);
	%D_2
	\draw[->] (W3Q2) to[in=210, out=330] node[midway,below] {$(\begin{smallmatrix} a \\ b \end{smallmatrix})$} (W3Q3);
	%C_2
	\draw[->] (W3Q3) to[in=30, out=150] node[midway,above] {$\begin{smallmatrix} (c & d) \end{smallmatrix}$}(W3Q2);
	%A_1
	\draw[->] (W3Q2) to[in=0,out=180] node[midway,above] {$\begin{smallmatrix}1\end{smallmatrix}$} (W3Q1);
	%A_2
	\draw[->] (W3Q4) to[in=0,out=180] node[midway,above] {$(\begin{smallmatrix} 1 & 0 \\ 0 & 1 \end{smallmatrix})$} (W3Q3);
	%A_3
	\draw[->] (W3Q5) to[in=0,out=180] node[midway,above] {$(\begin{smallmatrix} 1 & 0 \\ 0 & 1 \end{smallmatrix})$} (W3Q4);
	
	%b_1
	\draw[->] (W3Q2) to node[midway, below right=0 and -0.3]{$\begin{smallmatrix}-ac\end{smallmatrix}$} (W3F1);
	%a_1
	\draw[->] (W3F1) to node[midway, below left]{$\begin{smallmatrix}1\end{smallmatrix}$} (W3Q1);
	%b_2
	\draw[->] (W3Q4) to node[midway,below right]{$\begin{smallmatrix} (ac & ad) \end{smallmatrix}$} (W3F2);
	%a_2
	\draw[->] (W3F2) to node[midway, below left]{$(\begin{smallmatrix} 1 \\ 0 \end{smallmatrix})$} (W3Q3);
	%b_3
	\draw[->] (W3Q5) to node[midway, below right]{$\begin{smallmatrix}  -(bc & bd) \end{smallmatrix}$} (W3F3);
	%a_3
	\draw[->] (W3F3) to node[midway, right=0.15]{$(\begin{smallmatrix} 0 \\ 1 \end{smallmatrix})$} (W3Q4);
	
	\node (W4) at (0,-16) {$W_4:$};
	\node (W4Q1) at (2,-17) {$\mathbb C$};
	\node (W4Q2) at (6.5,-17) {$\mathbb C$};
	\node (W4Q3) at (11,-17) {$\mathbb C^2$};
	\node (W4Q4) at (15.5,-17) {$\mathbb C^2$};
	\node (W4Q5) at (20,-17) {$\mathbb C^2$};
	\node (W4F1) at (4.25,-19) {$\mathbb C$};
	\node (W4F2) at (13.25,-19) {$\mathbb C$};
	\node (W4F3) at (17.75,-19) {$\mathbb C$};
	
	%B_1^+
	\draw[->] (W4Q2) to[in=120, out=60,looseness=10] node[midway,above] {$\begin{smallmatrix} ab \end{smallmatrix}$} (W4Q2);
	%B_2^-	
	\draw[->] (W4Q3) to[in=120, out=60,looseness=10] node[midway,above] {$(\begin{smallmatrix} ab+dc & (ab+dc)c \\ -d& -dc \end{smallmatrix})$} (W4Q3);
	%B_3^-
	\draw[->] (W4Q4) to[in=120, out=60,looseness=10] node[midway,above] {$ \begin{smallmatrix}-\end{smallmatrix}(\begin{smallmatrix} 0 & 0 \\ d & dc \end{smallmatrix})$} (W4Q4);
	%D_2
	\draw[->] (W4Q2) to[in=210, out=330] node[midway,below] {$(\begin{smallmatrix} -ab-dc \\ d \end{smallmatrix})$} (W4Q3);
	%C_2
	\draw[->] (W4Q3) to[in=30, out=150] node[midway,above] {$\begin{smallmatrix} (1 & c) \end{smallmatrix}$}(W4Q2);
	%A_1
	\draw[->] (W4Q2) to[in=0,out=180] node[midway,above] {$\begin{smallmatrix}1\end{smallmatrix}$} (W4Q1);
	%A_2
	\draw[->] (W4Q4) to[in=0,out=180] node[midway,above] {$(\begin{smallmatrix} 1 & 0 \\ 0 & 1 \end{smallmatrix})$} (W4Q3);
	%A_3
	\draw[->] (W4Q5) to[in=0,out=180] node[midway,above] {$(\begin{smallmatrix} 1 & 0 \\ 0 & 1 \end{smallmatrix})$} (W4Q4);
	
	%b_1
	\draw[->] (W4Q2) to node[midway, below right]{$\begin{smallmatrix}b\end{smallmatrix}$} (W4F1);
	%a_1
	\draw[->] (W4F1) to node[midway, below left]{$\begin{smallmatrix}a\end{smallmatrix}$} (W4Q1);
	%b_2
	\draw[->] (W4Q4) to node[midway,below right = 0.1 and -0.45]{$\begin{smallmatrix}-(ab+dc)\end{smallmatrix}\begin{smallmatrix}( 1 & c )\end{smallmatrix}$} (W4F2);
	%a_2
	\draw[->] (W4F2) to node[midway, below left]{$(\begin{smallmatrix} 1 \\ 0 \end{smallmatrix})$} (W4Q3);
	%b_3
	\draw[->] (W4Q5) to node[midway, below right]{$\begin{smallmatrix} -( d & dc) \end{smallmatrix}$} (W4F3);
	%a_3
	\draw[->] (W4F3) to node[midway, right=0.15]{$(\begin{smallmatrix} 0 \\ 1 \end{smallmatrix})$} (W4Q4);
	
	\node (W5) at (0,-22) {$W_5:$};
	\node (W5Q1) at (2,-23) {$\mathbb C$};
	\node (W5Q2) at (6.5,-23) {$\mathbb C$};
	\node (W5Q3) at (11,-23) {$\mathbb C^2$};
	\node (W5Q4) at (15.5,-23) {$\mathbb C^2$};
	\node (W5Q5) at (20,-23) {$\mathbb C^2$};
	\node (W5F1) at (4.25,-25) {$\mathbb C$};
	\node (W5F2) at (13.25,-25) {$\mathbb C$};
	\node (W5F3) at (17.75,-25) {$\mathbb C$};
	
	%B_1^+
	\draw[->] (W5Q2) to[in=120, out=60,looseness=10] node[midway,above] {$\begin{smallmatrix}ab\end{smallmatrix}$} (W5Q2);
	%B_2^-	
	\draw[->] (W5Q3) to[in=120, out=60,looseness=10] node[midway,above] {$(\begin{smallmatrix} -dc & -d \\ (ab+dc)c& ab+dc\end{smallmatrix})$} (W5Q3);
	%B_3^-
	\draw[->] (W5Q4) to[in=120, out=60,looseness=10] node[midway,above] {$(\begin{smallmatrix} 0 & 0 \\ (ab+dc)c & ab+dc \end{smallmatrix})$} (W5Q4);
	%D_2
	\draw[->] (W5Q2) to[in=210, out=330] node[midway,below] {$(\begin{smallmatrix} d \\ -ab-dc \end{smallmatrix})$} (W5Q3);
	%C_2
	\draw[->] (W5Q3) to[in=30, out=150] node[midway,above] {$\begin{smallmatrix} (c & 1) \end{smallmatrix}$}(W5Q2);
	%A_1
	\draw[->] (W5Q2) to[in=0,out=180] node[midway,above] {$\begin{smallmatrix}1\end{smallmatrix}$} (W5Q1);
	%A_2
	\draw[->] (W5Q4) to[in=0,out=180] node[midway,above] {$(\begin{smallmatrix} 1 & 0 \\ 0 & 1 \end{smallmatrix})$} (W5Q3);
	%A_3
	\draw[->] (W5Q5) to[in=0,out=180] node[midway,above] {$(\begin{smallmatrix} 1 & 0 \\ 0 & 1 \end{smallmatrix})$} (W5Q4);
	
	%a_1
	\draw[->] (W5Q2) to node[midway, below right]{$\begin{smallmatrix}a\end{smallmatrix}$} (W5F1);
	%b_1
	\draw[->] (W5F1) to node[midway, below left]{$\begin{smallmatrix}b\end{smallmatrix}$} (W5Q1);
	%b_2
	\draw[->] (W5Q4) to node[midway,below right]{$\begin{smallmatrix} (dc & d) \end{smallmatrix}$} (W5F2);
	%a_2
	\draw[->] (W5F2) to node[midway, below left]{$(\begin{smallmatrix} 0 \\ 1 \end{smallmatrix})$} (W5Q3);
	%b_3
	\draw[->] (W5Q5) to node[midway, below right= 0.1 and -0.45]{$\begin{smallmatrix}-(ab+dc)\end{smallmatrix}\begin{smallmatrix} ( 1 & c) \end{smallmatrix}$} (W5F3);
	%a_3
	\draw[->] (W5F3) to node[midway, right=0.15]{$(\begin{smallmatrix} c \\ d \end{smallmatrix})$} (W5Q4);
	\end{tikzpicture}
	\caption{Parametrizations of the $\mathbb T$-invariant affine open neighborhoods $W_1,\ldots,W_5$ of the $\mathbb T$-fixed points $x_{D_1},\ldots,x_{D_5}$. Here, $a,b,c,d\in \mathbb C$.}\label{fig:TInvariantNeighborhoods}
	\end{figure}
	
	\begin{proof}
	We only prove the case $i=1$ since the other cases can be proved analogously. Let $\tilde\eta_1:\mathbb C^4\rightarrow m^{-1}(0)^{\mathrm{s}}$ be the morphism of varieties which maps a point $(a,b,c,d)\in\mathbb C^4$ to the class displayed in the diagram:
	\[
	\begin{tikzpicture}[scale=.7]
	\node (W1Q1) at (2,1) {$\mathbb C$};
	\node (W1Q2) at (6.5,1) {$\mathbb C$};
	\node (W1Q3) at (11,1) {$\mathbb C^2$};
	\node (W1Q4) at (15.5,1) {$\mathbb C^2$};
	\node (W1Q5) at (20,1) {$\mathbb C^2$};
	\node (W1F1) at (4.25,-1) {$\mathbb C$};
	\node (W1F2) at (13.25,-1) {$\mathbb C$};
	\node (W1F3) at (17.75,-1) {$\mathbb C$};
	
	%B_1^+
	\draw[->] (W1Q2) to[in=120, out=60,looseness=10] node[midway,above] {$\begin{smallmatrix} -ac
	\end{smallmatrix}$} (W1Q2);
	%B_2^-	
	\draw[->] (W1Q3) to[in=120, out=60,looseness=10] node[midway,above] {$\begin{smallmatrix}-\end{smallmatrix}(\begin{smallmatrix} ac & bc \\ 0& 0 \end{smallmatrix})$} (W1Q3);
	%B_3^-
	\draw[->] (W1Q4) to[in=120, out=60,looseness=10] node[midway,above] {$\begin{smallmatrix}-\end{smallmatrix}(\begin{smallmatrix} ac & bc \\ ad & bd \end{smallmatrix})$} (W1Q4);
	%D_2
	\draw[->] (W1Q2) to[in=210, out=330] node[midway,below] {$(\begin{smallmatrix} 1 \\ 0 \end{smallmatrix})$} (W1Q3);
	%C_2
	\draw[->] (W1Q3) to[in=30, out=150] node[midway,above] {$\begin{smallmatrix} (ac & bc) \end{smallmatrix}$}(W1Q2);
	%A_1
	\draw[->] (W1Q2) to[in=0,out=180] node[midway,above] {$\begin{smallmatrix} 1
	\end{smallmatrix}$} (W1Q1);
	%A_2
	\draw[->] (W1Q4) to[in=0,out=180] node[midway,above] {$(\begin{smallmatrix} 1 & 0 \\ 0 & 1 \end{smallmatrix})$} (W1Q3);
	%A_3
	\draw[->] (W1Q5) to[in=0,out=180] node[midway,above] {$(\begin{smallmatrix} 1 & 0 \\ 0 & 1 \end{smallmatrix})$} (W1Q4);
	
	%a_1
	\draw[->] (W1Q2) to node[midway, below right]{$\begin{smallmatrix} 1
	\end{smallmatrix}$} (W1F1);
	%b_1
	\draw[->] (W1F1) to node[midway, below left]{$\begin{smallmatrix} -ac
	\end{smallmatrix}$} (W1Q1);
	%b_2
	\draw[->] (W1Q4) to node[midway, below right=0 and -0.35]{$\begin{smallmatrix} -(ad & bd) \end{smallmatrix}$} (W1F2);
	%a_2
	\draw[->] (W1F2) to node[midway, below left]{$(\begin{smallmatrix} 0 \\ 1 \end{smallmatrix})$} (W1Q3);
	%b_3
	\draw[->] (W1Q5) to node[midway, below right]{$\begin{smallmatrix}  (a & b) \end{smallmatrix}$} (W1F3);
	%a_3
	\draw[->] (W1F3) to node[midway, right=0.15]{$(\begin{smallmatrix} c \\ d \end{smallmatrix})$} (W1Q4);
	\end{tikzpicture}
	\]
	Let $\eta_1:\mathbb C^4\rightarrow \mathcal C(\mathcal D)$ be the induced morphism. Clearly, $\mathrm{Im}(\eta_i)\subset W_i$. Conversely, given 
	\[
	x=[A_1,A_2,A_3,B_1^+,B_2^-,B_3^-,C_2,D_2,a_1,a_2,a_3,b_1,b_2,b_3] \in W_1,
	\]
	we may assume by the defining conditions of $W_1$ that 
	\[
	A_1=1, \quad A_2=A_3=\begin{pmatrix} 1 & 0 \\ 0 & 1 \end{pmatrix},\quad D_2=\begin{pmatrix} 1\\0 \end{pmatrix},\quad a_2=\begin{pmatrix} 0\\1 \end{pmatrix}.
	\]
	Let $b_3=(a\;b)$ and $a_3=(\begin{smallmatrix} c \\ d\end{smallmatrix})$. Then,~\eqref{eq:nonMomentMap} implies 
	$
	B_3^-= -(\begin{smallmatrix}
	ac & bc \\  ad & db
	\end{smallmatrix})
	$. Moreover, let $b_2=(x\;y)$ and $D_2=(\begin{smallmatrix} z \\ w\end{smallmatrix})$. By~\eqref{eq:nonMomentMap} and the moment map equation, we deduce
	\[
	-\begin{pmatrix}
	ac & bc \\ ad+x & db +y
	\end{pmatrix}
	= B_2^- = \begin{pmatrix}
	-z & -w \\ 0 & 0
	\end{pmatrix}.
	\]
	Hence, $x=-ad$, $y=-bd$, $z=ac$ and $w=bc$. Finally, ~\eqref{eq:nonMomentMap} and the moment map equation also gives $b_1=B_1^+=-z=-ac$ which proves that $x\in \mathrm{Im}(\eta_1)$ and thus, $\mathrm{Im}(\eta_i)=W_i$.
	
	To show that $\eta_1$ is an isomorphism, it suffices to show that $\eta_1$ is injective since $\mathbb C^4$ and $W_1$ are both smooth. Assume 
	$\eta_1(a,b,c,d)=\eta_1(a',b',c',d')$, so there exists $g=(g_1,g_2,g_3,g_4,g_5)\in \mathcal G$ such that $g.\tilde\eta_1(a,b,c,d)=\tilde\eta_1(a',b',c',d')$. This directly implies $g_1=g_2=1$ and $g_3=g_4=g_5$. In addition, the conditions $g_3(\begin{smallmatrix} 1 \\ 0 \end{smallmatrix})g_2^{-1}=(\begin{smallmatrix} 1 \\ 0 \end{smallmatrix})$ and $g_3(\begin{smallmatrix} 0 \\ 1 \end{smallmatrix})=(\begin{smallmatrix} 1 \\ 0 \end{smallmatrix})$ imply that $g_3$ is the identity matrix. Hence, $\tilde\eta_1(a,b,c,d)=\tilde\eta_1(a',b',c',d')$ which gives $(a,b,c,d)=(a',b',c',d')$. Thus, $\eta_1$ is injective and therefore an isomorphism. 
	\end{proof}
	\begin{center}
	\begin{table}
\begin{tabular}{ | c | c | c | c | c |}
 \hline 
 \diagbox{Open affine}{Coordinate axis} & $a$ & $b$ & $c$ & $d$ \\
 \hline
 \hline
 $W_1$ & $t_3-t_1+h$ & $t_3-t_2+h$ & $t_1-t_3$ & $t_2-t_3$ \\
 \hline 
 $W_2$ & $t_2-t_1+h$ & $t_2-t_3+h$ & $t_1-t_2$ & $t_3-t_2$ \\
 \hline
 $W_3$ & $t_2-t_1$ & $t_3-t_1$ & $t_1-t_2+h$ & $t_2-t_3+h$ \\
 \hline
 $W_4$ & $t_2-t_1-h$ & $t_1-t_2+2h$ & $t_2-t_3$ & $t_2-t_3+h$\\
 \hline
 $W_5$ & $t_3-t_1-h$ & $t_1-t_3+2h$ & $t_3-t_2 $ & $t_2-t_3+h$ \\
 \hline  
\end{tabular}
\caption{$\mathbb T$-weight space decomposition of $W_1,\ldots,W_5$.}
    \label{table:WeightSpaceDecompositionWi} 
\end{table}
\end{center}
The affine covering \eqref{eq:affcov} allows us now to get a hand on the attracting cells. One can directly check that via $\eta_i$ we obtain a linear $\mathbb T$-action on $\mathbb C^4$ and the standard basis vectors are weight vectors. The respective weights are recorded in  Table~\ref{table:WeightSpaceDecompositionWi}.
	For a given generic cocharacter $\sigma:\mathbb C^\ast\rightarrow \mathbb A$ the Ax--Grothendieck theorem provides
	\[
	W_{i,\sigma}^+= W_i\cap \operatorname{Attr}_{\sigma}(x_{D_i}) = \operatorname{Attr}_{\sigma}(x_{D_i}),\quad i=1,\ldots,5.
	\]
	Therefore, we can easily read off the attracting cells from Table~\ref{table:WeightSpaceDecompositionWi}. If we take for instance the cocharacter $\sigma_0(t)=(t,t^2,t^3)$ then the corresponding attracting cells are given by
	\begin{gather*}
	\operatorname{Attr}_{\sigma_0}(x_{D_1})= \mathbb C_{t_3-t_1+h}\oplus \mathbb C_{t_3-t_2+h},\quad 
	\operatorname{Attr}_{\sigma_0}(x_{D_2})= \mathbb C_{t_2-t_1+h}\oplus \mathbb C_{t_3-t_2},\\
	\operatorname{Attr}_{\sigma_0}(x_{D_3})= \mathbb C_{t_2-t_1}\oplus \mathbb C_{t_3-t_1},\quad
	\operatorname{Attr}_{\sigma_0}(x_{D_4})= \mathbb C_{t_2-t_1-h}\oplus \mathbb C_{t_3-t_2+h},\\
	\operatorname{Attr}_{\sigma_0}(x_{D_5})= \mathbb C_{t_3-t_1-h}\oplus \mathbb C_{t_3-t_2}.
	\end{gather*}

We leave it to the reader to consider other choices of cocharacters. We return now to the general framework and will see that the attracting cells are in fact constant along certain chambers inside the space of cocharacters.	
	\subsection{Independence of chamber}
	Let $\Lambda$ be the cocharacter lattice of $\mathbb A$ and consider the vector space $\Lambda_{\mathbb R}:= \Lambda\otimes_{\mathbb Z}\mathbb R$. For $1\le i,j\le N$ with $i\ne j$ we define the following hyperplanes:
	\[
	H_{i,j}:= \{(t_1,\ldots,t_N)\mid t_i=t_j\} \subset \Lambda_{\mathbb R}.
	\]
	The connected components of 
	\[
	\Lambda_{\mathbb R} \setminus\Big( \bigcup_{\substack{1\le i,j\le N\\ i\ne j}} H_{i,j}\Big)
	\]
	are called \textit{chambers}. There is a (well-known from Lie theory) one-to-one correspondence
	\[
	\begin{tikzcd}
	\{\textup{Chambers}\}\arrow[r,leftrightarrow, "1:1"]& S_N,
	\end{tikzcd}
	\]
	 where we assign to a permutation $\pi\in\mathfrak S_n$ the connected component
	\[
	\mathfrak C_{\pi}:= \{(t_1,\ldots, t_N)\mid t_{\pi(1)}> t_{\pi(2)}>\ldots>t_{\pi(N)}\}.
	\]
	\begin{remark} This correspondence  allows to connect the chambers with the combinatorics of the symmetric group. Note moreover, the parallel to the more Lie theoretic description of attracting cells, a.k.a. Schubert cells, of Grassmannians, see \cite{gorbounov2020yang}. For readers new to the subject it might be helpful to keep in the following this analogous framework in mind.
	\end{remark}
	
	We have the following independence result for attracting cells:
	
	\begin{prop}[Chambers invariance]\label{lemma:INDEPENDENCE} Let $\mathfrak C$ be a chamber and
	 $\sigma,\tau\in \mathfrak a\cap \mathfrak C$. Then 
	 \begin{equation}\label{eq:IndependenceTangentSpaces} T_p\mathcal C(\mathcal D)^+_{\sigma}=T_p\mathcal C(\mathcal D)^+_{\tau},\quad T_p\mathcal C(\mathcal D)^-_{\sigma}=T_p\mathcal C(\mathcal D)^-_{\tau}
	 \end{equation}
	 hold. Moreover, attracting cells are constant along chamber, i.e. for all $p\in \mathcal C(\mathcal D)^{\mathbb T}$ we have 
	 \begin{equation}\label{eq:IndependenceAttractingCells}
	  \operatorname{Attr}_{\sigma}(p)=\operatorname{Attr}_{\tau}(p).
	 \end{equation} 
	\end{prop}
	\begin{proof} Let $\pi \in \mathfrak S_n$ such that $\mathfrak C_{\pi}=\mathfrak C$ and fix a $\mathbb T$-fixed point $p$. Recall from Corollary~\ref{cor:tangentweights} that the $\mathbb A$-weights of $T_p(\mathcal C(\mathcal D))$ are of the form $t_i-t_j$ where $i\ne j$. It follows
	\[
	T_p\mathcal C(\mathcal D)^+_{\sigma} = \bigoplus_{\substack{1\le i,j\le n \\ \pi(i)>\pi(j) }} T_p(\mathcal C(\mathcal D))_{t_i-t_j} = T_p\mathcal C(\mathcal D)^+_{\tau}
	\]
	and
	\[
	T_p\mathcal C(\mathcal D)^-_{\sigma} = \bigoplus_{\substack{1\le i,j\le n \\ \pi(i)<\pi(j) }} T_p(\mathcal C(\mathcal D))_{t_i-t_j} = T_p\mathcal C(\mathcal D)^-_{\tau}.
	\] 
	Thus, we proved~\eqref{eq:IndependenceTangentSpaces}. The equality~\eqref{eq:IndependenceAttractingCells} follows directly from~\eqref{eq:IndependenceTangentSpaces} and Proposition~\ref{prop:affinestructure}. 
	\end{proof}
	
	By Proposition~\ref{lemma:INDEPENDENCE}, the $T_p\mathcal C(\mathcal D)^+_{\sigma},T_p\mathcal C(\mathcal D)^-_{\sigma}$ and $\operatorname{Attr}_{\sigma}(p)$ only depend on the chamber $\mathfrak C$ containing $\sigma$. Thus, we also denote them respectively by  $T_p\mathcal C(\mathcal D)^+_{\mathfrak C},T_p\mathcal C(\mathcal D)^-_{\mathfrak C}$ and $\operatorname{Attr}_{\mathfrak C}(p)$.
	
	\begin{remark} In~\cite{maulik2019quantum}, Maulik and Okounkov defined chambers in a slightly different way. They defined them as connected components of the complement of the union of all hyperplanes orthogonal to the $\mathbb A$-tangent weights of $\mathbb A$-fixed points. Corollary~\ref{cor:tangentweights} implies that the chambers defined in this subsections refine the chambers in the sense of~\cite{maulik2019quantum}. The inclusion may be strict as for instance the bow variety $\mathcal C(0\textcolor{red}{\slash}1\textcolor{red}{\slash}3\textcolor{blue}{\backslash}1\textcolor{blue}{\backslash}0)$ is just a single point. Hence, there exists only a single chamber in the sense of Maulik and Okounkov whereas the chambers defined in this subsection are in one-to-one correspondence with the elements of the symmetric group on two letters. 
	\end{remark}
	
	\subsection{Partial order by attraction}
	
	%In this and the upcoming subsection, we follow closely~\cite[Section~3.2.3]{maulik2019quantum}.
	
	Given a chamber $\mathfrak C$, we define a preorder $\preceq_{\mathfrak C}$ on $\mathcal C(\mathcal D)^{\mathbb T}$ as the transitive closure of the relation
	\begin{equation}\label{eq:DefinitionRelationPartialOrderByAttraction}
	p\in \overline{\operatorname{Attr}_{\mathfrak C}(q)} \Rightarrow p\preceq_{\mathfrak C} q,
	\end{equation}
	where $\overline{\operatorname{Attr}_{\mathfrak C}(q)}$ denotes the Zariski closure of $\operatorname{Attr}_{\mathfrak C}(q)$.
	
	As we usually work with a fixed choice of chamber, we denote $\preceq_{\mathfrak C}$ also just by $\preceq$.
	
	\begin{lemma}[Fixed point ordering]\label{lemma:partialOrdering} The preorder $\preceq$ is a partial order on 
	$\mathcal C(\mathcal D)^{\mathbb T}$.
	\end{lemma}
	\begin{proof} Evidently, the preorder $\preceq$ is reflexive and transitive. Hence, it is left to show that $\preceq$ is antisymmetric.
	Let $p,q\in \mathcal C(\mathcal D)^{\mathbb T}$ with $p\preceq q$ and $q\preceq p$. For the sake of contradiction, assume $p\ne q$. Let $\sigma\in\mathfrak C$ and as in Subsection~\ref{subsection:affinestr}, we fix a smooth $\sigma$-equivariant compactification $\mathcal C(\mathcal D)\hookrightarrow X$. Let $Z_1,\ldots,Z_r$ be the $\mathbb C^\ast$-fixed components of $X$ and $X_1^+,\ldots,X_r^+$ the corresponding attracting cells. By \cite{bialynickibirula1974fixed}, we can order the fixed points components in such a way such that the subsets
	\[
	Y_i:=\bigsqcup_{1\le j \le i}X_j^+ \subset X
	\]
	are Zariski closed. Let $i,j$ such that $Z_i=\{p\}$ and $Z_j=\{q\}$. Without loss of generality, we may assume $i<j$. In particular, we have $q\notin Y_i$. Note that $Y_i\cap \mathcal C(\mathcal D)$ is Zariski closed in $\mathcal C(\mathcal D)$ and stable under attraction. Therefore, $p'\preceq p$ yields $p'\in Y_i$. Thus, we must have $p\in Y_i$ which yields a contradiction. Thus, $p=q$ and hence, $\preceq$ is antisymmetric. 
	\end{proof}
	
	\subsection{Partial order by attraction in a concrete example}\label{subsection:ExamplePartialOrder}
	
	We continue in the setting of Subsection~\ref{subsection:ExampleAttractingCells} and compute the partial order corresponding to the cocharacter $\sigma_0$. For this, we first compute all intersections
	 $\overline{\operatorname{Attr}_{\sigma_0}(x_{D_i})}\cap W_j$. 
	\begin{claim} Given $i,j\in\{1,\ldots,5\}$, the intersection $\overline{\operatorname{Attr}_{\sigma_0}(x_{D_i})}\cap W_j$ is a $\mathbb T$-invariant linear subspace of $W_j$ whose weight space decomposition is recorded in Table~\ref{table:IntersectionAttractionOpenAffine}.
	\end{claim} 
	 \begin{proof} We only prove the case $i=2$ and $j=1$ since all other cases can be proved similarly. From Subsection~\ref{subsection:ExampleAttractingCells} we know 
$\operatorname{Attr}_{\sigma_0}(x_{D_i})=\{\eta_2(a,0,0,d)\mid a,d\in\mathbb C\}$.
	 We have that $\eta_2(a,0,0,d)\in W_1$ if and only if $d\ne 0$. A direct computation shows
	 \[
	 (1,1,(\begin{smallmatrix} 1 & 0 \\ 0 & d^{-1} \end{smallmatrix}),(\begin{smallmatrix} 1 & 0 \\ 0 & d^{-1} \end{smallmatrix}),(\begin{smallmatrix} 1 & 0 \\ 0 & d^{-1} \end{smallmatrix})).\tilde \eta_2(a,0,0,d) = \tilde \eta_1(-a,0,0,d^{-1}),
	 \]
	 which implies $\overline{\operatorname{Attr}_{\sigma_0}(x_{D_2})}\cap W_1= \{\eta_1(a,0,0,d)\mid a,d\in\mathbb C\}$.
	 According to Table~\ref{table:WeightSpaceDecompositionWi}, we have 
	  $\{\eta_1(a,0,0,d)\mid a,d\in\mathbb C\}=\mathbb C_{t_3-t_1+h}\oplus\mathbb C_{t_2-t_3}$.
	 \end{proof}
	Our computations directly imply an isomorphism of partially ordered sets
	\[
	(\{1,2,3,4,5\},\le)\xrightarrow{\phantom{x}\sim\phantom{x}} (\mathcal C(\mathcal D)^{\mathbb T},\preceq),\quad i\mapsto x_{D_i},
	\]
	where $\le$ is the usual ordering on $\{1,2,3,4,5\}$.
	\begin{center}
	\begin{table}
\begin{tabular}{ | c | c | c | c | c | c |}
 \hline 
 \diagbox{$i$}{$j$} & $1$ & $2$ & $3$ & $4$ & $5$ \\
 \hline
 \hline
 $1$ & 
 \tiny $\mathbb C_{t_3-t_1+h}\oplus\mathbb C_{t_3-t_2+h}$ & 
 $\emptyset$ & $\emptyset$ & $\emptyset$ & $\emptyset$ \\
 \hline 
 $2$ & 
 \tiny $\mathbb C_{t_3-t_1+h}\oplus\mathbb C_{t_2-t_3}$ & 
 \tiny $\mathbb C_{t_2-t_1+h}\oplus\mathbb C_{t_3-t_2}$ & 
 $\emptyset$ & $\emptyset$ & $\emptyset$ \\
 \hline
 $3$ & 
 \tiny $\mathbb C_{t_3-t_1+h}\oplus\mathbb C_{t_2-t_3}$ & 
 \tiny $\mathbb C_{t_1-t_2}\oplus\mathbb C_{t_3-t_2}$ & 
 \tiny $\mathbb C_{t_2-t_1}\oplus\mathbb C_{t_3-t_1}$ & 
 $\emptyset$ & $\emptyset$ \\
 \hline
 $4$ & 
 \tiny $\mathbb C_{t_3-t_2+h}\oplus\mathbb C_{t_1-t_3}$ & 
 $\emptyset$ & 
 \tiny $\mathbb C_{t_3-t_1}\oplus\mathbb C_{t_1-t_2+h}$ & 
 \tiny $\mathbb C_{t_2-t_1-h}\oplus\mathbb C_{t_3-t_2+h}$&
 $\emptyset$\\
 \hline
 $5$ & $\emptyset$ &
 \tiny $\mathbb C_{t_2-t_3+h}\oplus\mathbb C_{t_1-t_2}$ &
 \tiny $\mathbb C_{t_1-t_2+h}\oplus\mathbb C_{t_2-t_3+h} $ &
 \tiny $\mathbb C_{t_2-t_1-h}\oplus\mathbb C_{t_2-t_3}$ &
 \tiny $\mathbb C_{t_3-t_1-h}\oplus\mathbb C_{t_3-t_2}$ \\
 \hline  
\end{tabular}
\caption{Intersections $\overline{\operatorname{Attr}_{\sigma_0}(x_{D_i})}\cap W_j$ as subspaces of $W_j$.}
    \label{table:IntersectionAttractionOpenAffine} 
\end{table}
\end{center}
	\subsection{Full attraction cells}	
	In the proof of Lemma~\ref{lemma:partialOrdering}, we used that for a smooth projective variety with a one-parameter torus action, we can order the attracting cells such that the successive unions of the attracting cells are all Zariski closed. Motivated by this general result, we prove an analogous statement for bow varieties.

Let $p\in \mathcal C(\mathcal D)^{\mathbb T}$. The \textit{full attracting cell of $p$ with respect to the chamber $\mathfrak C$} is defined as 
	\[
	\operatorname{Attr}^f_{\mathfrak C}(p) := \bigsqcup_{q \preceq p}\operatorname{Attr}_{\mathfrak C}(q). 
	\]
	The main result of this subsection is the following result.
	\begin{prop}\label{prop:fullattractionunion} The full attracting cell
	$
	\operatorname{Attr}^f_{\mathfrak C}(p)
	$ is Zariski closed in $\mathcal C(\mathcal D)$.
	\end{prop}
	
	%The statement was proved in~\cite[Lemma~3.2.7]{maulik2019quantum} in the case of Nakajima quiver varieties and we use an analogous argument to prove Proposition~\ref{prop:fullattractionunion}. 
	At first we state an immediate consequence of the valuative criterion for properness:
	
	\begin{lemma}\label{lemma:limitexistence} Let $X,X'$ be algebraic varieties with $\mathbb C^\ast$-actions $\rho,\rho'$ and let $f:X\rightarrow X'$ be a proper $\mathbb C^\ast$-equivariant morphism. Then, we have
	\[
		\{x\in X\mid \lim_{t\to0}\rho(t).x\textit{ exists in }X\} = f^{-1}(\{x'\in X'\mid \lim_{t\to0}\rho'(t).x'\textit{ exists in }X'\}).
	\]
	\end{lemma}
	The construction of $\mathcal C(\mathcal D)$ as GIT quotient implies the next statement.
	\begin{prop}\label{prop:properaffine} There exists a proper $\mathbb T$-equivariant morphism $\mathcal C(\mathcal D)\rightarrow V$, where $V$ is a finite dimensional $\mathbb T$-representation.
	\end{prop}
	\begin{proof} Since the categorical quotient $m^{-1}(0)/\!\!/\mathcal G$ is an affine $\mathbb T$-variety, there exists a finite dimensional $\mathbb C$-vector space $V$ with a linear $\mathbb T$-action such that there is a $\mathbb T$-equivariant closed immersion $m^{-1}(0)/\!\!/\mathcal G\hookrightarrow V$. The inclusion $m^{-1}(0)^{\mathrm{ss}}\subset m^{-1}(0)$ induces a projective $\mathbb T$-equivariant morphism $\mathcal C(\mathcal D)\rightarrow m^{-1}(0)/\!\!/\mathcal G$. Thus, the composition
$\mathcal C(\mathcal D)\rightarrow m^{-1}(0)/\!\!/\mathcal G\hookrightarrow V$ is a proper $\mathbb T$-equivariant morphism.
	\end{proof}

%	\begin{lemma}\label{lemma:limitexistence} Let $X$ be an algebraic variety with a $\mathbb C^\ast$-action $\rho$ and let $W$ be a finite dimensional $\mathbb C$-vector space with a linear $\mathbb C^\ast$-action $\rho'$. Suppose there is a proper $\mathbb C^\ast$-equivariant morphism $f:X\rightarrow W$. Then, we have
%	\[
%		\{x\in X\mid \lim_{t\to0}\rho(t).x\textit{ exists in }X\} = f^{-1}(W^{\geq0}).
%	\]
%	\end{lemma}
%	\begin{proof} The inclusion $\supset$ follows directly from 
%	\[
%		W^{\geq0}=\{w\in W\mid \lim_{t\to0}\rho'(t).w\textup{ exists in }W\}. 
%	\] 
%	For the converse inclusion, note that $f$ is a proper morphism of varieties if and only if $f$ is a proper map with respect to the analytic topology on $X$ and $W$ in the usual topological sense, i.e. preimages of compact subspaces under $f$ are again compact. See e.g. \cite[Proposition~XII.3.1]{SGA1} for a proof. Since $X$ and $W$ are locally compact in the analytic topology, we know that if $(x_n)_{n\in\mathbb N}$ is a sequence in $X$ such that $(f(x_n))_{n\in\mathbb N}$ converges in $W$ then $(x_n)_{n\in\mathbb N}$ has a convergent subsequence in $X$. This implies that for given $x\in X$ the limit $\lim_{t\to0}\rho(t).x$ exists in $X$ if and only if the limit $\lim_{t\to0}\rho'(t).f(x)$ exists in $W$ which is equivalent to $f(x)\in W^{\geq0}$. This proves the inclusion $\subset$.
%	\end{proof}	
	
	\begin{proof}[Proof of Proposition~\ref{prop:fullattractionunion}]
	To prove that $\operatorname{Attr}_{\mathfrak C}^f$ is closed, we show
	\begin{equation}\label{eq:fullAttr}
		\operatorname{Attr}_{\mathfrak C}^f(p) = \bigcup_{q\preceq p} \overline{\operatorname{Attr}_{\mathfrak C}(q)}.
	\end{equation}
	The inclusion $\subset$ is clear. For the converse, fix $q\in\mathcal C(\mathcal D)$ with $q\preceq p$ and let $x\in \overline{\operatorname{Attr}_{\mathfrak C}(q)}$. In addition, we fix $\sigma\in\mathfrak C$. According to Proposition~\ref{prop:properaffine}, there exists a finite dimensional $\mathbb C$-vector space $V$ with a linear $\mathbb C^\ast$-action such that there is a proper $\sigma$-equivariant morphism $\pi:\mathcal C(\mathcal D)\rightarrow V$. Since $\pi(\operatorname{Attr}_{\mathfrak C}(q))\subset V^{\geq 0}$, we have
	$\overline{\operatorname{Attr}_{\mathfrak C}(q)}\subset \pi^{-1}(V^{\geq0})$. Hence, by Lemma~\ref{lemma:limitexistence}, the limit $\lim_{t\to0}\sigma(t).x$ exists in $\mathcal C(\mathcal D)$. Let $q':=\lim_{t\to0}\sigma(t).x$. As $\operatorname{Attr}_{\mathfrak C}(q)$ is $\sigma$-invariant, $\overline{\operatorname{Attr}_{\mathfrak C}(q)}$ is also $\sigma$-invariant and hence the $\mathbb C^\ast$-orbit of $x$ is contained in $\overline{\operatorname{Attr}_{\mathfrak C}(q)}$. Therefore, also the limit point $q'$ is contained in $\overline{\operatorname{Attr}_{\mathfrak C}(q)}$ which gives $q'\preceq q$. It follows that $x\in \operatorname{Attr}_{\mathfrak C}^f(p)$ which completes the proof.
	\end{proof}
	
	\subsection{Properness of intersections between opposite cells} In this subsection we define opposite attracting cells, in analogy to opposite Schubert cells.
	\begin{definition}
	The \textit{opposite chamber} of $\mathfrak C$ is defined as
	\[
		\mathfrak C^{\operatorname{op}}:= \{a\in \mathfrak a_{\mathbb R} \mid -a \in \mathfrak C\}.
	\]
	\end{definition}
	Note that $\sigma\in\mathfrak C$ if and only if $\sigma^{-1}\in\mathfrak C^{\operatorname{op}}$.
	The partial order $\preceq_{\mathfrak C^{\operatorname{op}}}$ is in fact opposite to the partial order of  $\preceq_{\mathfrak C}$. More precisely we have the following result which we prove in Appendix~\ref{section:Appendix}  using an analytic limit argument.
	 
	 \begin{prop}[Opposite order]\label{prop:OppositePartialOrdering} Let $p,q\in \mathcal C(\mathcal D)^{\mathbb T}$. Then, $p\preceq_{\mathfrak C} q$ if and only if $q\preceq_{\mathfrak C^{\operatorname{op}}} p$.
	 \end{prop}
	
	In general, the closure of the attracting cells to $\mathfrak C$ or to $\mathfrak C^{\operatorname{op}}$ need not be proper. However, their intersection is always proper: %A more general version of this result is contained in~\cite[Proof of Theorem~4.4.1]{maulik2019quantum}.

	\begin{theorem}[Properness]\label{thm:slopesupport} Let $p,q\in \mathcal C(\mathcal D)^{\mathbb T}$, then $\overline{\operatorname{Attr}_{\mathfrak C}(p)} \cap \overline{\operatorname{Attr}_{{\mathfrak C}^{\operatorname{op}}}(q)}$ is proper over $\mathbb C$.
	
%	 Then, $\operatorname{Attr}_{\mathfrak C}^f(p) \cap \operatorname{Attr}_{{\mathfrak C}^{\operatorname{op}}}^f(q)$ is proper over $\mathbb C$.
	\end{theorem}
	
	We immediately conclude, using \eqref{eq:fullAttr}, the analogous result for full attracting cells:
	
	\begin{cor}\label{cor:slopesupport} Let $p,q\in \mathcal C(\mathcal D)^{\mathbb T}$.
	 Then, $\operatorname{Attr}_{\mathfrak C}^f(p) \cap \operatorname{Attr}_{{\mathfrak C}^{\operatorname{op}}}^f(q)$ is proper over $\mathbb C$.
	\end{cor}
	
	For the proof, we set up some notation: as in the proof of Proposition~\ref{prop:fullattractionunion}, we pick a cocharacter $\sigma\in \mathfrak C$ and a proper 
	$\sigma$-equivariant morphism $\pi:\mathcal C(\mathcal D)\rightarrow V$ to a finite dimensional $\mathbb C$-vector space $V$ with a linear $\mathbb C^\ast$-action.
	Let $\operatorname{pr}_0:V\rightarrow V^0$ be the linear projection according to the direct sum decomposition $V=V^-\oplus V^0\oplus V^+$ and set $\bar\pi:=\operatorname{pr}_0\circ \pi$. Note that we have $\bar\pi(p)=\pi(p)$ for all $p\in \mathcal C(\mathcal D)^{\mathbb T}$. We first establish a technical tool:
	
	\begin{lemma}\label{lemma:projectionAttr} Let $p\in\mathcal C(\mathcal D)^{\mathbb T}$, $v:=\pi(p)\in V^0$ and $\mathfrak C'\in\{\mathfrak C,\mathfrak C^{\mathrm{op}}\}$. If $q\in \mathcal C(\mathcal D)^{\mathbb T}\cap \overline{\operatorname{Attr}_{\mathfrak C'}(p)}$ then $\overline{\operatorname{Attr}_{\mathfrak C'}(q)}\subset \bar\pi^{-1}(v)$.
%	\begin{enumerate}[label=(\roman*)]
%		\item\label{item:ProjectionAttraction1} For all $q\in \mathcal C(\mathcal D)^{\mathbb T}\cap \overline{\operatorname{Attr}_{\mathfrak C}(p)}$ we have $\overline{\operatorname{Attr}_{\mathfrak C}(q)}\subset \bar\pi^{-1}(v)$.
%		\item\label{item:ProjectionAttraction2} We have $\operatorname{Attr}_{\mathfrak C}^f(p)\subset \bar\pi^{-1}(v)$.
%	\end{enumerate}	 
%	 The assertions~\ref{item:ProjectionAttraction1} and~\ref{item:ProjectionAttraction2} remain true if we replace $\mathfrak C$ by $\mathfrak C^{\operatorname{op}}$.
	 \end{lemma}
	\begin{proof} %The assertion~\ref{item:ProjectionAttraction2} follows from~\ref{item:ProjectionAttraction1} by the definition of $\preceq$. For~\ref{item:ProjectionAttraction1}, note that 
	We only prove the case $\mathfrak C'=\mathfrak C$ as the proof for $\mathfrak C'=\mathfrak C^{\mathrm{op}}$ is analogous. Since $\bar\pi$ is $\sigma$-equivariant, we have $\operatorname{Attr}_{\mathfrak C}(p)\subset \bar\pi^{-1}(v)$ and hence $\overline{\operatorname{Attr}_{\mathfrak C}(p)}\subset \bar\pi^{-1}(v)$. Thus, $\pi(q)=v$. Using again that $\bar\pi$ is $\sigma$-equivriant, we conclude $\overline{\operatorname{Attr}_{\mathfrak C}(q)}\subset \bar\pi^{-1}(v)$.% which gives~\ref{item:ProjectionAttraction1}. 
\end{proof}
	
	%After these observations, we now give a proof of Proposition~\ref{prop:slopesupport}.
	
	\begin{proof}[Proof of Theorem~\ref{thm:slopesupport}] 
	As above, we set $v:=\pi(p)$. If $\pi(q)\ne v$ then Lemma~\ref{lemma:projectionAttr} implies \[ \overline{\operatorname{Attr}_{\mathfrak C}(p)} \cap \overline{\operatorname{Attr}_{{\mathfrak C}^{\operatorname{op}}}(q)} \subset \bar\pi^{-1}(v) \cap \bar\pi^{-1}(\pi(q))=\emptyset. \]
	So let us assume $\pi(q)=v$. Since
	$\overline{\operatorname{Attr}_{\mathfrak C}(p)} \subset \pi^{-1}(V^{\geq 0})$ and $\overline{\operatorname{Attr}_{\mathfrak C^{\operatorname{op}}}(q)} \subset \pi^{-1}(V^{\le 0})$ we conclude $\overline{\operatorname{Attr}_{\mathfrak C}(p)} \cap \overline{\operatorname{Attr}_{{\mathfrak C}^{\operatorname{op}}}(q)}\subset \pi^{-1}(V^0)$. Applying Lemma~\ref{lemma:projectionAttr} gives
	\[
		\overline{\operatorname{Attr}_{\mathfrak C}(p)} \cap \overline{\operatorname{Attr}_{{\mathfrak C}^{\operatorname{op}}}(q)}\subset \pi^{-1}(v).
	\]
	Since $\pi$ is proper, we know that $\pi^{-1}(v)$ is proper over $\mathbb C$. As $\overline{\operatorname{Attr}_{\mathfrak C}(p)} \cap \overline{\operatorname{Attr}_{{\mathfrak C}^{\operatorname{op}}}(q)}$ is a closed subvariety of $\pi^{-1}(v)$, it is also proper over $\mathbb C$.
	\end{proof}
	
	\begin{example} Continuing in the setting of Subsection~\ref{subsection:ExampleAttractingCells}, we have that the intersection $\overline{\operatorname{Attr}_{\sigma_0}(x_{D_2})}\cap \overline{\operatorname{Attr}_{\sigma_0^{-1}}(x_{D_1})}$ is isomorphic to the complex projective line $\mathbb P^1=\mathbb C\cup\{\infty\}$; an isomorphism is given by $\infty\mapsto x_{D_2}$, $x\mapsto \eta_1(0,0,0,x)$ for $x\in\mathbb C$.
	\end{example}
	
	\section{Stable envelopes}\label{section:stableEnvelopes}
	Stable envelopes were introduced in~\cite[Section~3.3.4]{maulik2019quantum} to define families of as equivariant cohomology classes satisfying certain stability conditions similar to the stability conditions of Schubert classes in equivariant Schubert calculus, see e.g.~\cite[Lemma~1]{knutson2003puzzles},  \cite[Lemma~3.8]{gorbounov2020yang}. As a consequence, one obtains families of stable envelope bases (similar to the family of equivariant Schubert bases  attached to different choices of Borels). These bases have remarkable characterisations and properties as described in~\cite[Section~3 and Section~4]{maulik2019quantum}. We describe now some of these in our context of bow varieties. 
	
Maulik and Okounkov consider in their setup smooth algebraic varieties $X$ admitting a holomorphic symplectic $\omega$ satisfying the following conditions:
	\begin{enumerate}[label=(Torus-\,\arabic*), wide=20pt]
		\item\label{item:MOAssumption1} There exists a pair of tori $A\subset T$ acting algebraically on $X$ such that $\omega$ is fixed by $A$ and scaled by $T$.
		\item\label{item:MOAssumption2} There exists an affine variety $X_0$ with algebraic $T$-action and an $T$-equivariant proper morphism $X\rightarrow X_0$.
	\end{enumerate}
	From our discussion in Subsection~\ref{subsection:torusAction}, respectively thanks to  Proposition~\ref{prop:properaffine}, bow varieties satisfy the conditions~\ref{item:MOAssumption1} and ~\ref{item:MOAssumption2}, and thus stable envelopes exist  by \cite{maulik2019quantum}. In this and the upcoming section, we give an exposition of the definition and existence of stable envelopes specially adapted to the framework of bow varieties and with the focus on proofs which allow explicit calculations.
	
	Before we go into detail, we briefly recall a version of the localization theorem in torus equivariant cohomology.

	\subsection{Localization in torus equivariant cohomology} Let $T=(\mathbb C^\ast)^r$ be a torus of rank $r$ and $X$ be a manifold. The localization theorem characterizes the $T$-equivariant cohomology of $X$ in terms of its $T$-fixed locus $X^T$. It provides in particular an interesting connection between local and global data. For more details on this subject see e.g.~\cite{hsiang1975cohomology},~\cite{tomdieck1987transformation},~\cite{allday1993cohomological},~\cite{goresky1998equivariant},  ~\cite{anderson2012introduction} and \cite{anderson2023equivariant}.
	%We briefly review a version of the localization theorem on torus equivariant cohomology. For an introduction to the subject see 
	%e.g.~\cite{hsiang1975cohomology},~\cite{tomdieck1987transformation} or~\cite{allday1993cohomological},~\cite{goresky1998equivariant},~\cite{anderson2012introduction}.
	
	 To formulate the localization theorem, we set up some notation:the $T$-equivariant cohomology of a single point is given by the polynomial ring $H_T^\ast(\mathrm{pt})=\mathbb Q[t_1,\ldots,t_r]$ where $t_1,\ldots,t_r$ are the equivariant parameters corresponding to the components of $T$. Let $S\subset \mathbb Q[t_1,\ldots,t_r]$ be the multiplicative subset generated by all linear polynomials $a_1t_1+\ldots+a_rt_r$ with $a_1,\ldots,a_r\in\mathbb Z$ and set $H_T^\ast(X)_{\mathrm{loc}}:=S^{-1}H_T^\ast(X)$.
	 
	%Let $T=(\mathbb C^\ast)^r$ be a torus of rank $r$ and $X$ be a topological space with a (topological) $T$-action. We denote the $T$-equivariant cohomology of $X$ with coefficients in $\mathbb Q$ by $H_T^\ast(X)$. T the $T$-equivariant cohomology of a single point is given by the polynomial ring $H_T^\ast(\mathrm{pt})=\mathbb Q[t_1,\ldots,t_r]$ where $t_1,\ldots,t_r$ are the equivariant parameters corresponding to the components of $T$. Each equivariant parameter is homogeneous of degree $2$. We have that $H_T^\ast(X)$ is naturally a $H_T^\ast(\mathrm{pt})$-module.
	%The localization theorem characterizes the $T$-equivariant cohomology of $X$ in terms of its $T$-fixed locus $X^T$ and hence provides an interesting connection between local and global data.
	
	\begin{theorem}[Localization] Suppose $X$ admits an $T$-equivariant open embedding into a smooth compact $T$-manifold. Then, the restriction $i^\ast:H_T^\ast(X)\rightarrow H_T^\ast(X^T)$ induces an isomorphism of the localized rings
	\[
	H_T^\ast(X)_{\mathrm{loc}}\xrightarrow\sim H_T^\ast(X^T)_{\mathrm{loc}}.
	\]
	\end{theorem}
	
	Note that if $X^T$ is finite, then $H_T^\ast(X^T)_{\mathrm{loc}}$ is just a product of $|X^T|$-many copies of $H_T^\ast(\mathrm{pt})_{\mathrm{loc}}$.
	
	\subsection{Stable envelopes}
	
	We return to the setup where $X$ is a bow variety $\mathcal C(\mathcal D)$ and $T$ is either $\mathbb A$ or $\mathbb T$. We denote the equivariant parameters by $H_{\mathbb A}^\ast(\mathrm{pt})=\mathbb Q[t_1,\ldots,t_N]$ and $H_{\mathbb T}^\ast(\mathrm{pt})=\mathbb Q[t_1,\ldots,t_N,h]$ according to the components of $\mathbb A$ and $\mathbb T$.
	\begin{definition}
	Let $d$ be the dimension of $\mathcal C(\mathcal D)$ as complex variety. Stable envelopes are maps, depending on a choice of a chamber $\mathfrak C$, 
	\[
	\mathcal C(\mathcal D)^{\mathbb T} \xrightarrow{\phantom{x}\mathrm{Stab}_{\mathfrak C}\phantom{x}}  H_{\mathbb T}^d(\mathcal C(\mathcal D)),
	\]
	 which are uniquely characterized by the properties ~\ref{item:StableEnvelopesNormalization}- \ref{item:StableEnvelopesSmallnes} from Theorem~\ref{thm:existenceStableEnvelopes}, called the \textit{normalization}, \textit{support} and \textit{smallness} condition, respectively. %The equivariant cohomology class $\mathrm{Stab}_{\mathfrak C}(p)$ is called the \textit{stable envelope class} attached to $\mathfrak C$ and $p$.
	\end{definition}
	 %The definition of stable envelopes directly transfers to the setting of bow varieties:
	
	\begin{theorem}[Stable envelopes]\label{thm:existenceStableEnvelopes} 
	Fix a chamber $\mathfrak C$. Then there exist a unique family  $(\mathrm{Stab}_{\mathfrak C}(p))_{p\in\mathcal C(\mathcal D)^{\mathbb T}}$ of elements in $H_{\mathbb T}^d(\mathcal C(\mathcal D))$ satisfying the following conditions:
	\begin{enumerate}[label=(Stab-\,\arabic*), wide=20pt]
		\item\label{item:StableEnvelopesNormalization} We have $\iota_p^\ast(\mathrm{Stab}_{\mathfrak C}(p))=e_{\mathbb T}(T_p\mathcal C(\mathcal D)_{\mathfrak C}^-$ for all $p\in\mathcal C(\mathcal D)^{\mathbb T}$.
		\item\label{item:StableEnvelopesSupport}  We have that $\mathrm{Stab}_{\mathfrak C}(p)$ is supported on $\operatorname{Attr}_{\mathfrak C}^f(p)$ for all $p\in\mathcal C(\mathcal D)^{\mathbb T}$.% i.e. the restriction of  $\mathrm{Stab}_{\mathfrak C}(p)$ to $\mathcal C(\mathcal D)\setminus \operatorname{Attr}_{\mathfrak C}^f(p)$ vanishes for all $p\in\mathcal C(\mathcal D)^{\mathbb T}$.
		\item\label{item:StableEnvelopesSmallnes} Let $p,q\in\mathcal C(\mathcal D)^{\mathbb T}$ with $q\prec p$. Then, $\iota_q^\ast(\mathrm{Stab}_{\mathfrak C}(p))$ is divisible by $h$.
	\end{enumerate}
	\end{theorem}

	The normalization and support condition directly imply that stable envelopes provide a basis of the localized equivariant cohomology ring: 
	
	\begin{cor}[Stable basis]\label{cor:StableEnvelopesBasis} For a fixed chamber $\mathfrak C$, the $\mathbb T$-equivariant cohomology classes $(\mathrm{Stab}_{\mathfrak C}(p))_{p\in\mathcal C(\mathcal D)^{\mathbb T}}$ form a $H^\ast_{\mathbb T}(\operatorname{pt})_{\mathrm{loc}}$-basis of $H_{\mathbb T}^\ast(\mathcal C(\mathcal D))_{\mathrm{loc}}$.
	\end{cor}
	
	We refer to $(\mathrm{Stab}_{\mathfrak C}(p))_{p\in\mathcal C(\mathcal D)^{\mathbb T}}$ as \textit{stable basis corresponding to $\mathfrak C$} and to the individual $\mathbb T$-equivariant cohomology classes $\mathrm{Stab}_{\mathfrak C}(p)$ as \textit{stable basis elements}.
	%The individual $\mathbb T$-equivariant cohomology classes $\mathrm{Stab}_{\mathfrak C}(p)$ 

	\begin{remark} The stable envelopes map ${\mathrm{Stab}_{\mathfrak C}}$ provides a map
	\[
	\{\textup{Chambers}\} \rightarrow \{\textup{Bases of $H_{\mathbb T}^\ast(\mathcal C(\mathcal D))_{\mathrm{loc}}$}\}.
	\]
	It is a central result of Maulik and Okounkov that the base change matrices with respect to adjacent chambers give solutions to Yang--Baxter equations,   providing an interesting connection to the theory of integrable systems; see ~\cite[Section~4]{maulik2019quantum} for more details.
	\end{remark}
	
	\begin{remark} In the case of Nakajima quiver varieties, the definition of stable envelopes in~\cite{maulik2019quantum} also includes a choice of signs in the normalization axiom, that corresponds to a choice of polarization of the involved Nakajima quiver variety. Polarizations can be defined in the setting of bow varieties too, see \cite[Section~4.4.1]{shou2021bow} and one could work with the more general definition. For simplicity, we however choose here all signs to be $1$. 
	\end{remark}
	
	\subsection{Proof of uniqueness}
	
	%The proof of the uniqueness of stable envelopes transfers from~\cite[Section~3.3.4]{maulik2019quantum} to our setting. For completeness, we include it at this point. The crucial observation is the following lemma:
	We now prove the uniqueness statement of Theorem~\ref{thm:existenceStableEnvelopes}. We use the following auxiliary statement.
	\begin{lemma}\label{lemma:UniquenessPreparation}
Assume  $(\gamma_p)_{p\in\mathcal C(\mathcal D)^{\mathbb T}}$ in $H_T^\ast(\mathcal C(\mathcal D))$ is a family of equivariant cohomology classes of homogeneous degree $\operatorname{dim}(\mathcal C(\mathcal D))$ satisfying the following two conditions:
	\begin{enumerate}[label=(\alph*)]
		\item\label{item:UniquenessProofCondition1} For all $p\in \mathcal C(\mathcal D)^{\mathbb T}$, the class $\gamma_p$ is supported on $\operatorname{Attr}^f_{\mathfrak C}(p)$.
		\item\label{item:UniquenessProofCondition2} If $p,q\in \mathcal C(\mathcal D)^{\mathbb T}$ with $q\preceq p$, then $\iota_q^{\ast}(\gamma_p)$ is divisible by $h$.
	\end{enumerate}
	Then, $\gamma_p=0$ for all $p\in\mathcal C(\mathcal D)^{\mathbb T}$.
	\end{lemma}
	Assuming Lemma~\ref{lemma:UniquenessPreparation}, we directly obtain a proof of the uniqueness of stable envelopes:
	\begin{proof}[Proof of Theorem~\ref{thm:existenceStableEnvelopes} (Uniqueness)]
	If $(\mathrm{Stab}_{\mathfrak C}(p))_{p\in\mathcal C(\mathcal D)^{\mathbb T}}$ and $(\mathrm{Stab}_{\mathfrak C}'(p))_{p\in\mathcal C(\mathcal D)^{\mathbb T}}$ satisfy the conditions of Theorem~\ref{thm:existenceStableEnvelopes} then the family 
	$(\mathrm{Stab}_{\mathfrak C}(p)-\mathrm{Stab}_{\mathfrak C}'(p))_{p\in\mathcal C(\mathcal D)^{\mathbb T}}$ satisfies the conditions of Lemma~\ref{lemma:UniquenessPreparation}. Hence, $\mathrm{Stab}_{\mathfrak C}(p)=\mathrm{Stab}_{\mathfrak C}'(p)$ for all $p\in\mathcal C(\mathcal D)^{\mathbb T}$.
	\end{proof}
	
	\begin{proof}[Proof of Lemma~\ref{lemma:UniquenessPreparation}]
	Fix $p\in\mathcal C(\mathcal D)^{\mathbb T}$ and let $\le$ be a total order refining $\preceq$. Let $n=|\mathcal C(\mathcal D)^{\mathbb T}|$ and denote the $\mathbb T$-fixed points of $\mathcal C(\mathcal D)$ by $q_1,\ldots,q_n$ where $q_i\le q_j$ if and only if $i\le j$. Let $i_0\in\{1,\ldots,n\}$ such that $p=q_{i_0}$. Furthermore, we set 
	\[
	A_i:=\bigsqcup_{j=1}^i\operatorname{Attr}_{\mathfrak C}(q_j) \quad\textup{for }i=1,\ldots,n, 
	\]
	and $A_0:=\emptyset$. According to Proposition~\ref{prop:fullattractionunion}, each $A_{q_i}$ is a closed $\mathbb T$-invariant subvariety of $\mathcal C(\mathcal D)$ and contains $\operatorname{Attr}_{\mathfrak C}(q_i)$ as open subvariety. To prove $\gamma_p=0$, we show that 
	\begin{center}
	{\it if $\gamma_p$ is supported on $A_i$ for some   $i\in\{1,\ldots,n\}$, then, $\gamma_p$ is also supported on $A_{i-1}$.}
	\end{center}
	  This implies $\gamma_p=0$, since $\gamma_p$ is supported on $A_{0}$ and thus has empty support by induction.
	
	So let us prove the above statement. 
	Let $\kappa: A_i\hookrightarrow \mathcal C(\mathcal D)$ be the inclusion.
	Since $\gamma_p$ is supported on $A_i$ there exists $\alpha\in \overline H_\ast^{\mathbb T} (A_i)$ such that $\kappa_\ast(\alpha)$ is Poinca\'e dual to $\gamma_p$. Let $f:\{q_i\}\hookrightarrow \operatorname{Attr}_{\mathfrak C}(q_i)$ and $g:\operatorname{Attr}_{\mathfrak C}(q_i)\hookrightarrow A_i$ denote the inclusions and let $\alpha'\in H_{\mathbb T}^\ast(\operatorname{Attr}_{\mathfrak C}(q_i))$ be the Poincar\'e dual of $g^\ast(\alpha)$.
%	Next, we apply the equivariant excess intersection formula to the following fiber diagram
%	\[
%	\begin{tikzcd}
%	\{q_i\} \arrow[r,"g"]\arrow[d]& \operatorname{Attr}_{\mathfrak C}(q_i)\arrow[d,"f"]\\
%	\{q_i\} \arrow[r]\arrow[d]& A_i\arrow[d,"\kappa"]\\
%	\{q_i\} \arrow[r,"\iota_{q_i}"]& \mathcal C(\mathcal D)
%	\end{tikzcd}
%	\]
%	The excess bundle over $\{q_i\}$ is the $\mathbb C$-vector space $T_{q_i}\mathcal C(\mathcal D)^-$.  
	Then, the equivariant excess intersection fomula gives
	\begin{equation}\label{eq:ExcessIntersection}
		e_{\mathbb T}(T_{q_i}\mathcal C(\mathcal D)_{\mathfrak C}^-) f^\ast(\alpha') = \iota_{q_i}^\ast(\gamma_p).
	\end{equation}
	According to Corollary~\ref{cor:tangentweights}, $e_{\mathbb T}(T_{q_i}\mathcal C(\mathcal D)^-_{\mathfrak C})$ is homogeneous of degree $d$ and not divisible by $h$. Thus, by condition~\ref{item:UniquenessProofCondition2} and~\eqref{eq:ExcessIntersection}, we infer $f^\ast(\alpha')=0$. Since $f^\ast$ is an isomorphism of rings, we conclude $\alpha'=0$ which is equivalent to $\alpha$ being supported on $A_{i-1}$. 
	 Hence, $\gamma_p$ is supported on $A_{i-1}$ as well.
	\end{proof}
%	
%	The uniqueness of stable envelopes is now an immediate consequence:
%	
%	\begin{proof}[Proof of Theorem~\ref{thm:existenceStableEnvelopes} (Uniqueness)]
%	If $(\mathrm{Stab}_{\mathfrak C}(p))_{p\in\mathcal C(\mathcal D)^{\mathbb T}}$ and $(\mathrm{Stab}_{\mathfrak C}'(p))_{p\in\mathcal C(\mathcal D)^{\mathbb T}}$ satisfy the conditions of Theorem~\ref{thm:existenceStableEnvelopes} then 
%	$(\mathrm{Stab}_{\mathfrak C}(p)-\mathrm{Stab}_{\mathfrak C}'(p))_{p\in\mathcal C(\mathcal D)^{\mathbb T}}$ satisfies the conditions of Lemma~\ref{lemma:UniquenessPreparation}. Hence, $\mathrm{Stab}_{\mathfrak C}(p)=\mathrm{Stab}_{\mathfrak C}'(p)$ for all $p\in\mathcal C(\mathcal D)^{\mathbb T}$.
%	\end{proof}
	
	\section{Existence of stable envelopes}\label{section:existence}
	%To prove the existence of stable envelopes, we follow closely~\cite[Section~3.4 and~3.5]{maulik2019quantum}. Since bow varieties admit only finitely many fixed points, some arguments simplify in our setting.
	In this section, we discuss the existence of stable envelopes. We will see that they can be constructed using an iterative procedure based on general properties of equivariant multiplicities of lagrangian subvarieties.
		
	We first recall basic facts about equivariant multiplicities and then consider in more detail the special case of equivariant multiplicities of lagrangian subvarieties before we finally prove the existence of stable envelopes and the statements from Subsection~\ref{subsection:EquivariantMultiplicitiesLagrangian}. In Subsection~\ref{subsection:ExampleConstructionStableEnvelopes}, we finally consider an example of the construction of stable envelopes.
	
	\subsection{Equivariant multiplicities}\label{subsection:EquivariantMultiplicities}
	In this subsection, we recall the notion of equivariant multiplicities. For details on equivariant intersection theory, see e.g.\cite{brion1997equivariant}, \cite{edidin1998equivariant} and \cite{anderson2023equivariant}. We fix a torus $T=(\mathbb C^\ast)^r$  acting algebraically on a smooth algebraic variety $X$. 
	%Equivariant multiplicities were introduced by Rossmann in~\cite{rossmann1989equivariant}. 
	
	Let $p\in X^T$ be an isolated $T$-fixed point. Let $\iota_p:\{p\}\hookrightarrow X$ be the inclusion and $\iota_p^\ast:H_T^\ast(X)\rightarrow H_T^\ast(\{p\})$ be the induced map on equivariant cohomology rings. The projection $\pi:T_pX\rightarrow \{p\}$ yields an isomorphism of rings $\pi^\ast:H_T^\ast(\{p\})\xrightarrow\sim H_T^\ast(T_pX)$. The inverse $\pi^\ast$ is called equivariant Gysin isomorphism and we denote it with $s^\ast$.
	
	Let $Y$ be a $T$-invariant subvariety of $X$ containing $p$ with structure sheaf $\mathcal O_Y$.
	and let $[Y]^T\in H_T^\ast(X)$ denote the equivariant cohomology class corresponding to $Y$. The \textit{equivariant multiplicity} of $Y$ at $p$ is defined as $\iota_p^\ast([Y]^T)$. Equivariant multiplicities can be characterized in terms of tangent cones as follows (see e.g. \cite[Remark~17.4.2]{anderson2023equivariant}): Let $\mathcal I\subset \mathcal O_Y$ be the ideal sheaf corresponding to $p$. Then, the \textit{tangent cone} of $Y$ at $p$ is defined as
	\[
	C_pY:=\mathrm{Spec}\Big(\bigoplus_{n=1}^\infty \mathcal{I}^n/\mathcal{I}^{n+1}\Big).
	\] 
	Note that $C_pY$ may be a non-reduced scheme. 
	
	By construction, $C_pY$ admits an algebraic $T$-action and we have a $T$-equivariant closed immersion $C_pY\hookrightarrow T_pX$. Let $Z_1,\ldots,Z_n$ be the irreducible components of $C_pY$ and let $m_i$ be the geometric multiplicity of $Z_i$ in $C_pY$, see e.g. \cite[Section~1.5]{fulton1984introduction} for a definition of geometric multiplicities. The irreducible components $Z_1,\ldots,Z_n$ are again $T$-invariant subvarieties of $T_pX$ and we have
	\begin{equation}\label{eq:EquivariantMultiplicities}
	\iota_p^\ast([Y]^T)=\sum_{i=1}^n m_i s^\ast([Z_i]^T).
	\end{equation}
	If $Y$ is regular at $p$, we have $C_pY=T_pY$ and $\iota_p^\ast([Y]^T)=e_T(N_p(Y/X))$ where 
	%$e_T$ denotes the equivariant Euler class and 
	$N_p(Y/X)$ is the fiber of the normal bundle corresponding to the closed immersion $Y\hookrightarrow X$ at the point~$p$.
	
	\begin{example}
	Consider the bow variety $\mathcal C(0\textcolor{red}{\slash} 1 \textcolor{blue}{\backslash} 1\textcolor{red}{\slash} 2  \textcolor{blue}{\backslash}2 \textcolor{blue}{\backslash} 2\textcolor{red}{\slash} 0)$ that was already studied in Subsection~\ref{subsection:ExampleAttractingCells} and Subsection~\ref{subsection:ExamplePartialOrder}. In virtue of~\eqref{eq:EquivariantMultiplicities}, we can simply read of all $\mathbb T$-equivariant multiplicities $\iota_{x_{D_j}}^\ast (\overline{\operatorname{Attr}_{\sigma_0}(x_{D_i})})$ from Table~\ref{table:WeightSpaceDecompositionWi} and Table~\ref{table:IntersectionAttractionOpenAffine}. The respective $\mathbb T$-equivariant multiplicities are recorded in Table~\ref{table:EquivariantMultiplicitiesAttractingCells} .
	
	\begin{center}
	\begin{table}
\begin{tabular}{ | c | c | c | c | c | c |}
 \hline 
 \diagbox{$i$}{$j$} & $1$ & $2$ & $3$ & $4$ & $5$ \\
 \hline
 \hline
 $1$ & 
 \makecell{\hspace{-1em}\tiny$(t_1-t_3)$ \\ \hspace{1em}\tiny $\cdot(t_2-t_3)$ } & 
 $0$ & $0$ & $0$ & $0$ \\
 \hline 
 $2$ & 
  \makecell{\hspace{-1em}\tiny$(t_1-t_3)$ \\ \hspace{2em}\tiny $\cdot(t_3-t_2+h)$ }  & 
  \makecell{\hspace{-1em}\tiny$(t_1-t_2)$ \\ \hspace{2em}\tiny $\cdot(t_2-t_3+h)$ } &
 $0$ & $0$ & $0$ \\
 \hline
 $3$ &
 \makecell{\hspace{-1em}\tiny$(t_3-t_1+h)$ \\ \hspace{1em}\tiny $\cdot(t_3-t_2+h)$ } & 
 \makecell{\hspace{-1em}\tiny$(t_2-t_1+h)$ \\ \hspace{1em}\tiny $\cdot(t_2-t_3+h)$ } &
 \makecell{\hspace{-1em}\tiny$(t_1-t_2+h)$ \\ \hspace{1em}\tiny $\cdot(t_2-t_3+h)$ } & 
 $0$ & $0$ \\
 \hline
 $4$ & 
 \makecell{\hspace{-1em}\tiny$(t_2-t_3)$ \\ \hspace{2em}\tiny $\cdot(t_3-t_1+h)$ } & 
 $0$ & 
 \makecell{\hspace{-1em}\tiny$(t_2-t_1)$ \\ \hspace{2em}\tiny $\cdot(t_2-t_3+h)$ } &
 \makecell{\hspace{-1em}\tiny$(t_2-t_3)$ \\ \hspace{2em}\tiny $\cdot (t_1-t_2+2h)$ } &
 $0$\\
 \hline
 $5$ & $0$ &
 \makecell{\hspace{-1em}\tiny$(t_3-t_2)$ \\ \hspace{2em}\tiny $\cdot(t_2-t_1+h)$ } &
 \makecell{\hspace{-1em}\tiny$(t_2-t_1)$ \\ \hspace{1em}\tiny $\cdot(t_3-t_3)$ } &
 \makecell{\hspace{-1em}\tiny$(t_1-t_2+2h)$ \\ \hspace{1em}\tiny $\cdot(t_3-t_2+h)$ } &
 \makecell{\hspace{-1em}\tiny$(t_1-t_3+2h)$ \\ \hspace{1em}\tiny $\cdot(t_2-t_3+h)$ } \\
 \hline  
\end{tabular}
\caption{Equivariant multiplicities $\iota_{x_{D_j}}^\ast (\overline{\operatorname{Attr}_{\sigma_0}(x_{D_i})})$}
    \label{table:EquivariantMultiplicitiesAttractingCells} 
\end{table}
\end{center}
	\end{example}
	
	\subsection{Equivariant multiplicities of lagrangian subvarieties}\label{subsection:EquivariantMultiplicitiesLagrangian}
	In this Subsection we formulate the main ingredient in the proof of the existence of stable envelopes. We call this the Langrangian Multiplicity Result since it characterizes the equivariant multiplicity of lagrangian subvarieties at $p$ in terms of the $T$-weights. Thus, from now on we need additionally
	\begin{assumption}
	 We assume from now on until Subsection~\ref{subsection:flatDeformationsLagrangianVarieties} that $X$ is symplectic of dimension $2n$ with symplectic form $\omega$ and $\omega$ is invariant under the $T$-action.
	\end{assumption}
	%Let $X$ be a smooth symplectic variety with symplectic form $\omega$ of dimension $2n$.
	Recall that a subvariety $L\subset X$ is called \textit{isotropic} if the restriction of $\omega$ to the smooth locus of $L$ vanishes. We call $L$ \textit{lagrangian} if $\operatorname{dim}(L)=n$ and moreover $L$ is isotropic.
	
	%Suppose further that $X$ admits an action of a torus $A=(\mathbb C^\ast)^r$ such that $\omega$ is invariant under the $A$-action. 
	Given an isolated $T$-fixed point $p\in X$, we can choose a decomposition into $T$-invariant subspaces $T_pX=V\oplus W$ such that $\omega_p$ induces an isomorphism $V\cong W^\ast$.

	\begin{prop}[Langrangian Multiplicities]\label{prop:equivariantMultiplicitiesLagrangian}
	Let $\chi_1,\ldots,\chi_n$ be the $T$-characters corresponding to the $T$-action on $V$. We view $\chi_1,\ldots,\chi_n$ as homogeneous elements in $H_T^\ast(\{p\})$.
	Given a $T$-invariant lagrangian subvariety $L\subset X$, we can find  $a_{p,L}\in\mathbb Z$  such  that the following equality holds:
	\[
	\iota_p^\ast ([L]^T)=a_{p,L} \prod_{i=1}^n \chi_i.
	\] 
	\end{prop}
	We prove Proposition~\ref{prop:equivariantMultiplicitiesLagrangian} in Subsection~\ref{subsection:flatDeformationsLagrangianVarieties}. First, we apply this result to the setting of bow varieties. For this, set $L_p:=\overline{\operatorname{Attr}_{\mathfrak C}(p)}$ for $p\in\mathcal C(\mathcal D)^{\mathbb T}$.
	\begin{cor}\label{cor:equivariantMultiplicitiesAttractingCells} Let $p,q\in\mathcal C(\mathcal D)^{\mathbb T}$ and suppose $p\in L_q$. Then,
	\[\iota_p^\ast ([L_q]^{ \mathbb A})= a_{p,q} e_{\mathbb A}(T_p(\mathcal C(\mathcal D))^-_{\mathfrak C}),
	\]
	where $a_{p,q}$ is an integer depending on $p$ and $q$.
	\end{cor}
	\begin{proof} The tangent space $T_qL_q=T_q(\mathcal C(\mathcal D))^+$ is an isotropic subspace of $T_q(\mathcal C(\mathcal D))$. As the form $\omega_{\mathcal C(\mathcal D)}$ is $\mathbb A$-invariant, we conclude that $T_xL_q$ is an isotropic subspace of $T_x(\mathcal C(\mathcal D))$ for all $x\in \operatorname{Attr}_{\mathfrak C}(q)$. Hence, $L_q$ is a lagrangian subvariety of $\mathcal C(\mathcal D)$ as the restriction of $\omega_{\mathcal C(\mathcal D)}$ to an open dense smooth subvariety vanishes. Applying Proposition~\ref{prop:equivariantMultiplicitiesLagrangian} according to the decomposition $T_p\mathcal C(\mathcal D)=T_p\mathcal C(\mathcal D)^-_{\mathfrak C}\oplus T_p\mathcal C(\mathcal D)^+_{\mathfrak C}$ then finishes the proof.
	\end{proof}
	\subsection{Proof of the existence of stable envelopes}\label{subsection:ProofExistenceStableEnvelopes}
	
	%Following \cite[Proposition~3.5.1]{maulik2019quantum}, we use Corollary~\ref{cor:equivariantMultiplicitiesAttractingCells} to give a direct construction of stable envelopes:
	
	We now use Corollary~\ref{cor:equivariantMultiplicitiesAttractingCells} to give a direct construction of stable envelopes:
	
	\begin{proof}[Proof of Theorem~\ref{thm:existenceStableEnvelopes}]
	Let $\preceq'$ be a total order on $\mathcal C(\mathcal D)^{\mathbb T}$ refining $\preceq$. Let $s$
 be the cardinality of 	$\mathcal C(\mathcal D)^{\mathbb T}$ and denote the elements of $\mathcal C(\mathcal D)^{\mathbb T}$ by $p_1,\ldots,p_s$ where the labeling is compatible with our choice of total ordering, i.e. we have $p_i\preceq'p_j$ if and only if $i\le j$. For each $i\in\{1,\ldots,s\}$, we construct a family of cohomology classes $\gamma_{i,1},\ldots,\gamma_{i,i}\in H_{\mathbb T}^\ast(\mathcal C(\mathcal D))$ such that each $\gamma_{i,j}$ satisfies the following three properties
 \begin{enumerate}[label=(\alph*)]
 	\item\label{item:ExistenceProofNormalization} $\iota_{p_i}(\gamma_{i,j})=e_{\mathbb T}(T_{p_i}\mathcal C(\mathcal D)^-_{\mathfrak C})$,
 	\item\label{item:ExistenceProofSumOfLagrangians} there exist $a_{i,j,1},\ldots,a_{i,j,i-j}\in\mathbb Z$ such that
 	$\gamma_{i,j}= [L_{p_i}]^{\mathbb T} + \sum_{l=1}^{i-j} a_{i,j,l}[L_{p_{i-l}}]^{\mathbb T}$,
 	\item\label{item:ExistenceProofSmallness} we have that $\iota_{p_l}^\ast(\gamma_{i,j})$ is divisible by $h$ for $l=i-1,i-2,\ldots,j$.
 \end{enumerate}
 	We set $\gamma_{i,i}:= [L_{p_i}]^{\mathbb T}$. Then $\gamma_{i,i}$ clearly satisfies the properties~\ref{item:ExistenceProofNormalization}-\ref{item:ExistenceProofSmallness}. Suppose $\gamma_{i,i},\ldots,\gamma_{i,j}$ have been constructed, then we define $\gamma_{i,j-1}$ as follows: 
 	since $\gamma_{i,j}$ satisfies~\ref{item:ExistenceProofSumOfLagrangians}, we know by Corollary~\ref{cor:equivariantMultiplicitiesAttractingCells} that there exists $a\in\mathbb Z$ such that
 	\[
 	\iota_{p_{j-1}}^\ast(\gamma_{i,j})\equiv ae_{\mathbb T}(T_{p_{j-1}}\mathcal C(\mathcal D)^-_{\mathfrak C}) \;\mathrm{mod}\;(h). 
 	\]
 	Set $\gamma_{i,j-1}:=\gamma_{i,j}-a[L_{p_{j-1}}]^{\mathbb T}$. By construction, $\gamma_{i,j-1}$ satisfies~\ref{item:ExistenceProofSumOfLagrangians} and $\iota_{p_{j-1}}^\ast(\gamma_{i,j-1})$ is divisible by $h$. Hence, properties~\ref{item:ExistenceProofNormalization} and~\ref{item:ExistenceProofSmallness}  follow from $p_i,p_{i-1},\ldots,p_j\notin L_{p_{j-1}}$. Thus, $\gamma_{i,j-1}$ satisfies all the desired properties.
 	
 	Now, set $\mathrm{Stab}(p_i):=\gamma_{i,1}$ for $i=1,\ldots,s$. Then, the normalization condition follows immediately from~\ref{item:ExistenceProofNormalization}, the support condition from~\ref{item:ExistenceProofSumOfLagrangians} and the smallness condition from~\ref{item:ExistenceProofSmallness}. This completes the proof of Theorem~\ref{thm:existenceStableEnvelopes}.
	\end{proof}
	
	The following two subsections are devoted to the proof of Proposition~\ref{prop:equivariantMultiplicitiesLagrangian}. We begin with general observations on tangent cones of lagrangian subvarieties.
	
	\subsection{Tangent cones of lagrangian subvarieties}\label{subsection:tangentCones}
	
	In this subsection, we pass to the complex analytic setting. At first, we recall some basic facts that are similar to the algebraic set up. Let $Y$ be a complex analytic space. We denote the smooth locus of $Y$ by $Y_{\mathrm{sm}}$. If $Y$ is reduced, then $Y_{\mathrm{sm}}$ is an open dense complex subspace of $Y$. Moreover, we call the reduction of $Y$ by $Y_{\mathrm{red}}$. Given a closed analytic subspace $Z\subset Y$ with ideal sheaf $\mathcal I$, the \textit{normal cone} is defined as
	\[
	C_{Z/Y}:=\mathrm{Specan}\Big(\bigoplus_{i\geq0}\mathcal I^i/\mathcal I^{i+1}\Big).
	\]
	Here, $\mathrm{Specan}$ denotes the analytic spectrum, see e.g.~\cite[Section~II.3]{grauert1994several} for a definition.
	
	 In the special case where $Z=\{p\}$ is a point on $Y$, the normal cone $C_{Z/Y}$ is also called \textit{tangent cone} and denoted by $C_pY$. If $Y$ is a closed complex subspace of a complex manifold $X$, we have a canonical closed embedding $C_pY\hookrightarrow T_pX$.

	\begin{prop}\label{prop:AnalyticTangentConeLagrangian} Let $X$ be a holomorphic symplectic manifold with symplectic form $\omega$ and let $L\subset X$ be an analytic lagrangian subvariety and $p\in L$. Then, $C_pL_{\mathrm{red}}$ is a lagrangian subvariety of $T_pX$.
	\end{prop}
	Here, $C_pL_{\mathrm{red}}$ denotes the reduction of $C_pL$. Before we prove Proposition~\ref{prop:AnalyticTangentConeLagrangian}, we recall some basic facts about the deformation to the normal cone construction which was introduced by Fulton in~\cite[Chapter~5]{fulton1984introduction}. Given a closed embedding $Z\hookrightarrow Y$ as above, we denote by $\mathrm{Bl}_{Z}(Y)$ the blow up of $Y$ along $Z$. For an introduction to the theory of blow ups in the analytic framework see e.g.~\cite[Chapter~4]{fischer1976complex},~\cite[Chapter~VII]{grauert1994several}.
	 The \textit{deformation to the normal cone} with respect to the closed embedding $Z\hookrightarrow X$ is defined as 
	\[
	M_{Z/X}^0:=\mathrm{Bl}_{Z\times\{0\}}(Y\times\mathbb C)\setminus \mathrm{Bl}_{Z}(Y).
	\]
	Let $\pi:M_{Z/Y}^0\rightarrow\mathbb C$ be the projection. Then, $\pi^{-1}(\mathbb C\setminus\{0\})$ is isomorphic to $Y\times (\mathbb C\setminus\{0\})$ and $\pi^{-1}(\mathbb C\setminus\{0\})$ is dense in $M_{Z/Y}^0$. The special fiber $\pi^{-1}(0)$ is isomorphic to the normal cone $C_{Z/X}$.
	Moreover, if $Y$ is reduced, then also $M_{Z/Y}^0$ is reduced. 
	%After these preparations, we proceed with the proof of the above proposition:
	
	\begin{proof}[Proof of Proposition~\ref{prop:AnalyticTangentConeLagrangian}] Since $C_pL$ is pure of dimension $\operatorname{dim}(L)$, it is left to show that $C_pL_{\mathrm{red}}$ is an isotropic subvariety of $T_pX$.
	The normal cone $C_pL$ only depends on an analytic neighborhood of $p$. Hence, by the holomorphic Darboux theorem, we may assume that $X$ is an analytic neighborhood of the origin in $\mathbb C^{2n}$ and $\omega$ is the standard symplectic form
	\[
	\omega= \sum_{i=1}^n dx_i\wedge dx_i',
	\]
	where $(x_1,\ldots,x_n,x_1',\ldots,x_n')$ are the coordinates of $\mathbb C^{2n}$. We further may assume that $p$ equals the origin. By construction, $M^{0}_{p/\mathbb C^{2n}}$ is isomorphic to $\mathbb C^{2n}\times\mathbb C$ and the open embedding $\mathbb C^{2n}\times (\mathbb C\setminus\{0\})\hookrightarrow M^{0}_{p/\mathbb C^{2n}}$ corresponds under this identification to the morphism
	\begin{gather*}
	\iota:\mathbb C^{2n}\times (\mathbb C\setminus\{0\})\hookrightarrow \mathbb C^{2n}\times\mathbb C,\\ (x_1,\ldots,x_n,x_1',\ldots,x_n',t)\mapsto (t^{-1}x_1,\ldots,t^{-1}x_n,t^{-1}x_1',\ldots,t^{-1}x_n',t).
	\end{gather*}
	The closed embedding $T_p\mathbb C^{2n}\hookrightarrow M^{0}_{p/\mathbb C^{2n}}$ corresponds to
	\begin{gather*}
	\kappa:T_p\mathbb C^{2n}\hookrightarrow  \mathbb C^{2n}\times\mathbb C,\\
	(\alpha_1\frac{\partial}{\partial x_1 },\ldots,\alpha_n\frac{\partial}{\partial x_n},\alpha_1'\frac{\partial}{\partial x_1' },\ldots,\alpha_n'\frac{\partial}{\partial x_n'})\mapsto (\alpha_1,\ldots,\alpha_n,\alpha_1',\ldots,\alpha_n',0).
	\end{gather*}
	Let $\operatorname{pr}:\mathbb C^{2n}\times\mathbb C\rightarrow \mathbb C^{2n}$ be the projection to the first factor and $\omega'=\operatorname{pr}^\ast(\omega)$. Then $\omega'$ is a holomorphic alternating bilinear form on the tangent bundle $T(\mathbb C^{2n}\times\mathbb C)$. By definition, $\kappa^\ast \omega'=\omega$. Moreover, the holomorphic bilinear form $\iota^\ast \omega'$ on $T(\mathbb C^{2n}\times(\mathbb C\setminus \{0\}))$ is given by
	\begin{equation}\label{eq:pullBackOmegaPrime}
		\iota^\ast \omega'(x,t)=t^{-1}\sum_{i=1}^n dx_i\wedge dx_i',\quad\textup{where }{x\in\mathbb C^{2n}},t\in\mathbb C\setminus\{0\}.
	\end{equation}
	In the following, we will also denote by $\omega'$ the corresponding form on $TM_{p/\mathbb C^{2n}}^0$. We have the following commuting diagram, where the horizontal arrows are closed embeddings in the first column and open embeddings in the second column. 
	\[
	\begin{tikzcd}
	C_pL \arrow[r]\arrow[d]&
	T_pX\arrow[r,"\sim"]\arrow[d] &
	T_p\mathbb C^{2n}\arrow[d] \\
	M_{p/L}^0\arrow[r] &
	M_{p/X}^0\arrow[r] &
	M_{p/\mathbb C^{2n}}^0 \\
	L\times(\mathbb C\setminus \{0\})\arrow[r]\arrow[u] &
	X\times(\mathbb C\setminus \{0\})\arrow[r]\arrow[u] &
	\mathbb C^{2n}\times(\mathbb C\setminus \{0\})\arrow[u]
	\end{tikzcd}
	\]
	Since $L\subset X$ is lagrangian, we deduce from~\eqref{eq:pullBackOmegaPrime} that $\iota^{\ast}\omega'$ vanishes on the tangent bundle of $L_{\mathrm{sm}}\times(\mathbb C\setminus \{0\})$. Since $L_{\mathrm{sm}}\subset L$ and $L\times(\mathbb C\setminus \{0\})\subset M_{p/L}^0$ are both dense, we also have that $L_{\mathrm{sm}}\times(\mathbb C\setminus \{0\})\subset (M_{p/L}^0)_{\mathrm{sm}}$ is dense. It follows that $\omega'$ vanishes on $(M_{p/L}^0)_{\mathrm{sm}}$. By~\cite[Lemma~19.3]{whitney1965tangents}, there exists a smooth open dense subvariety $U\subset C_pL_{\mathrm{red}}$ such that $M_{p/L}^0$ is regular over $U$ (in the sense of \cite{whitney1965tangents}). This implies that the closure of $T(M_{p/L}^0)_{\mathrm{sm}}$ in $TM_{p/\mathbb C^{2n}}^0$ contains the tangent bundle $TU$. Hence, $\omega'$ also vanishes on $TU$ and thus on $T(C_pL_{\mathrm{red}})_{\mathrm{sm}}$. As $\kappa^{\ast}\omega'=\omega$, we conclude that $C_pL_{\mathrm{red}}$ is an isotropic subvariety of $T_pX$.
	\end{proof}
	
	\subsection{Flat deformations of conical lagrangian subvarieties}\label{subsection:flatDeformationsLagrangianVarieties}
	
	We now return to the algebraic setting. In the previous subsection, we proved that $C_pL$ is a lagrangian subvariety of $T_pX$. Furthermore, $C_pL$ is $T$-invariant and conical by construction. In this subsection, we show that $C_pL$ can be deformed to a possibly non-reduced union of lagrangian hyperplanes which enables us to characterize the equivariant multiplicity of $L$ at $p$.
	 %By construction, $C_pL$ is a $T$-invariant conical closed subvariety of $T_pX$. Moreover, in the previous subsection we proved that $C_pL$ is also lagrangian.
	 
	 \begin{prop}\label{prop:FlatDeformationLagrangianToHyperplanes}
	 We have $[C_pL]^T=\sum_{i=1}^r m_i[H_i]^T$ in $H_T^\ast(T_pX)$ where $H_1,\ldots,H_r\subset T_pX$ are lagrangian hyperplanes and $m_1,\ldots,m_r\in\mathbb N_0$.
	 \end{prop}
	 
	 Assuming Proposition~\ref{prop:FlatDeformationLagrangianToHyperplanes}, we obtain directly a proof of Proposition~\ref{prop:equivariantMultiplicitiesLagrangian}.
	
	\begin{proof}[Proof of Proposition~\ref{prop:equivariantMultiplicitiesLagrangian}]
	Given a lagrangian hyperplane $H\subset T_pX$, we have $s^\ast([H]^T)=\pm \prod_{j=1}^n\chi_j$. Thus, Proposition~\ref{prop:FlatDeformationLagrangianToHyperplanes} yields
	that $s^\ast([C_pL]^T)$ is an integer multiple of $\prod_{j=1}^n\chi_j$ which completes the proof.
	\end{proof}
	 
	 The remainder of this subsection is devoted to the proof of Proposition~\ref{prop:FlatDeformationLagrangianToHyperplanes}.
	
	Consider the symplectic vector space $\mathbb C^{2n}$ with 
	basis $e_1,\ldots,e_n,e_1',\ldots,e_n'$ standard symplectic form 
	$\omega(e_i,e_j')=-\omega(e_j',e_i)=\delta_{ij}$,
	$\omega(e_i,e_j)=\omega(e_i',e_j')=0$. As before, we view $\mathbb C^{2n}$ as symplectic variety. Let $A$ be a torus which acts diagonally on $\mathbb C^{2n}$ preserving the symplectic form. Moreover, let $T'=(\mathbb C^\ast)^{n+2}$ acting on $\mathbb C^{2n}$ via
	\[
	(t_1,\ldots,t_{n+2})\mapsto \begin{pmatrix}
	t_1t_{n+1} \\
	&\ddots \\
	&&t_nt_{n+1}\\
	&&& t_1^{-1}t_{n+2} \\
	&&&&\ddots
	&&&&& t_n^{-1}t_{n+2}
	\end{pmatrix} .
	\]
	Proposition~\ref{prop:FlatDeformationLagrangianToHyperplanes} is basically a consequence of the following lemmas:
	
	\begin{lemma}\label{lemma:flatDeformations}
	Let $C\subset \mathbb C^{2n}$ be an $A$-invariant, conical, lagrangian subvariety. Then, there exist irreducible, $T'$-invariant, conical, lagrangian subvarieties $Z_1,\ldots,Z_r\subset \mathbb C^{2n}$ and natural numbers $m_1,\ldots,m_r$ such that $[C]^A=\sum_{i=1}^{r} m_i [Z_i]^A$ in $H^\ast_A(\mathbb C^{2n})$.
	\end{lemma}
	
	\begin{lemma}\label{lemma:InvariantConicalLagrangian} Let $Z\subset \mathbb C^{2n}$ be an irreducible, $T'$-invariant, lagrangian subvariety. Then, there exist $v_1,\ldots,v_n\in \mathbb C^{2n}$ with $v_i\in\{e_i,e_i'\}$ for $i=1,\ldots,n$ and we have $Z=\mathrm{span}_{\mathbb C}(v_1,\ldots,v_n)$. 
	\end{lemma}
	
	\begin{proof}[Proof of Proposition~\ref{prop:FlatDeformationLagrangianToHyperplanes}]
	Choose a symplectic identification $T_pX\cong \mathbb C^{2n}$ where the form on $T_pX$ gets identified with $\omega$. By Lemma~\ref{lemma:flatDeformations}, there exist 
	$T'$-invariant, conical, lagrangian subvarieties $Z_1,\ldots,Z_r\subset \mathbb C^{2n}$ and natural numbers $m_1,\ldots,m_r$ such that $[C]^T=\sum_{i=1}^{r} m_i [Z_i]^T$. According to Lemma~\ref{lemma:InvariantConicalLagrangian}, $Z_1,\ldots,Z_r$ are lagrangian hyperplanes.
	\end{proof}
	
	We finish this subsection with the proofs of Lemmas~\ref{lemma:flatDeformations} and \ref{lemma:InvariantConicalLagrangian}.
	%To prove Lemma~\ref{lemma:flatDeformations} we use an repetitive argument.
	\begin{proof}[Proof of Lemma~\ref{lemma:flatDeformations}]
	Define subtori $T_0,\ldots,T_{n+2}\subset T'$ as
	\[
	T_i:=\{(t_1,\ldots,t_i,1,\ldots,1)\in T| t_1,\ldots,t_i\in\mathbb C^\ast\},\quad i=0,\ldots,n+2.
	\]
	Note that $T_0$ is the trivial subgroup and $T_{n+2}=T'$. We define co\-cha\-rac\-ters
	$\sigma_1,\ldots,\sigma_{n+2}:\mathbb C^\ast\rightarrow T'$ as
	\begin{equation}\label{eq:DefinitionOfCocharactersWRTSymplecticStructure}
	\sigma_i(t)_j=\begin{cases}
	t &\textup{if }j=i,\\
	1 &\textup{else},
	\end{cases}
	\end{equation}
	where $j=1,\ldots,n+2$. 
	\begin{claim}%\label{lemma:flatDeformationLemma}
	Let $i\in\{1,\ldots,n+2\}$ and $C\subset \mathbb C^{2n}$ be a $T_{i-1}$- and $A$-invariant, conical, lagrangian subvariety. Then, there exist irreducible, $T_i$-invariant, conical, lagrangian subvarieties $Z_1,\ldots,Z_r\subset V$ and natural numbers $m_1,\ldots,m_r$ such that $[C]^A=\sum_{i=1}^{r} m_i [Z_i]^A$ in $H^\ast_A(\mathbb C^{2n})$.
	\end{claim}
	\begin{proof}[Proof of Claim]
	Let $Y\subset \mathbb C\setminus\{0\}\times\mathbb C^{2n}$ be the flat family corresponding to the $\sigma_i$-orbit of $C$. Since, $\sigma_i$ scales the symplectic form on $\mathbb C^{2n}$, this is a family of $T_{i-1}$- and $A$-invariant, conical, lagrangian subvarieties of $\mathbb C^{2n}$. 
	Let $\overline Y\subset  \mathbb P^1\times\mathbb C^{2n}$ be the closure of $Y$ under the standard embedding $ \mathbb C\setminus\{0\}\times\mathbb C^{2n}\hookrightarrow \mathbb P^1\times\mathbb C^{2n}$.
	 According to~\cite[Proposition~II.9.7]{hartshorne1977algebraic}, $\overline Y$ is  a flat family over $\mathbb P^1$. Let $\pi:\mathbb P^1\times\mathbb C^{2n}\rightarrow \mathbb P^{1}$ be the projection and set $Z:=(\pi^{-1}(0))_{\mathrm{red}}$. As $\mathbb C\setminus\{0\}\times\{0\}\subset Y$, we have $0\in Z$. So, in particular, $Z$ is non-empty. By its construction as limit of the $\sigma_i$-action, it follows that $Z$ is $\sigma_i$-invariant. Moreover, since $Y$ is a $T_{i-1}$- and $A$-invariant, conical subvariety, we conclude that also $Z$ is a $T_{i-1}$- and $A$-invariant, conical subvariety. Finally, by same approximation argument as in the proof of Proposition~\ref{prop:AnalyticTangentConeLagrangian}, we conclude that $Z$ is also lagrangian.
	
	Since $\pi^{-1}(1)=C$, we obtain an $A$-equivariant rational equivalence between $C$ and $\pi^{-1}(0)$ provided by $Z$. It follows that $[C]^A=[\pi^{-1}(0)]^A$. Then the irreducible components of $Z$, say $Z_1,\ldots,Z_r$,  are also $T_{i}$- and $A$-invariant, conical, lagrangian subvarieties and $[\pi^{-1}(0)]^A=\sum_{i=1}^{r} m_i [Z_i]^A$, where  $m_i$ is the geometric multiplicity of $Z_i$.
	\end{proof}
	Using the above claim we can now easily deduce Lemma~\ref{lemma:flatDeformations} using a repetitive argument. Applying the claim to $T_1$ and $C$ we conclude that there exist irreducible, $T_1$- and $A$-invariant, conical, lagrangian subvarieties $Z_{1,1},\ldots,Z_{1,r_1}\subset \mathbb C^{2n}$ as well as natural numbers $m_{1,1},\ldots,m_{1,r_1}$ such that $[C]^A=\sum_{i=1}^{r_1} m_{l,i} [Z_{1,i}]^A$. Now, repeat this procedure by applying the claim to $T_2$ and $Z_{1,1},\ldots,Z_{1,r_1}$ and repeat this procedure. After $n+2$ repetitions, we obtain subvarieties $Z_1,\ldots,Z_r\subset \mathbb C^{2n}$ satisfying the desired conditions of Lemma~\ref{lemma:flatDeformations}. 
	\end{proof}
%	After this preparation, we now give a proof of Proposition~\ref{prop:flatDeformations}.
%	
%	\begin{proof}[Proof of Proposition~\ref{prop:flatDeformations}]
%	By Lemma~\ref{lemma:flatDeformationLemma} there exist irreducible, $T_1$- and $A$-invariant, conical, lagrangian subvarieties $Z_{1,1},\ldots,Z_{1,r_1}\subset \mathbb C^{2n}$ and natural numbers $m_{1,1},\ldots,m_{1,r_1}$ such that $[C]^A=\sum_{i=1}^{r_1} m_{l,i} [Z_{1,i}]^A$. Now, repeat this procedure by applying Lemma~\ref{lemma:flatDeformationLemma} to $Z_{1,1},\ldots,Z_{1,r_1}$ and so on. After $n+2$ repetitions, we obtain subvarieties $Z_1,\ldots,Z_r\subset \mathbb C^{2n}$ satisfying the desired conditions of Proposition~\ref{prop:flatDeformations}.
%	\end{proof}
	
	%The next lemma proves that irreducible, $T$-invariant, lagrangian subvarieties of $\mathbb C^{2n}$ are actually linear lagrangian subspaces.
	
%	\begin{lemma}\label{lemma:InvariantConicalLagrangian} Let $Z\subset \mathbb C^{2n}$ be an irreducible, $T$-invariant, lagrangian subvariety. Then, there exist $v_1,\ldots,v_n\in \mathbb C^{2n}$ with $v_i\in\{e_i,e_i'\}$ for $i=1,\ldots,n$ and we have $Z=\mathrm{span}_{\mathbb C}(v_1,\ldots,v_n)$. 
%	\end{lemma}
	\begin{proof}[Proof of Lemma~\ref{lemma:InvariantConicalLagrangian}]
	Since $Z$ is of dimension $n$, there exists a smooth point
	\[
	z=(z_1,\ldots,z_n,z_1',\ldots,z_n')\in Z_{\mathrm{red}}
	\]
	such that at least $n$ coordinates of $z$ are non-zero. We show that for each $i\in\{1,\ldots,n\}$ exactly one of the coordinates $z_i,z_i'$ is zero and the other non-zero.
	For the sake of contradiction, suppose that there exists $i\in\{1,\ldots,n\}$ such that $z_i$ and $z_i'$ are both non-zero. 
	Let $\sigma_i:\mathbb C^\ast\rightarrow T'$ be the cocharacter from \eqref{eq:DefinitionOfCocharactersWRTSymplecticStructure} and $\sigma:\mathbb C^\ast\rightarrow T'$ be the cocharacter given by
	\[
	\sigma(t)_j=\begin{cases} 1 &\textup{if }j\ne i,n+2, \\
	t &\textup{if }j=i,\\
	t^{-1}&\textup{if }j=n+2.
	\end{cases}
	%\quad
	%\sigma(t)_j= \begin{cases}
	%1 &\textup{if }j\ne i,n+2, \\
	%t &\textup{if }j=n+1,\\
	%t^{-1}&\textup{if }j=i.
	%\end{cases}
	\]
	Then, the pairing of tangent vectors at $z$ corresponding to the $\sigma$- and the $\sigma_i$-orbit under $\omega$ is non-zero which contradicts the assumption that $Z$ is lagrangian.
	
	For $i\in\{1,\ldots,n\}$, we define
	\[
	v_i=\begin{cases}
	e_i &\textup{if }z_i\ne 0,\\
	e_i'&\textup{if }z_i'\ne 0.
	\end{cases}
	\]
	As $Z$ is $T$-invariant, we conclude $\mathrm{span}_{\mathbb C}(v_1,\ldots,v_n)\subset Z$. Since $Z$ is reduced, irreducible and of dimension $n$, the inclusion must be an equality.
	\end{proof}

	\subsection{Example of construction of stable envelopes}\label{subsection:ExampleConstructionStableEnvelopes}
	
	We illustrate the construction of stable envelopes from Subsection~\ref{subsection:ProofExistenceStableEnvelopes} for the bow variety $\mathcal C(0\textcolor{red}{\slash}1\textcolor{blue}{\backslash}1\textcolor{red}{\slash}2\textcolor{blue}{\backslash}2\textcolor{blue}{\backslash}2\textcolor{red}{\slash}0)$. Let $\mathfrak C$ be the chamber containing the cocharacter $\sigma_0$. To compute for instance $\mathrm{Stab}_{\mathfrak C}(x_{D_3})$ we start with $\gamma_{3,3}:=[\overline{\operatorname{Attr}_{\sigma_0}(x_{D_3})}]^{\mathbb T}$. 
	According to Table~\ref{table:EquivariantMultiplicitiesAttractingCells}, we have
	\[
	\iota_{x_{D_2}}^\ast(\gamma_{3,3}) = (t_2-t_1+h)(t_2-t_3+h),\quad \iota_{x_{D_2}}^\ast([\overline{\operatorname{Attr}_{\sigma_0}(x_{D_2})}]^{\mathbb T}) = (t_1-t_2)(t_2-t_3+h).
	\]
	Thus, we set $\gamma_{3,2}=[\overline{\operatorname{Attr}_{\sigma_0}(x_{D_3})}]^{\mathbb T}+[\overline{\operatorname{Attr}_{\sigma_0}(x_{D_2})}]^{\mathbb T}$. By construction,
	\[
	\iota_{x_{D_1}}^\ast(\gamma_{3,2})= h(t_3-t_2+h).
	\]
	So $\gamma_{3,2}$ is already divisible by $h$ and hence we have 
	\[
	\mathrm{Stab}_{\mathfrak C}(x_{D_3}) = \gamma_{3,2} = [\overline{\operatorname{Attr}_{\sigma_0}(x_{D_3})}]^{\mathbb T}+[\overline{\operatorname{Attr}_{\sigma_0}(x_{D_2})}]^{\mathbb T}.
	\]
	The other stable basis elements can be computed in exactly the same way using Table~\ref{table:EquivariantMultiplicitiesAttractingCells}. They are given by
	\begin{align*}
	\mathrm{Stab}_{\mathfrak C}(x_{D_1}) &= [\overline{\operatorname{Attr}_{\mathfrak C}(x_{D_1})}]^{\mathbb T},\\
	\mathrm{Stab}_{\mathfrak C}(x_{D_2}) &= [\overline{\operatorname{Attr}_{\mathfrak C}(x_{D_2})}]^{\mathbb T}+[\overline{\operatorname{Attr}_{\mathfrak C}(x_{D_3})}]^{\mathbb T},\\
	\mathrm{Stab}_{\mathfrak C}(x_{D_4}) &= [\overline{\operatorname{Attr}_{\mathfrak C}(x_{D_4})}]^{\mathbb T}+[\overline{\operatorname{Attr}_{\mathfrak C}(x_{D_3})}]^{\mathbb T}+[\overline{\operatorname{Attr}_{\mathfrak C}(x_{D_2})}]^{\mathbb T},\\
	\mathrm{Stab}_{\mathfrak C}(x_{D_5}) &= [\overline{\operatorname{Attr}_{\mathfrak C}(x_{D_5})}]^{\mathbb T}+[\overline{\operatorname{Attr}_{\mathfrak C}(x_{D_4})}]^{\mathbb T}-[\overline{\operatorname{Attr}_{\mathfrak C}(x_{D_2})}]^{\mathbb T}.
	\end{align*}

	\section{Orthogonality of stable basis elements}\label{section:orthogonality}
	
	%Following~\cite[Theorem~4.4.1]{maulik2019quantum}, we prove that stable envelopes corresponding to opposite chambers are orthogonal with respect to the virtual intersection pairing on $\mathcal C(\mathcal D)$. 
	The Orthogonality Theorem from \cite[Theorem~4.4.1]{maulik2019quantum} states that stable basis elements corresponding to opposite chambers are orthogonal with respect to the virtual intersection pairing on $\mathcal C(\mathcal D)$. We finish with a proof of this result with a view towards explicit calculations.
	
	\begin{theorem}[Orthogonality Theorem]\label{thm:Orthogonality} We have
\[
(\mathrm{Stab}_{\mathfrak C}(p),\mathrm{Stab}_{{\mathfrak C}^{\mathrm{op}}}(q))_{\operatorname{virt}} = \begin{cases} 1&\textit{if }p=q, \\
0&\textit{if }p\ne q.
\end{cases}
\]
for all $p,q\in \mathcal C(\mathcal D)^{\mathbb T}$.
	\end{theorem}
	The \textit{virtual intersection pairing} is defined as
	\[
		(.,.)_{\mathrm{virt}}: H^\ast_{\mathbb T}(\mathcal C(\mathcal D))\times H^\ast_{\mathbb T}(\mathcal C(\mathcal D))  \rightarrow H^\ast_{\mathbb T}(\mathrm{pt})_{\mathrm{loc}},
		\quad (\beta,\gamma)\mapsto \sum_{p\in \mathcal C(\mathcal D)^{\mathbb T}} \frac{\iota_p^\ast(\beta \cup \gamma)}{e_{\mathbb T}(T_p\mathcal C(\mathcal D))}.
	\]
	This definition is motivated by the Atiyah--Bott--Berline--Vergne integration formula for smooth projective varieties, see e.g. \cite[Theorem~2.10]{anderson2012introduction}.
	
	The Orthogonality Theorem describes a parallel between stable basis elements and Schubert classes which also have an analogous orthogonality property. Moreover, it is useful for the concrete computation of multiplication matrices of equivariant cohomology classes in $H_\ast^{\mathbb T}(\mathcal C(\mathcal D))$ with respect to the stable envelope basis. For example, in~\cite{wehrhan2023chevalley} the Orthogonality Theorem is used to compute the multiplication matrices of first Chern classes of tautological bundles which can be seen as Chevalley--Monk formulas for bow varieties.

	\subsection{Polynomiality of matrix coefficients}
	Before we prove the Orthogonality Theorem, we first prove the following proposition:
	\begin{prop}[Polynomiality]\label{prop:nonlocalizedpairing}
	We have 
	\[
	(\mathrm{Stab}_{\mathfrak C}(p)\cup\gamma,\mathrm{Stab}_{{\mathfrak C}^{\mathrm{op}}}(q))_{\mathrm{virt}}\in H_{\mathbb T}^\ast(\operatorname{pt}),
	\]
	for all $\gamma\in H_{\mathbb T}^\ast(\mathcal C(\mathcal D))$ and $p,q\in\mathcal C(\mathcal D)^{\mathbb T}$.
	\end{prop}
	
%	We first make some general observations. Let $T=(\mathbb C^\ast)^r$ be a torus and $X$ a smooth quasi-projective variety endowed with an algebraic $T$-action such that $X$ admits finitely many torus fixed points. Then, again motivated by the Atiyah--Bott--Berline--Vergne formula the \textit{virtual integration map} is defined as
%	\[
%	\varepsilon_{X,\mathrm{virt},\ast}: H^\ast_{ T}(X) \rightarrow H^\ast_{ T}						(\mathrm{pt})_{\mathrm{loc}},
%	\quad
%	\gamma\mapsto
%	\sum_{p\in X^{T}} \frac{\iota_p^\ast(\gamma)}{e_T(T_pX)}.
%	\]
%	
%	\begin{lemma}\label{lemma:virtpushforward}
%	Let $Y$ be a proper variety with a $T$-action and $i:Y\hookrightarrow X$ a $T$-equivariant closed immersion. Let $\varepsilon_Y:Y\rightarrow\mathrm{pt}$ be the canonical morphism and $\gamma\in H_{ T}^\ast(Y)$. Then, we have
%\[
%\varepsilon_{Y,\ast}(\gamma) = \varepsilon_{X,\mathrm{vitr},\ast}i_\ast(\gamma)),
%\]
%where $\varepsilon_{Y,\ast}$ and $i_\ast$ are the proper pushforwards in $\mathbb T$-equivariant cohomology.
%	\end{lemma}	
	\begin{proof}[Proof of Proposition~\ref{prop:nonlocalizedpairing}]
Set $X:=\mathrm{Attr}^f_{\mathfrak C}(p)\cap \mathrm{Attr}^f_{{\mathfrak C}^{\mathrm{op}}}(q)$ and \[\beta:=\mathrm{Stab}_{\mathfrak C}(p)\cup\gamma\cup\mathrm{Stab}_{{\mathfrak C}^{\mathrm{op}}}(q). \] 
By definition, $\mathrm{Stab}_{\mathfrak C}(p)$ is supported on $\mathrm{Attr}^f_{{\mathfrak C}}(p)$ and $\mathrm{Stab}_{{\mathfrak C}^{\mathrm{op}}}(p)$ is supported on $\mathrm{Attr}^f_{{\mathfrak C}^{\mathrm{op}}}(q)$. It follows that $\beta$ is supported on $X$. Thus, there exists
$\alpha\in \overline{H}_\ast^{\mathbb T}(X)$ such that $\iota_\ast(\alpha)$ is Poincar\'e dual to $\beta$. Now, as in the proof of Proposition~\ref{prop:affinestructure}, choose a smooth $\mathbb T$-equivariant compactification $\mathcal C(\mathcal D)\hookrightarrow Y$. Let $C_1,\ldots,C_r$ be the $\mathbb T$-fixed components of $Y$. By~\cite{iversen1972fixed}, $C_1,\ldots,C_r$ are smooth varieties. We denote by $N_{C_1},\ldots,N_{C_r}$ the respective normal bundles and 
let $\kappa:X\hookrightarrow Y$ be the inclusion. Using the Atiyah--Bott--Berline--Vergne integration formula and the fact that $X$ can only intersect fixed point components that are contained in $\mathcal C(\mathcal D)$, we obtain
\begin{align*}
\int_Y \kappa_\ast(\alpha) &=  \sum_{i=1}^r \int_{C_i} \frac{\kappa_\ast(\alpha)}{e_{\mathbb T}(N_{C_i})}
= \sum_{p\in\mathcal C(\mathcal D)^{\mathbb T}} \int_{p} \frac{\kappa_\ast(\alpha)}{e_{\mathbb T}(T_p\mathcal C(\mathcal D))} 
= \sum_{p\in\mathcal C(\mathcal D)^{\mathbb T}} \frac{\iota_p^\ast (\beta)}{e_{\mathbb T}(T_p\mathcal C(\mathcal D))} \\
&= (\mathrm{Stab}_{\mathfrak C}(p)\cup\gamma,\mathrm{Stab}_{{\mathfrak C}^{\mathrm{op}}}(q))_{\mathrm{virt}}.
\end{align*}
Since $\int_Y \kappa_\ast(\alpha)$ is contained in $\overline H_\ast^{\mathbb T}(\mathrm{pt})\cong H_{\mathbb T}^\ast (\mathrm{pt})$, the proof is finished.
%there exists $\beta\in H^\ast_{\mathbb T}(X)$ such that 
%\[
%i_\ast(\beta)=\mathrm{Stab}_{\mathfrak C}(p)\cup\gamma\cup\mathrm{Stab}_{{\mathfrak C}^{\mathrm{op}}}(q).
%\]
%According to Corollary~\ref{cor:slopesupport}, $X$ is a proper closed subvariety of $\mathcal C(\mathcal D)$. Thus, Lemma~\ref{lemma:virtpushforward} gives
%\[
%(\operatorname{Stab}_{\mathfrak C}(p)\cup\gamma,\operatorname{Stab}_{{\mathfrak C}^{\operatorname{op}}}(q))_{\operatorname{virt}} = \varepsilon_{X,\ast}(\beta) \in H_{\mathbb T}^\ast(\operatorname{pt}),
%\]
%which completes the proof.
\end{proof}

	\subsection{Proof of the Orthogonality Theorem}%~\ref{thm:Orthogonality}}
\begin{proof}[Proof of Theorem~\ref{thm:Orthogonality}] By definition of the virtual intersection form,
\begin{equation}\label{eq:stabpairing}
(\mathrm{Stab}_{\mathfrak C}(p),\mathrm{Stab}_{{\mathfrak C}^{\mathrm{op}}}(q))_{\mathrm{virt}} = \sum_{z\in\mathcal C(\mathcal D)^{\mathbb T}} \frac{\iota_z^{\ast}(\mathrm{Stab}_{\mathfrak C}(p)) \cup \iota_z^\ast(\mathrm{Stab}_{{\mathfrak C}^{\mathrm{op}}}(q)) }{e_{\mathbb T}(T_z\mathcal C(\mathcal D))}.
\end{equation}
Proposition~\ref{prop:nonlocalizedpairing} implies that~\eqref{eq:stabpairing} is actually contained in $H_{\mathbb T}^\ast(\mathrm{pt})$. If $p\ne q$, we know by the smallness axiom that $h$ divides $\iota_z^{\ast}(\mathrm{Stab}_{\mathfrak C}(p)) \cup \iota_z^\ast(\mathrm{Stab}_{{\mathfrak C}^{\mathrm{op}}}(q))$ for all $z\in \mathcal C(\mathcal D)^{\mathbb T}$. However, Corollary~\ref{cor:tangentweights} gives $h\nmid e_{\mathbb T}(T_z\mathcal C(\mathcal D))$ for all $z\in \mathcal C(\mathcal D)^{\mathbb T}$. It follows that~\eqref{eq:stabpairing} is divisible by $h$. As $\iota_z^{\ast}(\mathrm{Stab}_{\mathfrak C}(p)) \cup \iota_z^\ast(\mathrm{Stab}_{{\mathfrak C}^{\mathrm{op}}}(q))$ and $e_{\mathbb T}(T_z\mathcal C(\mathcal D))$ are homogeneous of the same degree, we conclude that~\eqref{eq:stabpairing} is a degree $0$ polynomial in the equivariant parameters. Hence, \eqref{eq:stabpairing} has to vanish. Now, let us consider the case $p=q$. By the normalization axiom, we can infer
\[
\frac{i_p^\ast (\mathrm{Stab}_{\mathfrak C}(p)) \cup  \iota_p^\ast(\mathrm{Stab}_{{\mathfrak C}^{\mathrm{op}}}(p))} {e_{\mathbb T}(T_p\mathcal C(\mathcal D))}= \frac{e_{\mathbb T}(T_p\mathcal C(\mathcal D)_{\mathfrak C}^- )\cup e_{\mathbb T}(T_p\mathcal C(\mathcal D)_{\mathfrak C}^+)} {e_{\mathbb T}(T_p\mathcal C(\mathcal D))}=1.
\]
In addition, the same argument as in the case $p\ne q$ gives
\[
\sum_{\substack{z\in\mathcal C(\mathcal D)^{\mathbb T}\\ z\ne p} } = \frac{\iota_z^{\ast}(\mathrm{Stab}_{\mathfrak C}(p)) \cup \iota_z^\ast(\mathrm{Stab}_{{\mathfrak C}^{\mathrm{op}}}(p)) }{e_{\mathbb T}(T_z\mathcal C(\mathcal D))}=0.
\]
Thus, we deduce 
\[
\sum_{z\in\mathcal C(\mathcal D)^{\mathbb T}} \frac{\iota_z^{\ast}(\mathrm{Stab}_{\mathfrak C}(p)) \cup \iota_z^\ast(\mathrm{Stab}_{{\mathfrak C}^{\mathrm{op}}}(p)) }{e_{\mathbb T}(T_z\mathcal C(\mathcal D))}=1.
\]
This finishes the proof of the Orthogonality Theorem.
\end{proof}
	
\appendix
\section{Partial orders of opposite chambers}\label{section:Appendix}

\subsection{Linearized embeddings and attracting cells}

Let $\mathcal C(\mathcal D)$ be a bow variety with $\mathbb T=\mathbb A\times \mathbb C^\ast_h$-action from Subsection~\ref{subsection:torusAction} and let $\sigma:\mathbb C^\ast\rightarrow \mathbb A$ a generic cocharater with chamber $\mathfrak C$. Via $\sigma$, we view $\mathcal C(\mathcal D)$ as $\mathbb C^\ast$-variety. Since $\mathcal C(\mathcal D)$ is a smooth and quasi-projective variety, there exists, by \cite[Theorem~1]{sumihiro1974equivariant}, a locally closed $\mathbb C^\ast$-equivariant embedding $\iota_\sigma:\mathcal C(\mathcal D)\hookrightarrow \mathbb P(V)$ for some finite dimensional $\mathbb C^\ast$-representation $V$. For given $p\in \mathcal C(\mathcal D)^{\mathbb T}$, we denote by $Z_{\sigma,p}$ the Zariski closure of $\iota_\sigma(\mathrm{Attr}_\sigma(p))$ in $\mathbb P(V)$. Thus, $Z_{\sigma,p}$ is a closed $\mathbb C^\ast$-invariant subvariety of $\mathbb P(V)$ that contains $\iota_\sigma(\mathrm{Attr}_\sigma(p))$ as open dense $\mathbb C^\ast$-invariant subvariety.

We have the weight space decomposition \[
V=\bigoplus_{a\in \mathbb Z}V_a, \quad\textup{where}\quad V_a=\{v\in V\mid t.v=t^av,\;\textup{for all $t\in\mathbb C^\ast$} \}.
\]
We denote the dimension of $V_a$ by $n_a$. The $\mathbb C^\ast$-fixed point locus of $\mathbb P(V)$ is given as
\[
\mathbb P(V)^{\mathbb C^\ast}=\{[v]\mid v\in V_a\textup{ for some $a\in \mathbb Z$}\}.
\]
Given $a\in \mathbb Z$ and $v\in V_a$, the attracting cell of $[v]$ in $\mathbb P(V)$ equals
\begin{equation}\label{eq:AttractingCellProjectiveSpace} 
\{x\in\mathbb P(V)\mid \lim_{t\to0}t.x =[v] \}= \{ [v+w]\mid w\in \bigoplus_{a'>a}V_{a'} \}.
\end{equation}
Its Zariski closure in $\mathbb P(V)$ is the projective subspace $\mathbb P(\mathrm{span}_{\mathbb C}(v) \oplus \bigoplus_{a'>a}V_{a'})$.  

For each $p\in \mathcal C(\mathcal D)^{\mathbb T}$ there exists a weight vector $v_p$ such that $\iota_\sigma(p)=[v_p]$. Let $a_p$ be the weight of $v_p$. Suppose $p\in \overline{\mathrm{Attr}_\sigma(q)}$ for some $q\in\mathcal C(\mathcal D)^{\mathbb T}$. Then \eqref{eq:AttractingCellProjectiveSpace} implies $a_q\le a_p$ and we have equality if and only if $p=q$. 

\subsection{Proof of Proposition~\ref{prop:OppositePartialOrdering}} We begin with the following lemma:

\begin{lemma}\label{lemma:OppositeAttractionOrdering}
For given $p,q\in \mathcal C(\mathcal D)^{\mathbb T}$ with $p\in \overline{\mathrm{Attr}_\sigma(q)}$ and $p\ne q$ there exists $p'\in \mathcal C(\mathcal D)^{\mathbb T}$ with $p'\in \overline{\mathrm{Attr}_\sigma(q)} \cap \overline{\mathrm{Attr}_{\sigma^{-1}}(p)}$ and $a_q\le a_{p'}<a_p$.
\end{lemma}

\begin{proof}[Proof of Proposition~\ref{prop:OppositePartialOrdering}] Assuming Lemma~\ref{lemma:OppositeAttractionOrdering}, let $p,q\in \mathcal C(\mathcal D)^{\mathbb T}$ be distinct with $p\preceq_{\mathfrak C} q$. Thus, by definition of $\preceq_{\mathfrak C}$, there exists pairwise distinct $q_1,\ldots,q_r\in \mathcal C(\mathcal D)^{\mathbb T}$ with $q_1=q$, $q_r=p$ and $q_{i+1}\in \overline{\mathrm{Attr}_\sigma(q_i)}$ for all $i$. In order to show $q\preceq_{\mathfrak C^{\operatorname{op}}} p$, we prove that $q_{i}\in \overline{\mathrm{Attr}_{\sigma^{-1}}(q_{i+1})}$ for all $i$. 
For given $i$, there exists, by Lemma~\ref{lemma:OppositeAttractionOrdering}, a sequence $p_{i,n}$ in $\overline{\mathrm{Attr}_\sigma(q)}\cap \mathcal C(\mathcal D)^{\mathbb T}$ such that
\begin{enumerate}[label=(\alph*)]
\item $p_{i,n} \in \overline{\mathrm{Attr}_\sigma(q_i)} \cap \overline{\mathrm{Attr}_{\sigma^{-1}}(q_{i+1})}$ for all $n$,
\item $a_{p_{i,n}}=a_{q_i}$ for almost all $n$.
\end{enumerate}
Since $\{p'\in \overline{\mathrm{Attr}_\sigma(q_i)}\cap \mathcal C(\mathcal D)^{\mathbb T} \mid a_{p'}=a_{q_i}\}=\{q_i\}$, we have $p_{i,n}=q_i$ for almost all $n$ which yields $q_i\in \overline{\mathrm{Attr}_{\sigma^{-1}}(q_{i+1})}$.
\end{proof}

To prove Lemma~\ref{lemma:OppositeAttractionOrdering}, we use an analytic limit argument. Let $p,q\in \mathcal C(\mathcal D)^{\mathbb T}$ with $p\in \overline{\mathrm{Attr}_\sigma(q)}$ and $p\ne q$. We choose bases $(v_{a,1},\ldots,v_{a,n_a})$ of the weight spaces $V_a$. Without loss of generality, $v_p=v_{a_p,1}$ and $v_q=v_{a_q,1}$. Moreover, let 
\begin{equation}\label{eq:ProofOfLemma:OppositeAttractionOrderingDefinitionW}
W:=\Big(\bigoplus_{a<a_p}V_a\Big) \oplus \mathrm{span}_{\mathbb C}(v_{a_p,2},\ldots,v_{a_p,n_{a_p}}) \oplus  \Big(\bigoplus_{a>a_p}V_a\Big)
\end{equation}
and let $U_p=\{[v_p+w]\mid w\in W\}\subset \mathbb P(V)$ be the coordinate chart with origin $[v_p]$. We have that $U_p$ is $\mathbb C^\ast$-invariant and $t.[v_p+v]=v_p+t^{a-a_p}v$ for all $v\in V_a$, $a\in\mathbb Z$ and $t\in\mathbb C^\ast$. We equip $W$ with a hermitian product with unitary basis given by \eqref{eq:ProofOfLemma:OppositeAttractionOrderingDefinitionW} and thus view $W$ as metric space.
Via the isomorphism of (analytic) varieties $W\xrightarrow\sim U_p,w\mapsto [v_p+w]$, we also view $U_p$ as metric space and denote by $|\cdot|$ the induced absolute value and by $\operatorname{dist}(.,.)$ the induced distance function on $U_p$.

We set
\[
W':=\{ [v_p+\lambda v_q+w] \mid \lambda\in\mathbb C, w\in\bigoplus_{a_q<a<a_p}V_a \}\subset U_p.
\]
Note that $W'$ is a $\mathbb C^\ast$-invariant linear subspace of $U_p$.

\begin{proof}[Proof of Lemma~\ref{lemma:OppositeAttractionOrdering}] We want to construct a sequence of elements in $\mathrm{Attr}_\sigma(q)\cap U_p$ which approaches $W'$ but is far away from $[v_p]$. First, we show that for all $\varepsilon>0$ there exists a $z\in \iota_\sigma(\mathrm{Attr}_\sigma(q))\cap U_p$ such that
\begin{equation}\label{eq:ExistenceOfSuitableElements}
|z|\in [1,2]\quad \textit{and}\quad \operatorname{dist}(z,W')<\varepsilon.
\end{equation}
By Lemma~\ref{lemma:ExistenceOfSuitablePaths} below, there exists a path $\gamma:[0,1]\rightarrow Z_p\cap U_p$, continuous in the analytic topology, such that $\gamma([0,1))\subset \iota(\mathrm{Attr}_\sigma(q))$ and $\gamma(1)=[v_p]$. According to our choice of basis, we can write
\begin{align*}
\gamma(s)= \Bigl[\gamma_{a_q}(s)v
+\Big( \sum_{a_q<a'<a_p}\sum_{i=1}^{n_{a'}} &\gamma_{a',i}(s)v_{a',i}
\Big)+v_p \\
&+
\Big(
\sum_{i=2}^{n_{a_p}} \gamma_{a_p,i}(s) v_{a_p,i}
\Big)
+
\Big( \sum_{a'>a_p}\sum_{i=1}^{n_{a'}} \gamma_{a',i}(s)v_{a',i}
\Big)
\Bigr].
\end{align*}
The property $\gamma([0,1))\subset \iota(\mathrm{Attr}_\sigma(q))$ implies $\gamma_{a_q}(s)\ne 0$ for $s\in [0,1)$. Since $\gamma(1)=[v_p]$, we have $\gamma_{i,j}(s)\to 0$ for $s\to 1$ for all $i,j$. Hence, we may assume that all $\gamma_{a_{p+i},j}$ satisfy for all $s\in [0,1]$
\begin{equation}\label{eq:PathCoefficientsSufficientlySmall}
|\gamma_{a_{p+i},j}(s)|<n^{-1} \varepsilon,
\quad \textup{where 
$n=\sum_{a'\geq a_p} n_{a'}$.}
\end{equation}
Choose $t_0\in\mathbb C^\ast$ with $|t_0|<1$ such that $|t_0^{a_p-a_q}\gamma_{a_q}(0)|>2$. Then, also $|t_0.\gamma(0)|>2$. Thus, as $t_0.\gamma(1)=[v_p]$, the intermediate value theorem implies that there exists $s_0\in (0,1)$ such that 
$|t_0.\gamma(s_0)|\in [1,2]$. In addition, \eqref{eq:PathCoefficientsSufficientlySmall} yields
\[
\operatorname{dist}(t_0.\gamma(s_0),W')
= \Big| \Big(
\sum_{i=2}^{n_{a_p}} \gamma_{a_p,i}(s) v_{a_p,i}
\Big)
+
\Big( \sum_{a'>a_p}\sum_{i=1}^{n_{a'}} t_0^{a'-a_p} \gamma_{a',i}(s)v_{a',i}
\Big)
\Big| < \varepsilon.
\]
Hence $z:=t_0.\gamma(s_0)$ satisfies \eqref{eq:ExistenceOfSuitableElements}. Since $\iota_\sigma(\mathrm{Attr}_\sigma(q))\cap U_p$ is $\mathbb C^\ast$-invariant, we conclude $z\in \iota_\sigma(\mathrm{Attr}_\sigma(q))$. Thus, $z$ satisfies all desired properties. 

As a direct consequence of \eqref{eq:ExistenceOfSuitableElements}, we conclude that there exist a sequence $z_n\in \iota(\mathrm{Attr}_\sigma(q))$ such that 
$
\operatorname{dist}(z_n,[v_p])\in [1,2]
$,
for all $n$ and $\operatorname{dist}(z_n,W')\to 0$, for $n\to\infty$. By the Heine--Borel theorem, $z_n$ has a convergent subsequence with limit $z'\in U_p\cap Z_q$. As $\operatorname{dist}(z_n,W')\to 0$, we also have $z'\in W'$. The condition $\operatorname{dist}(z_n,[v_p])\in [1,2]$ yields $z'\ne[v_p]$. So by the definition of $W'$, we can write
\begin{equation}\label{eq:ZPrimeExpression}
z'=[w_{a_q} + w_{a_q+1}+\ldots + w_{a_p-1} + v_p ],\quad  w_{a_q}\in \mathrm{span}_{\mathbb C}(v_q) \oplus w_{a_q+i}\in V_{a_q+i}\quad \textup{for $i>1$.}
\end{equation}
As $z'\ne[v_p]$, we have $w_{a_q+r}\ne 0$ for some $r\in\{0,\ldots,a_p-a_q+1\}$. Set
\begin{equation}\label{eq:ChoiceOfWeightVector}
r_0:=\mathrm{min}(\{ r\in\{0,\ldots,a_p-a_q+1\}\mid  w_{a_q+r}\ne 0\}).
\end{equation}
Then, $\lim_{t\to 0} t.z'=[w_{a_q+r_0}]$ and $\lim_{t\to \infty} t.z'=[v_p]$. Recall that $\iota_\sigma(\mathrm{Attr}_\sigma(q))$ is an open dense $\mathbb C^\ast$-subvariety of $Z_q$. Since $[v_p]$ is contained in the orbit closure $\overline{\mathbb C^\ast.z'}$, the intersection $\iota_\sigma(\mathcal C(\mathcal D))\cap \overline{\mathbb C^\ast.z'}$ is a non-empty open $\mathbb C^\ast$-invariant subvariety of $\overline{\mathbb C^\ast.z'}$. Hence, $z'\in \iota_\sigma(\mathcal C(\mathcal D))$. Since $z'\in Z_q$, we have $z'\in \iota_\sigma(\overline{\mathrm{Attr}_\sigma(q)})$. As $\lim_{t\to \infty} t.z'=[v_p]$, we also have $z'\in \iota_\sigma(\overline{\mathrm{Attr}_{\sigma^{-1}}(p)})$. By Theorem~\ref{thm:slopesupport}, $\overline{\mathrm{Attr}_\sigma(q)} \cap \overline{\mathrm{Attr}_{\sigma^{-1}}(p)}$ is a closed proper $\mathbb C^\ast$-invariant subvariety of $\mathcal C(\mathcal D)$ which implies 
$
\lim_{t\to 0} t.z'=[w_{a_q+r_0}] \in \iota_\sigma(\overline{\mathrm{Attr}_\sigma(q)} \cap \overline{\mathrm{Attr}_{\sigma^{-1}}(p)}).
$
Set $p':=\iota^{-1}_\sigma([w_{a_q+r_0}])$. Then, as $[w_{a_q+r_0}]$ is a $\mathbb C^\ast$-fixed point of $\mathbb P(V)$, we have $p'\in\mathcal C(\mathcal D)^\sigma$. The Generic Cocharacter Theorem then gives that $p'\in \mathcal C(\mathcal D)^{\mathbb T}$. Moreover, \eqref{eq:ChoiceOfWeightVector} yields $a_{p'}=a_q+r_0<a_p$ . So $p'$ satisfies all desired properties.
\end{proof}

\begin{remark}\label{imZug}
The above proof implies that the partial order $\preceq_{\mathfrak C}$ is completely determined by the $\mathbb T$-invariant curves of $\mathcal C(\mathcal D)$. Namely, given $p,q\in \mathcal C(\mathcal D)^{\mathbb T}$ such that $p\in\overline{\mathrm{Attr}_\sigma(q)}$, our construction of the element $z$ with the properties \eqref{eq:ExistenceOfSuitableElements} implies that there exist closed immersions $f_1,\ldots,f_r:\mathbb P^1\rightarrow \mathcal C(\mathcal D)$
such that all $f_i(\mathbb P^1)$ are $\sigma$-invariant satisfying $f_1(0)=q$, $f_r(\infty )=p$ and $f_{i+1}(0)=f_i(\infty)$ for $i=1,\ldots,r-1$. Using the same deformation argument as in Lemma~\ref{lemma:flatDeformations}, we then deduce that we can actually choose $f_1,\ldots,f_r$ such that all $f_i(\mathbb P^1)$ are $\mathbb T$-invariant. 
\end{remark}
\begin{remark}
Independently of our work,  Foster and Shou obtained in \cite{foster2023tangent} a classification of the $\mathbb T$-invariant curves from Remark~\ref{imZug}. In \cite{botta2023geometric} this is used to explicitly identify the partial order $\preceq_{\mathfrak C}$ with the secondary Bruhat order on $(0,1)$-matrices. Our approach is more direct allowing to establish the relevant properties of the ordering using only the existence, without a full knowledge, of the $\mathbb T$-invariant curves. 
\end{remark}

\subsection{Approximation of boundary points via paths} In the proof of Lemma~\ref{lemma:OppositeAttractionOrdering}, we used the following statement:

\begin{lemma}\label{lemma:ExistenceOfSuitablePaths}
Let $X$ be a smooth algebraic variety of dimension $d$ which is embedded into a projective variety $Y$ as open dense subvariety. Then, for all $y\in Y\setminus X=:Z$ there exists a path $\gamma:[0,1]\rightarrow Y$ continuous with respect to the analytic topology on $Y$ such that $\gamma([0,1))\subset X$ and $\gamma(1)=y$.
\end{lemma}

\begin{proof}
By the monomalization theorem, see e.g. \cite[Theorem~3.35]{kollar2009lectures}, there exists a smooth projective variety $Y'$ and a morphism of varieties $f:Y'\rightarrow Y$ such that
\begin{enumerate}[label=(\alph*)]
\item $f$ restricts to an isomorphism $f^{-1}(X)\xrightarrow\sim X$,
\item $f^{-1}(Z)$ is a normal crossing divisor. 
\end{enumerate}
Thus, we may assume that $Z$ is a normal crossing divisor. Given $y\in Z$. Then, as $Z$ is a normal crossing divisor, there exists an analytic neighborhood of $y$ in $Y$ which is analytically isomorphic to a neighborhood $U$ of the origin in $\mathbb C^d$ such that under this isomorphism $y$ is identified with the origin and $Z$ equals the vanishing locus of the function $f_1\cdots f_r$, where $f_1,\ldots,f_r:U\rightarrow \mathbb C$ are holomorphic functions with $r\le d$ and $l_1,\ldots,l_r$ are linearly independent, where $l_i$ denotes the first order approximation of $f_i$. After applying a linear transformation, we may assume that $l_i$ is the projection the $i$-th coordinate in $\mathbb C^d$. By further shrinking $U$, we can thus assume that there exists a constant $C>0$ such that
\[
|f_i(z)-z_i|< C|z|^2,\quad z=(z_1,\ldots,z_d)\in U,i=1,\ldots,r.
\]
Hence, we conclude
$
\{z\in U\mid |z_i|>C|z|^2 \} \cap Z=\emptyset
$. It follows that $\mu(1,\ldots,1)\notin Z$ for $0<\mu<(C\sqrt{d})^{-1}$. 
By choosing $C$ large enough, we may assume that the closed ball centered at the origin with radius $C'=\frac12 (C\sqrt{d})^{-1}$ is entirely contained in $U$. Thus, if we set
\[
\gamma:[0,1]\rightarrow U,\quad s\mapsto sC'(1,\ldots,1)
\]
then $\gamma$ yields a path with the desired properties.
\end{proof}

	\bibliographystyle{alpha}
	\bibliography{References.bib}

\begin{thebibliography}{GKM98}

\bibitem[AF23]{anderson2023equivariant}
David Anderson and William Fulton.
\newblock {\em Equivariant cohomology in algebraic geometry}, volume 210 of
  {\em Cambridge Studies in Advanced Mathematics}.
\newblock Cambridge University Press, 2023.

\bibitem[And12]{anderson2012introduction}
Dave Anderson.
\newblock Introduction to equivariant cohomology in algebraic geometry.
\newblock In {\em Contributions to algebraic geometry.}, pages 71--92. European
  Mathematical Society, 2012.

\bibitem[AO17]{aganagic2017quasimap}
Mina Aganagic and Andrei Okounkov.
\newblock Quasimap counts and {Bethe} eigenfunctions.
\newblock {\em Mosc. Math. J.}, 17(4):565--600, 2017.

\bibitem[AO21]{aganagic2021elliptic}
Mina Aganagic and Andrei Okounkov.
\newblock Elliptic stable envelopes.
\newblock {\em J. Am. Math. Soc.}, 34(1):79--133, 2021.

\bibitem[AP93]{allday1993cohomological}
Christopher Allday and Volker Puppe.
\newblock {\em Cohomological methods in transformation groups}.
\newblock Cambridge University Press, 1993.

\bibitem[Aud04]{audin2004torus}
Michele Audin.
\newblock {\em Torus actions on symplectic manifolds}, volume~93 of {\em
  Progress in Mathematics}.
\newblock Springer, 2004.

\bibitem[BB73]{bialynicki1973some}
Andrzej Bia\l{}ynicki-Birula.
\newblock Some theorems on actions of algebraic groups.
\newblock {\em Ann. of Math. (2)}, 98(3):480--497, 1973.

\bibitem[BB74]{bialynickibirula1974fixed}
Andrzej Bia\l{}ynicki-Birula.
\newblock Fixed-points of torus actions on projective varieties.
\newblock {\em Bulletin de l'Acad\'emie polonaise des sciences. S\'erie des
  sciences math\'ematiques, astronomiques, et physiques}, 22:1097--1101, 1974.

\bibitem[BFR23]{botta2023geometric}
Tommaso Botta, Alex Foster, and Richard Rim\'anyi.
\newblock Geometric {Bruhat} order on (0,1)-matrices.
\newblock {\em arXiv preprint
  \href{https://arxiv.org/abs/2311.05531}{arXiv:2311.05531}}, 2023.

\bibitem[BR23]{botta2023mirror}
Tommaso~Maria Botta and Rich{\'a}rd Rim{\'a}nyi.
\newblock Bow varieties: stable envelopes and their 3d mirror symmetry.
\newblock {\em arXiv preprint
  \href{https://arxiv.org/abs/2308.07300}{arXiv:2308.07300}}, 2023.

\bibitem[Bri97]{brion1997equivariant}
Michel Brion.
\newblock Equivariant {Chow} groups for torus actions.
\newblock {\em Transform. Groups}, 2(3):225--267, 1997.

\bibitem[Che09]{cherkis2009moduli}
Sergey Cherkis.
\newblock Moduli spaces of instantons on the {Taub}-{NUT} space.
\newblock {\em Commun. Math. Phys.}, 290:719--736, 2009.

\bibitem[Che10]{cherkis2010instantons}
Sergey Cherkis.
\newblock Instantons on the {Taub}-{NUT} space.
\newblock {\em Adv. Theor. Math. Phys.}, 14(2):609--642, 2010.

\bibitem[Che11]{cherkis2011instantons}
Sergey Cherkis.
\newblock Instantons on gravitons.
\newblock {\em Commun. Math. Phys.}, 306:449--483, 2011.

\bibitem[EG96]{edidin1998equivariant}
Dan Edidin and William Graham.
\newblock Equivariant intersection theory.
\newblock {\em arXiv preprint
  \href{https://arxiv.org/abs/alg-geom/9609018}{arXiv:alg-geom/9609018}}, 1996.

\bibitem[Fis76]{fischer1976complex}
Gerd Fischer.
\newblock {\em Complex analytic geometry}, volume 538 of {\em Lecture Notes in
  Mathematics}.
\newblock Springer, 1976.

\bibitem[FRV18]{felder2018elliptic}
Giovanni Felder, Rich{\'a}rd Rim{\'a}nyi, and Alexander Varchenko.
\newblock Elliptic dynamical quantum groups and equivariant elliptic
  cohomology.
\newblock {\em SIGMA, Symmetry Integrability Geom. Methods Appl.}, 14:132,
  2018.

\bibitem[FS23]{foster2023tangent}
Alex Foster and Yiyan Shou.
\newblock Tangent weights and invariant curves in type {A} bow varieties.
\newblock {\em arXiv preprint
  \href{https://arxiv.org/abs/2310.04973}{arXiv:2310.04973}}, 2023.

\bibitem[Ful84]{fulton1984introduction}
William Fulton.
\newblock {\em Introduction to intersection theory in algebraic geometry}.
\newblock American mathematical society, 1984.

\bibitem[GK14]{gorbounov2014equivariant}
Vassily Gorbounov and Christian Korff.
\newblock Equivariant quantum cohomology and {Y}ang-{B}axter algebras.
\newblock {\em arXiv preprint
  \href{https://arxiv.org/abs/1402.2907}{arXiv:1402.2907}}, 2014.

\bibitem[GK17]{gorbounov2017quantum}
Vassily Gorbounov and Christian Korff.
\newblock Quantum integrability and generalised quantum {S}chubert calculus.
\newblock {\em Adv. Math.}, 313:282--356, 2017.

\bibitem[GKM98]{goresky1998equivariant}
Mark Goresky, Robert Kottwitz, and Robert MacPherson.
\newblock Equivariant cohomology, {K}oszul duality, and the localization
  theorem.
\newblock {\em Invent. Math.}, 131(1):25--84, 1998.

\bibitem[GKS20]{gorbounov2020yang}
Vassily Gorbounov, Christian Korff, and Catharina Stroppel.
\newblock Yang--{Baxter} algebras, convolution algebras, and {Grassmannians}.
\newblock {\em Russ. Math. Sur.}, 75(5):791--842, 2020.

\bibitem[GPR94]{grauert1994several}
Hans Grauert, Thomas Peternell, and Reinhold Remmert.
\newblock {\em Several complex variables {VII}}, volume~74 of {\em
  Encyclopaedia of Mathematical Sciences}.
\newblock Springer, 1994.

\bibitem[Har77]{hartshorne1977algebraic}
Robin Hartshorne.
\newblock {\em Algebraic geometry}, volume~52 of {\em Graduate Texts in
  Mathematics}.
\newblock Springer, 1977.

\bibitem[Hsi75]{hsiang1975cohomology}
Wu~Yi Hsiang.
\newblock {\em Cohomology theory of topological transformation groups}.
\newblock Springer, 1975.

\bibitem[Hur89]{hurtubise1989classification}
Jacques Hurtubise.
\newblock The classification of monopoles for the classical groups.
\newblock {\em Commun. Math. Phys.}, 120(4):613--641, 1989.

\bibitem[HW97]{HW97}
Amihay Hanany and Edward Witten.
\newblock Type {IIB} superstrings, {BPS} monopoles, and three-dimensional gauge
  dynamics.
\newblock {\em Nuclear Phys. B}, 492(1-2):152--190, 1997.

\bibitem[Ive72]{iversen1972fixed}
Birger Iversen.
\newblock A fixed point formula for action of tori on algebraic varieties.
\newblock {\em Invent. Math.}, 16(3):229--236, 1972.

\bibitem[Kin94]{king1994moduli}
Alastair King.
\newblock Moduli of representations of finite dimensional algebras.
\newblock {\em Q. J. Math., Oxf. II. Ser.}, 45(4):515--530, 1994.

\bibitem[Kir16]{kirillov2016quiver}
Alexander Kirillov.
\newblock {\em Quiver representations and quiver varieties}, volume 174 of {\em
  Graduate Studies in Mathematics}.
\newblock American Mathematical Society, 2016.

\bibitem[Kol09]{kollar2009lectures}
J{\'a}nos Koll{\'a}r.
\newblock {\em Lectures on resolution of singularities (AM-166)}, volume 166.
\newblock Princeton University Press, 2009.

\bibitem[Kon96]{konarski1996bb}
Jerzy Konarski.
\newblock The {B}-{B} decomposition via {Sumihiro's} theorem.
\newblock {\em J. Algebra}, 182(1):45--51, 1996.

\bibitem[KS20]{kononov2020pursuing}
Yakov Kononov and Andrey Smirnov.
\newblock Pursuing quantum difference equations {II}: 3{D}-mirror symmetry.
\newblock {\em arXiv preprint
  \href{https://arxiv.org/abs/2008.06309}{arXiv:2008.06309}}, 2020.

\bibitem[KS22]{kononov2022pursuing}
Yakov Kononov and Andrey Smirnov.
\newblock Pursuing quantum difference equations {I}: stable envelopes of
  subvarieties.
\newblock {\em Lett. Math. Phys.}, 112(4):69, 2022.

\bibitem[KT03]{knutson2003puzzles}
Allen Knutson and Terence Tao.
\newblock Puzzles and (equivariant) cohomology of {Grassmannians}.
\newblock {\em Duke Math. J.}, 119(2):221--260, 2003.

\bibitem[MFK94]{mumford1994geometric}
David Mumford, John Fogarty, and Frances Kirwan.
\newblock {\em Geometric invariant theory}, volume~34.
\newblock Springer, 1994.

\bibitem[Mil13]{milne1998lectures}
James Milne.
\newblock Lectures on {\'e}tale cohomology, 2013.
\newblock Available at \href{www.jmilne.org/math/}{www.jmilne.org/math/}.

\bibitem[MO19]{maulik2019quantum}
Davesh Maulik and Andrei Okounkov.
\newblock Quantum groups and quantum cohomology.
\newblock {\em Asterisque}, 408:1--225, 2019.

\bibitem[Muk03]{mukai2003introduction}
Shigeru Mukai.
\newblock {\em An introduction to invariants and moduli}, volume~81 of {\em
  Cambridge Studies in Advanced Mathematics}.
\newblock Cambridge University Press, 2003.

\bibitem[Nak18]{nakajima2021geometric}
Hiraku Nakajima.
\newblock Towards geometric {Satake} correspondence for {Kac}--{Moody}
  algebras--{Cherkis} bow varieties and affine {Lie} algebras of type $ a$.
\newblock {\em arXiv preprint
  \href{https://arxiv.org/abs/1810.04293}{arXiv:1810.04293}}, 2018.

\bibitem[New09]{newstead2009geometric}
Peter Newstead.
\newblock Geometric invariant theory.
\newblock In {\em Moduli spaces and vector bundles}, volume 359 of {\em London
  Mathematical Society Lecture Notes Series}, pages 99--127. Cambridge
  University Press, 2009.

\bibitem[NT17]{nakajima2017cherkis}
Hiraku Nakajima and Yuuya Takayama.
\newblock Cherkis bow varieties and {Coulomb} branches of quiver gauge theories
  of affine type {A}.
\newblock {\em Sel. Math., New Ser.}, 23(4):2553--2633, 2017.

\bibitem[Oko17]{okounkov2017k}
Andrei Okounkov.
\newblock K-theoretic computations in enumerative geometry.
\newblock In {\em Geometry of Moduli Spaces and Representation Theory},
  volume~24 of {\em IAS/Park City Mathematics Series}, pages 251–--380.
  American Mathematical Society, 2017.

\bibitem[Oko20]{okounkov2020nonabelian}
Andrei Okounkov.
\newblock Nonabelian stable envelopes, vertex functions with descendents, and
  integral solutions of $ q $-difference equations.
\newblock {\em arXiv preprint
  \href{https://arxiv.org/abs/2010.13217}{arXiv:2010.13217}}, 2020.

\bibitem[Oko21]{okounkov2021inductive}
Andrei Okounkov.
\newblock Inductive construction of stable envelopes.
\newblock {\em Lett. Math. Phys.}, 111:1--56, 2021.

\bibitem[OS22]{okounkov2022quantum}
Andrei Okounkov and Andrey Smirnov.
\newblock Quantum difference equation for {N}akajima varieties.
\newblock {\em Invent. Math.}, 229(3):1203--1299, 2022.

\bibitem[RS20]{rimanyi2020bow}
Rich{\'a}rd Rim{\'a}nyi and Yiyan Shou.
\newblock Bow varieties---geometry, combinatorics, characteristic classes.
\newblock {\em arXiv preprint
  \href{https://arxiv.org/abs/2012.07814}{arXiv:2012.07814}}, 2020.

\bibitem[RTV19]{rimanyi2019elliptic}
Richard Rim{\'a}nyi, Vitali Tarasov, and Alexander Varchenko.
\newblock Elliptic and {K}-theoretic stable envelopes and {Newton} polytopes.
\newblock {\em Sel. Math., New Ser.}, 25(1):16, 2019.

\bibitem[She13]{shenfeld2013abelianization}
Daniel Shenfeld.
\newblock {\em Abelianization of stable envelopes in symplectic resolutions}.
\newblock PhD thesis, Princeton University, 2013.

\bibitem[Sho21]{shou2021bow}
Yiyan Shou.
\newblock {\em Bow Varieties---Geometry, Combinatorics, Characteristic
  Classes}.
\newblock PhD thesis, University of North Carolina at Chapel Hill, 2021.

\bibitem[Sum74]{sumihiro1974equivariant}
Hideyasu Sumihiro.
\newblock Equivariant completion.
\newblock {\em J. Math. Kyoto Univ.}, 14(1):1--28, 1974.

\bibitem[Tak16]{takayama2016nahm}
Yuuya Takayama.
\newblock Nahm's equations, quiver varieties and parabolic sheaves.
\newblock {\em Publ. Res. Inst. Math. Sci.}, 52(1):1--41, 2016.

\bibitem[tD87]{tomdieck1987transformation}
Tammo tom Dieck.
\newblock {\em Transformation Groups}.
\newblock de Gruyter, 1987.

\bibitem[TV14]{tarasov2014hypergeometric}
Vitaly Tarasov and Alexander Varchenko.
\newblock Hypergeometric solutions of the quantum differential equation of the
  cotangent bundle of a partial flag variety.
\newblock {\em Open Math.}, 12(5):694--710, 2014.

\bibitem[TV19]{tarasov2019q}
Vitaly Tarasov and Alexander Varchenko.
\newblock q-{Hypergeometric} solutions of quantum differential equations,
  quantum {Pieri} rules, and {Gamma} theorem.
\newblock {\em J. Geom. Phys.}, 142:179--212, 2019.

\bibitem[TV22]{tarasov2022monodromy}
Vitaly Tarasov and Alexander Varchenko.
\newblock Monodromy of the equivariant quantum differential equation of the
  cotangent bundle of a {G}rassmannian.
\newblock {\em arXiv preprint
  \href{https://arxiv.org/abs/2212.09011}{arXiv:2212.09011}}, 2022.

\bibitem[Weh23]{wehrhan2023chevalley}
Till Wehrhan.
\newblock Chevalley--{Monk} formulas for bow varieties.
\newblock {\em arXiv preprint
  \href{https://arxiv.org/abs/2310.11235}{arXiv:2310.11235}}, 2023.

\bibitem[Whi65]{whitney1965tangents}
Hassler Whitney.
\newblock Tangents to an analytic variety.
\newblock {\em Ann. of Math. (2)}, 81(3):496--549, 1965.

\bibitem[W{\l}o05]{wlodarczyk2005simple}
Jaros{\l}aw W{\l}odarczyk.
\newblock Simple {Hironaka} resolution in characteristic zero.
\newblock {\em J. Am. Math. Soc.}, 18(4):779--822, 2005.

\end{thebibliography}

\end{document}